\documentclass{thielstyle-ams}

\newcommand{\hypalg}{\mathfrak{u}}

\DeclareMathOperator{\udim}{\underline{dim}}
\newcommand{\bdm}{\begin{displaymath}}
\newcommand{\edm}{\end{displaymath}}
\newcommand{\bthm}{\begin{thm}}
\newcommand{\ethm}{\end{thm}}
\newcommand{\blem}{\begin{lem}}
\newcommand{\elem}{\end{lem}}
\newcommand{\bcor}{\begin{cor}}
\newcommand{\ecor}{\end{cor}}
\newcommand{\beq}{\begin{equation}\label}
\newcommand{\eeq}{\end{equation}}
\newcommand{\bprop}{\begin{prop}}
\newcommand{\eprop}{\end{prop}}
\newcommand{\bdefn}{\begin{defn}}
\newcommand{\edefn}{\end{defn}}

\newcommand{\Z}{\mathbb{Z}}
\newcommand{\N}{\mathbb{N}}
\newcommand{\Q}{\mathbb{Q}}

\let\C\undefined
\newcommand{\C}{\mathbb{C}}
\newcommand{\h}{\mathfrak{h}}

\let\mc\mathcal
\let\mf\mathfrak

\let\Res\undefined
\DeclareMathOperator{\Res}{Res}

\renewcommand{\mod}{\textrm{mod}}

\renewcommand{\H}{\mathsf{H}}
\newcommand{\s}{\mathfrak{S}}

\DeclareMathOperator{\Supp}{\mathrm{Supp}}
\newcommand{\op}{\mathrm{op}}

\newcommand{\idot}{{\:\raisebox{2pt}{\text{\circle*{1.5}}}}}

\newcommand{\ds}{\dots}

\usepackage{calc}

\usepackage{enumitem}

\begin{document}  

\title[Highest weight theory for algebras with triangular decomposition]{Highest weight theory for finite-dimensional graded algebras with triangular decomposition}

\author{Gwyn Bellamy}

\address{School of Mathematics and Statistics, University Gardens, University of Glasgow, Glasgow, G12 8QW, UK}
\email{gwyn.bellamy@glasgow.ac.uk}

\author{Ulrich Thiel}
\address{School of mathematics and Statistics, University of Sydney, NSW 2006, Australia}
\email{ulrich.thiel@sydney.edu.au}

\begin{abstract}
We show that the category of graded modules over a finite-dimensional graded algebra admitting a triangular decomposition can be endowed with the structure of a highest weight category. When the algebra is self-injective, we show furthermore that this highest weight category has  tilting modules in the sense of Ringel. This provides a new perspective on the representation theory of such algebras, and leads to several new structures attached to them. There are a wide variety of examples in algebraic Lie theory to which this applies: restricted enveloping algebras, Lusztig’s small quantum groups, hyperalgebras, finite quantum groups, and restricted rational Cherednik algebras.

\end{abstract}

\blfootnote{February 27, 2018. Final version to appear in \textit{Adv. Math.}}

\maketitle
\tableofcontents

\section{Introduction}


The goal of this paper, and its sequel \cite{hwtpaper2}, is to develop new structures in the representation theory of a class of algebras commonly encountered in algebraic Lie theory: finite-dimensional $\mathbb{Z}$-graded algebras $A$ which admit a \textit{triangular decomposition}, i.e., a 
vector space decomposition 
\begin{equation}
A^- \otimes T \otimes A^+ \overset{\sim}{\longrightarrow} A
\end{equation}
into graded subalgebras given by the multiplication map, where we assume that $A^-$ is concentrated in negative degree, $T$ in degree zero, and $A^+$ in positive degree.

There are a variety of examples:
\begin{enumerate}\label{eq:VIPlist}
\item Restricted enveloping algebras $\overline{U}(\mf{g}_K)$;
\item Lusztig's small quantum groups $\mathbf{u}_{\zeta}(\mf{g})$, at a root of unity $\zeta$; \item Hyperalgebras $\hypalg_r(\mf{g}) \dopgleich \mathrm{Dist}(G_r)$ on the Frobenius kernel $G_r$; 
\item Finite quantum groups $\mc{D}$ associated to a finite group $G$;
\item Restricted rational Cherednik algebras (RRCAs) $\overline{\H}_{\bc}(W)$ at $t = 0$;
\item The center of smooth blocks of RRCAs at $t=0$;
\item RRCAs $\overline{\H}_{1,\bc}(W)$ at $t = 1$ in positive characteristic.
\end{enumerate}
There are many more examples, but the above examples are the ones we will address in more detail in Section \ref{example_section}. The representation theory of these algebras has important applications to other areas of mathematics. For instance, to symplectic algebraic geometry \cite{Baby,Singular,BellamyCounting,BellSchedThiel}, to algebraic combinatorics \cite{GriffMac,TomaszP}, and to algebraic groups in positive characteristic \cite{Jantzen,AJS}.  The applications mostly derive in one way or another from computing the graded character of irreducible modules. 

If we look at the list above, we can say that examples 1 to 3 share a ``common background'', as do examples 5 to 7, but taken in their totality, the algebras do not have much in common—except that they all admit a triangular decomposition. On the other hand, their representation theory behaves in a remarkably uniform way. This suggests that it is worthwhile developing a systematic approach to the representation theory of algebras with triangular decomposition.

This was begun by Holmes and Nakano \cite{HN}. They:
\begin{enumerate}[label=(\alph*)]
\item Defined four families of $A$-modules using the triangular decomposition: standard modules $\Delta(\lambda)$, costandard modules $\nabla(\lambda)$, proper standard modules $\Delta(\lambda)$, and proper costandard modules $\ol{\nabla}(\lambda)$, for each irreducible $T$-module $\lambda$. 
\item Showed that each $\Delta(\lambda)$ has a simple head $L(\lambda)$, and these modules are precisely the simple $A$-modules.
\item Showed that the projective cover $P(\lambda)$ of $L(\lambda)$ admits a filtration by standard modules $\Delta(\mu)$ and showed that the multiplicities $\lbrack P(\lambda): \Delta(\mu) \rbrack$ are independent of the filtration.
\item Showed Brauer reciprocity $\lbrack P(\lambda): \Delta(\mu) \rbrack = \lbrack \ol{\nabla}(\mu):L(\lambda) \rbrack$.
\end{enumerate}
All of the above hold both in the ungraded and in the graded category of $A$-modules. Our paper essentially starts from here. 

\subsection*{Highest weight structures} The prototypical example of a category with standard and costandard modules coming from a triangular decomposition is the Bernstein--Gelfand--Gelfand category $\mathcal{O}$ for a finite-dimensional complex semisimple Lie algebra. Of the many remarkable properties of this category, perhaps the most useful is the fact that it is a \textit{highest weight category}. This (categorical) concept was introduced by Cline–Parshall–Scott \cite{CPS} and has become an extremely influential idea in representation theory. It requires that there is a partial ordering on simple modules, such that the standard modules in a standard filtration of a projective $P(\lambda)$ have the property that $\Delta(\lambda)$ occurs precisely once, and that for all other $\Delta(\mu)$ occurring we have $\mu > \lambda$.  

For a finite dimensional algebra with triangular decomposition, we have standard modules and we know that projectives admit a standard filtration, so it is natural to ask whether the category of finite-dimensional $A$-modules is highest weight. In the examples listed above, it is easily checked that there \textit{cannot} exist any partial ordering on simple $T$-modules satisfying the above conditions. Even worse, each of these algebras is \textit{symmetric} (and not semi-simple). Hence, they have infinite global dimension. By results of Cline--Parshall-Scott \cite{CPS}, this implies that the category of finite-dimensional modules over $A$ cannot be a highest weight category.

We noticed that this is simply a case of asking the wrong question. Instead, one must consider the category $\mathcal{G}(A)$ of \textit{graded} finite-dimensional modules. This category has infinitely many simple objects, so the results from \cite{CPS} no longer apply. In fact, in this setting there is a very natural partial ordering on simple objects: it comes from the grading. More precisely, since $T$ is concentrated in degree zero, every graded simple $T$-module $\lambda$ is concentrated in a single degree $\deg \lambda$. This defines a function
\begin{equation} 
\deg \colon \Irr \mathcal{G}(T) \to \bbZ 
\end{equation}
on the set of irreducible objects in $\mathcal{G}(T)$, and this induces a partial ordering on $\Irr \mathcal{G}(A) \simeq \Irr \mathcal{G}(T)$. With respect to this ordering, we show that projective objects satisfy the requirements of a highest weight category, see Corollary \ref{semisimple_hwc}:

\begin{thm}
If $T$ is semi-simple, then $\mathcal{G}(A)$ is a highest weight category.
\end{thm}

In the body of the paper we actually do not assume that $T$ is semi-simple. We show more generally, without any additional assumptions to admitting a triangular decomposition, that $\mathcal{G}(A)$ is \textit{standardly stratified} in the sense of Cline--Parshall--Scott \cite{StratifiedEnd} and Losev--Webster \cite{UniqueLosevWebester}. Here, the distinction between standard and proper standard objects becomes important. We think it is an interesting problem to identify those highest weight categories which are equivalent, as highest weight categories, to $\mathcal{G}(A)$ for $A$ an algebra with triangular decomposition. In Section \ref{sec:multigrading} we describe how these results generalize naturally to algebras admitting a multi-grading. This is useful for studying examples such as hyperalgebras; see Section \ref{sec:hyperalg}. 

The theorem applies to \textit{all} examples mentioned at the beginning. The example of restricted rational Cherednik algebras at $t=0$ was actually our motivation for searching for a highest weight category structure. Namely, for rational Cherednik algebras at $t \neq 0$ the representation theory is extremely rich. In particular, category $\mathcal{O}$ as introduced by Ginzburg--Guay--Opdam--Rouquier \cite{GGOR} is a highest weight category with strong links to cyclotomic Hecke algebras, $q$-Schur algebras, Hilbert schemes etc. At $t=0$ this theory does not, a priori, exist and the limiting process $t \rightarrow 0$ is poorly understood. At $t = 0$, the focus shifts to restricted rational Cherednik algebras, which are finite-dimensional quotients of rational Cherednik algebras, still admitting a triangular decomposition. In light of the above theorem, we see that the category of graded modules for these algebras is again a highest weight category. We expect this highest weight structure will play an important role in a conceptual understanding of the limit $t \rightarrow 0$. Indeed, Bonnafé and Rouquier \cite{BonnafeRouquier-New} recently transferred the highest weight category structure we introduce here to the non-restricted setting, and were able to relate this to category $\mathcal{O}$ at $t = 1$.

The fact that $\mc{G}(A)$ is highest weight has several implications for the representation theory of $A$. We being to explore these implications here, continuing in a sequel \cite{hwtpaper2}. 

\subsection*{Tilting objects} An extremely important role in the theory of highest weight categories is played by \textit{tilting objects}, i.e., objects having both a standard and a costandard filtration. In the context of quasi-hereditary algebras, i.e., highest weight categories with \textit{finitely} many simple objects, Ringel \cite{Ringel-Filtrations} showed that for each $\lambda$ there is an indecomposable tilting object $T(\lambda)$, uniquely characterized by the property that it has highest weight $\lambda$, an injection $\Delta(\lambda) \hookrightarrow T(\lambda)$, and a surjection $T(\lambda) \twoheadrightarrow \nabla(\lambda)$. The $T(\lambda)$ are precisely the indecomposable tilting objects in the (Krull--Schmidt) category of tilting objects.

Once again, we cannot apply this result directly to our highest weight category $\mathcal{G}(A)$ since it does not have finitely many simple objects. None the less, we show that (see Corollary \ref{good_tilting_theory}):

\begin{thm}\label{thm:introtilting}
If $T$ is semisimple and $A$ is self-injective, then $\mathcal{G}(A)$ has tilting objects $T(\lambda)$ in the sense of Ringel.
\end{thm}

It is remarkable that self-injectivity, which does not allow any highest weight structure to exist on the category of finite-dimensional modules, is exactly what is needed to get the ``correct'' tilting theory for the graded category $\mathcal{G}(A)$. In Example \ref{example_no_tilting} we show that without self-injectivity, it can happen that $\mathcal{G}(A)$ does not possess \textit{any} tilting objects at all. 

Atypically, the tilting objects of $\mathcal{G}(A)$ have a very concise characterization: they are precisely the projective-injective objects, see Theorem \ref{tilings_projinj}. Hence, if $A$ is self-injective, we immediately deduce that the indecomposable tilting objects are the projective covers $P(\lambda)$. But it is by no means true that $T(\lambda)$ equals $P(\lambda)$! Rather, there is an extremely interesting permutation appearing on the set of simple graded modules. Namely,
\begin{equation}
T(\lambda^h) \simeq P(\lambda),
\end{equation}
where $h$ is the permutation on $\Irr \mathcal{G}(T)$ defined by $h \circ \nu = \dagger$ with $\dagger$ being given by $\Soc \Delta(\lambda) = L(\lambda^\dagger)$, and $\nu$ the Nakayama permutation coming from the self-injectivity of $A$. See Section \ref{more_on_tilting} for details. In the case of hyperalgebras, we explicitly determine the permutation $h$; see Lemma \ref{lem:permuhyper}.

Once again, Theorem \ref{thm:introtilting} applies to \textit{all} examples mentioned at the beginning, see Section \ref{example_section}. We note that in these examples the theorem, together with BGG reciprocity explained below, implies that the problem of computing the character of tilting modules in $\mc{G}(A)$ is equivalent to computing the multiplicities $[\Delta(\lambda) : L(\mu)]$. Though tilting theory is often studied in the context of the examples 1 to 3, this is usually as rational $G$-modules, resp. $U_{\zeta}$-modules, with both a good filtration and a Weyl filtration \cite{DonkinTilting}, resp. \cite{AndersenTilting}. The restriction of these tilting modules to the hyperalgebra $\hypalg_r(\mf{g})$, resp. to the small quantum group $\mathbf{u}_{\zeta}(\mf{g})$, are not in general tilting in our sense, see Proposition \ref{prop:tiltingrestilting} for a precise statement. With the exception of \cite{AndersenKaneda}, we are not aware of any work that systematically studies tilting modules defined directly, as above, for the algebra $A$.\\

\subsection*{BGG-reciprocity} If we once again compare properties (1)--(4) listed above for an algebra with triangular decomposition to BGG category $\mc{O}$, then one also notices that Brauer reciprocity (4) has the defect of involving both standard and costandard modules. We would prefer to have an actual \textit{BGG reciprocity} 
$$
\lbrack P(\lambda): \Delta(\mu) \rbrack = \lbrack \Delta(\mu):L(\lambda) \rbrack
$$
since this has several desired implications. For example, it implies that the blocks of the algebra can be obtained from knowing the constituents of standard modules. A general algebra with triangular decomposition will not satisfy BGG reciprocity, however. We relate BGG reciprocity to a symmetry between the negative and positive \textit{Borel subalgebras} of $A$, see Proposition \ref{prop:BGGbimodule}:

\begin{thm}
	Suppose that $T$ is semisimple. Let $B^\pm$ be the subalgebra of $A$ generated by $A^\pm$ and $T$. If $B^- \simeq (B^+)^\circledast$ as graded $T$-bimodules, then $A$ satisfies BGG reciprocity. 
\end{thm}

Here, $(-)^\circledast$ is the standard duality between finite-dimensional left and right modules induced by $\Hom_K(-,K)$, see Section \ref{notations}. There are two important aspects to this theorem. First, it provides an explicit condition for BGG reciprocity to hold, which can be easily checked in examples. Indeed, we show in Section \ref{example_section} that this holds for \textit{all} examples listed above, so they all satisfy BGG reciprocity. Secondly, we do not need any form of duality coming from an involution on the algebra (such dualities are discussed in Section \ref{duality_from_anti}). Indeed, restricted rational Cherednik algebras for a general complex reflection group are not know to admit any such duality—but the above theorem implies that they satisfy BGG reciprocity none the less.

The proof of the above theorem relies on a general property of the algebra $A$, namely that the \textit{graded character map} $\chi \from \K_0(\mathcal{G}(A)) \to \K_0(\mathcal{G}(T))$ between the graded Grothendieck groups of $A$ and $T$ induced by restriction is \textit{injective}; see Proposition \ref{prop:injectiveK}. This has several strong implications for the representation theory of $A$; see for instance Corollary \ref{dec_matrix_relation} and \cite{hwtpaper2}.

\subsection*{Abstract KL-theory} As a further application of the highest weight structure developed here, we consider briefly \textit{abstract Kazhdan--Lusztig theories} for $\mathcal{G}(A)$, in the sense of Cline--Parshall-Scott \cite{CPSKL}; see Section \ref{sec:kl_theory}. One says that a highest weight category admits an abstract Kazhdan--Lusztig theory if certain Ext-groups vanish. This definition is motivated by Lusztig's conjectures, and has several important consequences for the representation theory of the corresponding quasi-hereditary algebra. In particular, abstract Kazhdan--Lusztig polynomials can be defined, providing subtle invariants of the category. Naturally, we would like an abstract Kazhdan--Lusztig theory to exist for the highest weight category $\mathcal{G}(A)$. In general, however, it seems that there is no easy way to decide when $\mathcal{G}(A)$ admits such a theory. In Proposition \ref{prop:RRCAsingleblock} we give the answer for restricted rational Cherednik algebras (at generic $\bc$) for wreath products. It would be very interesting to know the answer (and in those cases where the answer is yes, the corresponding Kazhdan--Lusztig polynomials) in the other key examples. \\

\subsection*{Summary of \cite{hwtpaper2}} This paper also lays the foundation on which the sequel \cite{hwtpaper2} builds. In \textit{loc. cit.}, the tilting object $A$ of $\mathcal{G}(A)$, for a self-injective algebra $A$ with triangular decomposition, is the protagonist. In \textit{loc. cit.}, we prove three general results, which again all apply to the examples mentioned at the beginning:
\begin{enumerate}
\item The degree zero subalgebra $A_0$ of $A$ captures all important information about the graded representation theory of $A$.
\item $A_0$ is a standardly based algebra, in the sense of Du-Rui \cite{DuRui}. In the presence of a triangular anti-involution, see Section \ref{duality_from_anti}, it is actually cellular in the sense of Graham--Lehrer \cite{Graham-Lehrer-Cellular}. This implies the existence of cell modules and cells. It is an extremely interesting problem to determine these cells in the key examples.
\item We show that a certain subquotient of $\mathcal{G}(A)$ provides a \textit{highest weight cover} of $A_0$, in the sense of Rouquier \cite{RouquierQSchur}. This provides us with a quasi-hereditary algebra attached to $A$ which essentially contains all the information about the graded representation theory of $A$. Again, we believe it will be extremely interesting to determine this quasi-hereditary algebra in the examples.
\end{enumerate}

\subsection*{Outline}

In Section \ref{notations} we fix notation, in particular with regards to standard duality, and recall some facts about the graded representation theory of a finite-dimensional graded algebra. In Section \ref{sec:triangulardecomp} we introduce triangular decompositions and review some basic facts regarding their representation theory. This is mostly due to Holmes and Nakano \cite{HN}, but contains certain new results such as the injectivity of the graded character map. In Section \ref{sec:hwc} we show that $\mathcal{G}(A)$ is a highest weight category. Section \ref{more_on_tilting} is devoted to tilting theory. In Section \ref{duality_from_anti} we consider triangular anti-involutions and the induced duality on $\mathcal{G}(A)$. In Section \ref{sec:kl_theory} we discuss abstract Kazhdan--Lusztig theory for $\mathcal{G}(A)$. Up until this point the article is general, and essentially free of examples (with the exception of certain toy examples provided to illustrate pathologies). Section \ref{example_section} can be considered the second part of the article; here we address the examples mentioned at the beginning of the introduction. We show that they satisfy all properties needed to apply the results proved in the first part of the article.

\subsection*{Acknowledgements}

The first author was partially supported by EPSRC grant EP/N005058/1. The second author was partially supported by the DFG SPP 1489, by a Research Support Fund from the Edinburgh Mathematical Society, and by the Australian Research Council Discovery Projects grant no. DP160103897. We would especially like to thank S. Koenig for patiently answering several questions. We would also like to thank I. Losev and R. Rouquier for fruitful discussions. Moreover, we thank G. Malle, A. Henderson and O. Yacobi for comments on a preliminary version of this article. We would like to thank the referee of Adv. Math. for some helpful comments.

\tableofcontents

\section{Notation} \label{notations}

 Unless otherwise stated, all modules are \textit{left} modules and \textit{graded} always means $\bbZ$-graded. For a graded vector space $M$ we denote by $M_i$ the homogeneous component of degree $i$. We denote by $M \lbrack n \rbrack$ the \textit{right} shift of $M$ by $n \in \bbZ$, i.e., $M[n]_i = M_{i-n}$. So, if $M$ is concentrated in degree zero, then $M \lbrack n \rbrack$ is concentrated in degree $n$. The \word{graded dimension} of $M$ is defined as
 \begin{equation}
\udim M \dopgleich \sum_{i \in \bbZ} (\dim M_i) t^i \;.
\end{equation}
The \word{support} of~$M$ is defined to be $\Supp M \dopgleich \lbrace i \in \bbZ \mid M_i \neq 0 \rbrace$.
 
 \subsection{The category of graded modules} \label{graded_mod_cat_notations}
 Let $A$ be a finite-dimensional graded algebra over a field~$K$. We denote by $\mathcal{M}(A)$ the category of finitely generated $A$-modules and by $\mathcal{G}(A)$ the category of graded finitely generated $A$-modules with morphisms preserving the grading. We use the symbol $\mathcal{C}$ to denote either of the two categories, i.e., $\mathcal{C} \in \lbrace \mathcal{M},\mathcal{G} \rbrace$. The category $\mc{C}(A)$ is easily seen to be $K$-linear, essentially small, abelian, of finite length, Hom-finite (hence Krull–Schmidt by  \cite{Krause-KS}), and having enough projectives and injectives (see \cite{Gordon-Green-Rep-theory-of-graded-artin-82} for the graded case). It is clear that taking projective covers and injective hulls commutes with shifting. Due to $\mc{C}(A)$ being essentially small, the collection $\Irr \mc{C}(A)$ of isomorphism classes of simple objects of $\mc{C}(A)$ forms a set, and since $\mc{C}(A)$ is of finite length, the Grothendieck group $\K_0(\mathcal{C}(A))$ is the free abelian group with basis $\Irr \mc{C}(A)$. The shift functor $\lbrack 1 \rbrack$ makes $\K_0(\mathcal{G}(A))$ into a module over the Laurent polynomial ring $\bbZ \lbrack t,t^{-1} \rbrack$ such that multiplication by $t$ corresponds to the shift $\lbrack 1 \rbrack$. We denote by $\lbrack M:S \rbrack$ the multiplicity of a simple object $S$ in an object $M$ of $\mc{C}(A)$. Sometimes, we write $\lbrack M:S \rbrack_{\mc{C}(A)}$ to clarify in which category multiplicities are considered. In the graded setting we denote by
\begin{equation}
\lbrack M:S \rbrack^{\mrm{gr}} \dopgleich \sum_{t \in \bbZ} \lbrack M:S \lbrack i \rbrack \rbrack  t^i \in \bbZ \lbrack t,t^{-1} \rbrack
\end{equation}
the \word{graded multiplicity} of $S$ in $M$.
 
Forgetting the grading yields a $K$-linear functor $F:\mathcal{G}(A) \rarr \mathcal{M}(A)$ which is faithful and exact. The modules in the essential image of $F$ are called \word{gradable}.
  
 \begin{lem}[Gordon–Green \cite{GG-Graded-Artin, Gordon-Green-Rep-theory-of-graded-artin-82}] \label{forget_functor} \hfill 
\begin{enum_thm}
\item $F$ preserves and reflects: indecomposables, simples, projectives, injectives.
\item $F$ commutes with: radicals, socles, projective covers, injective hulls.
\item Simples, projectives, and injectives are gradable.
\item If $M \in \mc{G}(A)$ is indecomposable, then $F^{-1}(F(M))$ consists up to isomorphism only of the shifts of $M$.
\end{enum_thm}
\end{lem}

The forget functor thus induces a surjective map $\Irr \mathcal{G}(A) \twoheadrightarrow \Irr \mathcal{M}(A)$ with fibers just consisting of shifts of an arbitrary fixed object in the fiber. Hence, if $\Irr_0 \mathcal{G}(A)$ is the image of a section of the above map (so, a choice of grading on the ungraded simple modules), then $\Irr_0 \mathcal{G}(A)$ is in bijection with $\Irr \mathcal{M}(A)$, and the map $\Irr_0 \mathcal{G}(A) \times \bbZ \rarr \Irr \mathcal{G}(A)$, $(S,n) \mapsto  S[n]$, is a bijection. Moreover, $\K_0(\mathcal{G}(A))$ is a free $\bbZ \lbrack t,t^{-1} \rbrack$-module with basis $\Irr_0 \mathcal{G}(A)$. The representation of the class of $M \in \mc{G}(A)$ in $\K_0(\mc{G}(A))$ is thus 
\begin{equation}
\lbrack M \rbrack = \sum_{S \in \Irr_0 \mc{G}(A)} \lbrack M:S \rbrack^{\mrm{gr}} \lbrack S \rbrack 
\end{equation}
and the image of $[M]$ in $\K_0(\mc{M}(A))$ bis given by evaluation at $t=1$.

\subsection{Standard duality}\label{sec:standarddual}
We will adopt an important convention about standard duality which allows us to streamline the presentation later. First, if $M$ is a graded vector space, we write $M^\op$ for the same vector space, but with grading reversed, i.e., 
\begin{equation}
M^\op_i \dopgleich M_{-i} \;.
\end{equation}
We thus have $\Supp M^\op = - \Supp M$. With the reversed grading the opposite ring $A^\op$ of $A$ is again graded, see \cite[1.2.4]{Methods-of-graded-rings}. If $M$ is a (graded) \textit{right} $A$-module, then $M^\op$ is naturally a (graded) \textit{left} $A^\op$-module and vice versa. The assignment $M \mapsto M^\op$ with the identity on morphisms thus yields a natural identification between $\mathcal{C}(A^\op)$ and the category of finitely generated (graded) \textit{right} $A$-modules. For a $K$-vector space $M$ we denote by $M^* \dopgleich \Hom_K(M,K)$ its \word{dual}. If $M \in {\mathcal{C}}(A)$, then $M^*$ is naturally an object in $\mathcal{C}(A^{\op})$ with grading defined by 
\begin{equation}
(M^*)_i \dopgleich \lbrace f \in M^* \mid f(M_j) = 0 \tn{ for all } j \neq i \rbrace \simeq M_i^*
\end{equation}
and $A^{\op}$-action on $M^*$ defined by $(a^\op f)(m) \dopgleich f(am)$, for $f \in M^*$, $a^\op \in A^\op$, and $m \in M$. With $\Hom_K(-,K)$ on morphisms, this defines a contravariant functor
\begin{equation}
(-)^*:{ \mathcal{C}}(A) \rarr \mathcal{C}(A^{\op}) \;,
\end{equation}
called \word{standard duality}. Since $A$ is finite-dimensional, this functor is indeed a duality, i.e., $(-)^* \circ (-)^* \simeq \id_{\mathcal{C}(A)}$. It induces a bijection $\Irr \mathcal{C}(A) \simeq \Irr \mathcal{C}(A^{\op})$. We have $M^*\lbrack n \rbrack = M \lbrack n \rbrack^*$ and $\Supp M^* = \Supp M$. Occasionally, we still wish to consider $M^*$ as a \textit{right} $A$-module instead of a \textit{left} $A^\op$-module. To this end, we write 
\begin{equation}
(-)^\circledast \dopgleich (-)^\op \circ (-)^* \;.
\end{equation}
Standard duality commutes with the forget functor $F$, see \cite[Proposition 2.5]{GG-Graded-Artin}. If $M \in \mathcal{C}(A)$ and $U \leq M$ is a subobject, we write $U^\perp \dopgleich \lbrace f \in M^* \mid f(U) = 0 \rbrace \leq M^*$. This is a subobject of $M^*$ and the operation $\perp$ yields an inclusion reversing bijection between the subobjects of $M$ and $M^*$.  We have $\Rad(M)^\perp = \Soc(M^*)$ and $\Soc(M)^\perp = \Rad(M^*)$. Here $\Rad (M)$ is the radical of $M$ and $\Soc (M)$ is its socle. This implies that $\Hd(M)^* \simeq \Soc(M^*)$, where $\Hd(M)$ is the head of $M$.  

\section{Triangular decompositions}\label{sec:triangulardecomp}

In this section we review the notion of triangular decompositions of finite-dimensional graded algebras as introduced by Holmes and Nakano \cite{HN}. This includes the construction of standard modules and the corresponding parametrization of simple modules. We consider  some additional properties of algebras with triangular decomposition, like ambidexterity, the socle of costandard modules, splitting, and semisimplicity. In Section \ref{sec_dec_matrices} we prove the injectivity of the graded character map mentioned in the introduction.\\

\begin{tcolorbox}
Throughout, $A$ is a finite-dimensional graded algebra over a field $K$.
\end{tcolorbox}

\begin{defn}[Holmes–Nakano]\label{defn:triangle}
A \word{triangular decomposition} of $A$ is a triple $\mathcal{T} = (A^-,T,A^+)$ of graded subalgebras of $A$ such that:
\begin{enum_thm}
\item the multiplication map $A^- \otimes_K T \otimes_K A^+ \rarr A$ is an isomorphism of vector spaces, 
\item $T$ is concentrated in degree zero, and $\Supp A^+ \subset \Z_{\ge 0}$, $\Supp A^- \subset \Z_{\le 0}$,
\item \label{defn:connected} $A_0^- = K = A_0^+$,
\item $A^+ T = T A^+$ and $A^- T = T A^-$ as subspaces of $A$,
\item $T$ is a split $K$-algebra, i.e., $\End_T(\lambda) = K$ for all $\lambda \in \Irr \mathcal{M}(T)$.
\end{enum_thm}
\end{defn}

\begin{ex} \label{first_examples_of_triang_dec}
As a simple example consider the polynomial ring $K \lbrack x,y \rbrack$ with $\deg(x)=-1$ and $\deg(y)=1$. Then $K \lbrack x,y \rbrack/(x^2,y^2) = K \lbrack x \rbrack/(x^2) \otimes_K K \otimes_K K \lbrack y \rbrack/(y^2)$ is a triangular decomposition. On the other hand, the three-dimensional algebra $K \lbrack x,y \rbrack/(x^2,yx,y^2)$ does \textit{not} have a triangular decomposition: the only non-trivial graded subalgebras are $K \lbrack x \rbrack/(x^2)$ and $K \lbrack y \rbrack/(y^2)$ but these do not yield a triangular decomposition for dimension reasons. The reader may look at Section \ref{sec:toy} for several further simple examples.
\end{ex}

Holmes and Nakano just considered algebraically closed fields $K$ but it is useful for applications to instead just assume that $T$ is a split $K$-algebra. We show in Proposition \ref{splitting} that this implies that $A$ is a split $K$-algebra.

Our convention to put the left part of the decomposition into negative degree is adapted to the highest weight theory we are going to develop later. One may of course also put the left part into positive degree and all results are still valid with the appropriate modifications. However, as the following example shows, interchanging the left and right parts of a decomposition does not necessarily give a decomposition such that the multiplication map is an isomorphism.

\begin{ex} \label{non_ambidext}
Let $A = K \langle x,y \rangle / \langle x^2, yx, y^2 \rangle$ with $\deg (x) = -1$ and $\deg(y) = 1$. Here, $K \langle x,y \rangle$ is the free non-commutative algebra on two generators. Then $A^- = K[x]/(x^2)$, $A^+ = K[y] / (y^2)$, and $T = K$ define a triangular decomposition $(A^-,T,A^+)$ of $A$. However, the multiplication map $A^+ \otimes_K T \otimes_K A^- \rarr A$ is not injective since $y \otimes 1 \otimes x$ is sent to zero.
\end{ex}

We thus make the following definition.

\begin{defn} \label{defn_ambidextrous}
A triangular decomposition $\mc{T} = (A^-,T,A^+)$ is \textit{ambidextrous} if $\mc{T}^\dagger \dopgleich (A^+,T,A^-)$ is also a multiplicative decomposition of $A$.
\end{defn}

We note that all of the examples from the introduction are ambidextrous and this property will play an important role again in \cite{hwtpaper2}. Regardless, a triangular decomposition of $A$ always gives, after interchanging the left and right parts, a triangular decomposition of the \textit{opposite} algebra: 

\begin{lem}\label{lem:opptriangle}
If $\mc{T}=(A^-,T,A^+)$ is a triangular decomposition of $A$, then the triple
\begin{equation}
\mathcal{T}^\op \dopgleich ((A^+)^{\op}, T^{\op}, (A^-)^{\op})
\end{equation}
is a triangular decomposition of $A^{\op}$. \qed
\end{lem}

This simple observation will be used frequently in the sequel. We call $\mathcal{T}^\op$ the \word{opposite} of $\mathcal{T}$ and, without further mention, $A^\op$ will always be equipped with this triangular decomposition. \\

\begin{tcolorbox}
We fix a triangular decomposition $\mc{T} = (A^-,T,A^+)$ of $A$. 
\end{tcolorbox}

\subsection{$T$-modules} \label{section_t_modules} 
The assumption $T = T_0$ implies that the graded module category of $T$ has a rather simple structure:  every simple object $\lambda$ of $\mathcal{G}(T)$ is concentrated in a single degree $\deg \lambda \in \bbZ$. This yields a function
\begin{equation} \label{degree_function}
\deg: \Irr \mathcal{G}(T) \rarr \bbZ \;.
\end{equation}
For $d \in \bbZ$, let $\mathcal{G}_d(T)$ be the full subcategory of $\mathcal{G}(T)$ consisting of modules concentrated in degree $d$. We have a canonical equivalence $\mathcal{G}_d(T) \simeq \mathcal{M}(T)$ of abelian categories. Moreover, for each $d \in \bbZ$ the homogeneous component $M_{d}$ of $M \in \mathcal{G}(T)$ is an object of $\mathcal{G}_d(T)$ and the decomposition of $M$ into its homogeneous components yields a canonical isomorphism 
\begin{equation} \label{graded_t_mod_cat_decomp}
\mathcal{G}(T) \simeq \bigoplus_{d \in \bbZ} \mathcal{G}_d(T) \simeq \bigoplus_{d \in \bbZ} \mathcal{M}(T) \;.
\end{equation}
In particular, we have bijections 
\begin{equation}
\Irr \mathcal{G}(T) \leftrightarrow \Irr \mathcal{G}_0(T) \times \bbZ \leftrightarrow \Irr \mathcal{M}(T) \times \bbZ \;.
\end{equation}
If $M \in \mathcal{G}_d(T)$, then $M^*$ is also concentrated in degree $d$, so $M^* \in \mathcal{G}_d(T^\op)$. The degree function is thus invariant under duality. 

\subsection{Borel subalgebras}

Let $B^{\pm} \dopgleich A^{\pm} T = TA^{\pm}$. This is a graded subalgebra of $A$, with the same support as $A^{\pm}$. We call it the (\textit{negative}, respectively \textit{positive}) \word{Borel subalgebra} of $A$. The action of $T$ by multiplication makes $B^\pm$ a graded $(T,T)$-bimodule. In the same fashion, it is a graded $(A^{\pm},T)$-bimodule and a graded $(T,A^{\pm})$-bimodule. 

\begin{lem} \label{borel_tensor_dec}
The multiplication maps
\begin{equation} \label{borel_tensor_dec_iso1}
A^{\pm} \otimes_K T \overset{\sim}{\longrightarrow} A^{\pm} T = B^{\pm}
\end{equation}
and
\begin{equation} \label{borel_tensor_dec_iso2}
T \otimes_K A^{\pm} \overset{\sim}{\longrightarrow} T A^{\pm} = B^{\pm}
\end{equation}
are isomorphisms of graded $(A^{\pm},T)$-bimodules and graded $(T,A^{\pm})$-bimodules, respectively. In particular, $B^{\pm}$ is both a free left $A^{\pm}$-module, and a free right $A^{\pm}$-module, and each homogeneous component of $B^{\pm}$ (and thus $B^{\pm}$ itself) is both a free left $T$-module and a free right $T$-module. Moreover, $A$ is a free as a: left $A^-$-module, left $B^-$-module, right $A^+$-module, and right $B^+$-module.
\end{lem}

\begin{proof}
We just consider the ``negative case'', the other case is similar. Both maps are clearly bimodule morphisms and surjective by definition of $B^-$. Moreover, the first map is injective by Definition \ref{defn:triangle}, and thus an isomorphism. Because the $K$-vector space dimensions of the domain and codomain of the second map are equal, it also has to be injective, thus an isomorphism. The multiplication map $B^- \otimes_K A^+ \rarr A$ is now an isomorphism of $(B^-,A^+)$-bimodules, hence $A$ is a free left $B^-$-module and a free right $A^+$-module.  
\end{proof}

The composition of the two maps (\ref{borel_tensor_dec_iso1}) and (\ref{borel_tensor_dec_iso2}) yields a $K$-vector space isomorphism
\begin{equation}\label{eq:iso3}
R^{\pm}: T \otimes_K A^{\pm} \overset{\sim}{\longrightarrow} A^{\pm} \otimes_K T \;.
\end{equation}
For $a' \in A^{\pm}$ and $t \in T$ we can write $R^{\pm}(t \otimes a') = \sum_i a_i' \otimes t_i$ for certain $a_i' \in A^{\pm}$ and $t_i \in T$. Using the isomorphisms in Lemma \ref{borel_tensor_dec} we see that for the multiplication in $B^{\pm}$ we have
\[
(a \otimes t)(a' \otimes t') = (at)(a't') = a(ta')t' = a R^\pm(t \otimes a')t' = \sum_i aa_i' \otimes t_it' \;,
\]
in other words we have
\begin{equation}
m_{B^{\pm}} = (m_{A^{\pm}} \otimes m_T) \circ (\id_{A^{\pm}} \otimes R^{\pm} \otimes \id_T) \;,
\end{equation}
where $m_-$ denotes the ring multiplication map of the respective rings. In the notation of \cite[\S 2]{Factorization-Problem-Smash-Product}, this means that
\begin{equation}
B^{\pm} = A^{\pm} \#_{R^{\pm}} T \;,
\end{equation}
i.e., the algebra $B^{\pm}$ is the \word{smash product} of $A^{\pm}$ and $T$ with respect to the \word{braiding} $R^{\pm}$. The degree zero component of $B^{\pm}$ is equal to $T$. Since $B^{\pm}$ is finite-dimensional, its augmentation ideal 
\begin{equation}
J^{\pm} \dopgleich \sum_{\pm i >0} B^{\pm}_i
\end{equation}
is nilpotent. We thus have a natural surjective graded algebra and $(T,T)$-bimodule morphism 
\begin{equation}
B^{\pm} \twoheadrightarrow T
\end{equation}
 with nilpotent kernel $J^{\pm}$. Clearly, as a $(T,T)$-bimodule morphism, this map has a section and therefore we have
 \begin{equation} \label{borel_decomposition}
B^{\pm} = T \oplus J^{\pm}
\end{equation}
as $(T,T)$-bimodules.

\begin{rem} \label{opposite_borel}
Note that for $\mathcal{T}^\op$ the negative Borel is $(B^+)^\op$ and the positive Borel is $(B^-)^\op$.
\end{rem}

\subsection{Proper standard and costandard modules}
We denote by 
\begin{equation}
\Inf_T^{B^{+}}:\mathcal{C}(T) \rarr \mathcal{C}(B^{+})
\end{equation}
the scalar restriction functor induced by the quotient morphism $B^{+} \twoheadrightarrow T$, i.e., we let $B^{+}$ act on a $T$-module via this morphism. Note that this functor induces a bijection between isomorphism classes of simple objects, since the kernel of the morphism is the nilpotent ideal $J^+$, and thus contained in the Jacobson radical of $B^+$. We can now define the functor
\begin{equation} \label{verma_def}
\ol{\Delta} (-) \dopgleich A \otimes_{B^+} \Inf_{T}^{B^+} - : {\mathcal{C}(T)} \rarr {\mathcal{C}}(A).
\end{equation}
We call $\ol{\Delta}(M)$ the \word{proper standard module} associated to $M$. 

\begin{lem} \label{basic_delta_lemma}
For $M \in \mathcal{C}(T)$ we have:
\begin{enum_thm}
\item $\ol{\Delta}(M) \simeq B^- \otimes_T M $ in $\mathcal{C}(B^-)$.
\item $\ol{\Delta}(M) \simeq A^- \otimes_K M$ in $\mathcal{C}(A^-)$.
\item $\ol{\Delta}(M) \simeq (J^- \otimes_T M) \oplus M$ in $\mathcal{C}(T)$.
%
%
\end{enum_thm}
\end{lem}

\begin{proof}
As $A^+$ acts trivially on $\Inf_T^{B^+} M$, the canonical morphism $B^- \otimes_T M \rarr \ol{\Delta}(M)$ is surjective. The $K$-dimensions on both sides are equal; hence this is an isomorphism in $\mathcal{C}(B^-)$. Since $B^- \simeq A^- \otimes_K T$ in $\mathcal{C}(T)$ by Lemma \ref{borel_tensor_dec} and $B^- = J^- \oplus T$ in $\mathcal{C}(T)$ by (\ref{borel_decomposition}), this implies all other assertions.
\end{proof}

\begin{lem} \label{basic_delta_lemma_2}
For $M \in \mathcal{G}_d(T)$ we have:
\begin{enum_thm}
\item $\ol{\Delta}(M)_i = B_{i-d}^- \otimes_T M$ in $\mathcal{C}(T)$. In particular, $\ol{\Delta}(M)_d \simeq M$.
\item \label{basic_delta_lemma_2:supp_contained} $\Supp \ol{\Delta}(M) = d+ \Supp A^- \subset d - \bbN$.
\item $\ol{\Delta}(M)_d \simeq M$ in $\mathcal{C}(B^+)$.
\end{enum_thm}
\end{lem}

\begin{proof}
All assertions follow directly from Lemma \ref{basic_delta_lemma}. For the last assertion we note that, a priori, we have $\ol{\Delta}(M)_d \simeq M$ in $\mathcal{G}(T)$. But $A^+_{i\neq 0}$ acts as zero on $\ol{\Delta}(M)_d$ by \ref{basic_delta_lemma_2:supp_contained}, and therefore this isomorphism is in fact an isomorphism of $B^+$-modules.
\end{proof}

Using duality we can now define an additional functor
\begin{equation}
\ol{\nabla} - \dopgleich \ol{\Delta} (-^*) ^* : \mathcal{C}(T) \rarr \mathcal{C}(A) \;.
\end{equation}
Here, the inner dual applies to $T$-modules and the outer for $A$-modules. We call $\ol{\nabla}(M)$ the \word{proper costandard module} associated to $M$.

Recall that we have the triangular decomposition $\mathcal{T}^\op$ of $A^\op$ whose positive Borel is $(B^-)^\op$. For this triangular decomposition we thus also have a proper standard functor, given by 
\begin{equation}
\ol{\Delta}:A^\op \otimes_{(B^-)^\op} \Inf_{T^\op}^{(B^-)^\op}: \mathcal{C}(T^\op) \rarr \mathcal{C}(A^\op) \;,
\end{equation}
and therefore, for $M \in \mathcal{G}(T)$, we have
\begin{equation}
\ol{\nabla}(M^*) = (A^\op \otimes_{(B^-)^\op} \Inf_{T^\op}^{(B^-)^\op} M^*)^* \;.
\end{equation}
We can avoid additional notation like $\ol{\Delta}^\op$ for this functor since the argument (a $T$-module or a $T^\op$-module) makes it clear which one is meant.  Note again that we also have a proper costandard functor $\ol{\nabla}:\mathcal{C}(T^\op) \rarr \mathcal{C}(A^\op)$ associated to $\mathcal{T}^\op$. Applying Lemma \ref{basic_delta_lemma_2} to $\mathcal{T}^\op$ and dualizing we get: 

\begin{lem} \label{basic_nabla_lemma_2}
Let $M \in \mathcal{G}_d(T)$. Then:
\begin{enum_thm}
\item $\ol{\nabla}(M)_i = \left( (B^+)^\op_{i-d} \otimes_{T^\op} M^* \right)^* = (M^\circledast \otimes_T B_{-i+d}^+)^\circledast $ in $\mathcal{C}(T)$. In particular, $\ol{\nabla}(M)_d \simeq M$.
\item $\Supp \ol{\nabla}(M) = d + \Supp (A^+)^\op = d - \Supp A^+  \subset d - \bbN$.
\item $\ol{\nabla}(M)_d \simeq M_d$ in $\mathcal{C}(B^-)$. 
\end{enum_thm}
\end{lem}
%

 It follows directly from the definitions that proper standard and proper costandard modules are compatible with the degree shift in the sense that
\begin{equation} \label{standard_compat_with_shift}
\ol{\Delta}(M\lbrack n \rbrack) = \ol{\Delta}(M) \lbrack n \rbrack \quad \tn{and} \quad \ol{\nabla}(M \lbrack n \rbrack) = \ol{\nabla}(M) \lbrack n \rbrack
\end{equation}
for any $M \in \mathcal{G}(T)$ and $n \in \bbZ$. Moreover, they are also compatible with forgetting the grading.

\begin{lem} \label{delta_exact}
The functors $\ol{\Delta}$ and $\ol{\nabla}$ are exact.
\end{lem}

\begin{proof}
Scalar restriction, and thus inflation, is an exact functor. Moreover, since $A$ is a free right $B^+$-module by Lemma \ref{borel_tensor_dec}, the functor $A \otimes_{B^+} -$ is exact. Hence, $\ol{\Delta}$ is a composition of exact functors, thus exact. Using the same result for $\mathcal{T}^\op$ and using the fact that dualizing is exact, it follows that $\ol{\nabla}$ is also exact.
\end{proof}

\subsection{Simple modules} \label{sect_simple_modules}

The classification of simple $A$-modules given in Theorem \ref{heads} below is due to Holmes and Nakano \cite{HN}. We will need an additional result regarding the socle of the costandard module, and to prove it, it is easiest to repeat some of the arguments given in \cite{HN}. The following lemma is elementary.

\begin{lem} \label{socle_head_lemma}
Let $S \in \Irr {\mathcal{C}}(A)$ and $\lambda \in \Irr \mathcal{C}(T)$. If $\Inf_{T}^{B^+}  \lambda$ is a constituent of $\Soc(\Res_{B^+}^A S)$, then $S$ is a constituent of $\Hd(\ol{\Delta}( \lambda))$.
\end{lem}

\begin{proof}
The module $\Inf_{T}^{B^+} \lambda \in {\mathcal{C}}(B^+)$ is simple. As $\Soc(\Res_{B^+}^A S)$ is semisimple, its constituent $\Inf_{T}^{B^+} \lambda$ is already a direct summand. Hence, there is a non-zero morphism $\varphi: \Inf_{T}^{B^+} \lambda \hookrightarrow \Soc(\Res_{B^+}^A S) \subs S$ in $\mathcal{C}(B^+)$. The morphism $\varphi$ corresponds, under adjunction, to a morphism $\wt{\varphi}: \ol{\Delta} (\lambda )= A \otimes_{B^+} \Inf_{T}^{B^+} \lambda \rarr S$ in $\mathcal{C}(A)$. Explicitly, $\wt{\varphi}(a \otimes v) \dopgleich a\varphi(v)$ for $a \in A$ and $v \in \lambda$. This morphism is again non-zero and, since $S$ is simple, it is surjective so that $\Ker(\wt{\varphi})$ is a maximal subobject of $\ol{\Delta}( \lambda)$ in $\mathcal{C}(A)$. This shows that $S$ is a constituent of $\Hd(\ol{\Delta}( \lambda))$. 
\end{proof}

%


\begin{thm}[Holmes--Nakano] \label{heads}
The following holds:
\begin{enum_thm}
\item For any $\lambda \in \Irr \mathcal{C}(T)$ the head of $\ol{\Delta}(\lambda)$ is isomorphic to the socle of $\ol{\nabla}( \lambda )$, and this is a simple object in $\mathcal{C}(A)$, denoted $L( \lambda )$.
\item  The map $L:\Irr \mathcal{C}(T) \rarr \Irr \mathcal{C}(A)$, $\lambda \mapsto L (\lambda)$, is a bijection. 
\item \label{heads:forget} $L$ is compatible with forgetting the grading.
\item For any $\lambda \in \Irr \mathcal{C}(T)$, both $\ol{\Delta}(\lambda)$ and $\ol{\nabla}(\lambda)$ are indecomposable. 
\end{enum_thm}
\end{thm}

\begin{proof}
We set $E \dopgleich \Res_{B^-}^A \ol{\Delta}(\lambda)$ and for $\mu \in \Irr \mathcal{C}(T)$ we set $\hat{\mu} \dopgleich \Inf_T^{B^-} \mu$. Recall from Lemma \ref{basic_delta_lemma} that $E \simeq B^- \otimes_T \lambda$. Hence, if $\mu \in \Irr \mathcal{C}(T)$, then by adjunction we have
\[
\Hom_{\mathcal{C}(B^-)}(E, \hat{\mu}) \simeq \Hom_{\mathcal{C}(B^-)}(B^- \otimes_T \lambda, \hat{\mu}) \simeq \Hom_{\mathcal{C}(T)}(\lambda, \mu) \;.
\]
The $K$-dimension of this homomorphism space is always zero unless $\lambda \simeq \mu$, in which case it is equal to one, since $T$ is split over $K$. This already implies that 
\begin{equation}\label{heads_res_equ}
\Hd(\Res_{B^-}^A \ol{\Delta}(\lambda)) \simeq \hat{\lambda} \;.
\end{equation}
In fact, if $M$ is a maximal subobject of $E$ in $\mathcal{C}(B^-)$, then we get a non-trivial morphism $E \rarr E/M$. As $E/M$ is simple, we must have $E/M \simeq \hat{\lambda}$ by the above. Hence, the quotient of $E$ by any maximal subobject is isomorphic to $\hat{\lambda}$. This implies that $\Hd(E) \simeq \hat{\lambda}^{\oplus n}$. Since $\Hom_{\mathcal{C}(B^-)}(E,\hat{\lambda}) = \Hom_{\mathcal{C}(B^-)}(\hat{\lambda}^{\oplus n},\hat{\lambda}) = K^n$, we conclude that $n = 1$.

To show that $\ol{\Delta}(\lambda)$ itself has a simple head, let $\Sigma$ be the sum of all proper submodules of $\ol{\Delta}(\lambda)$. Then $\Res_{B^-}^A \Sigma \leq E$ is a sum of proper $B^-$-submodules. This must be a proper submodule since $E$ has a simple head. Hence, $\Sigma$ is a proper submodule and $\ol{\Delta}(\lambda)$ has a simple head. Using the same result for $\mathcal{T}^\op$, it follows that $\ol{\Delta}(\lambda^*)$ has a simple head; thus $\ol{\nabla}(\lambda) = \ol{\Delta}(\lambda^*)^*$ has simple socle. 

We want to show that $\Hd(\ol{\Delta}(\lambda)) \simeq \Soc(\ol{\nabla}(\lambda))$. From (\ref{heads_res_equ}), applied to $\mathcal{T}^\op$, we know that 
\[
\Inf_{T^\op}^{(B^+)^\op}\lambda^* \simeq \Hd(\Res_{(B^+)^\op}^{A^\op} \ol{\Delta}(\lambda^*)) \;.
\]
Since $\Hd(\Res_{(B^+)^\op}^{A^\op} \ol{\Delta}(\lambda^*))$ is a quotient of $\Res_{(B^+)^\op}^{A^\op} \Hd(\ol{\Delta}(\lambda^*))$, it follows that $\Inf_{T^\op}^{(B^+)^\op}\lambda^*$ is a constituent of $\Res_{(B^+)^\op}^{A^\op} \Hd(\ol{\Delta}(\lambda^*))$. Applying duality, this shows that $\Inf_T^{B^+} \lambda$ is a constituent of $\Res_{B^+}^A \Soc (\ol{\nabla}(\lambda))$. Now, we can use Lemma \ref{socle_head_lemma} to deduce that $\Soc(\ol{\nabla}(\lambda))$ is a constituent of $\Hd(\ol{\Delta}(\lambda))$ and so they are isomorphic. 

By Lemma \ref{socle_head_lemma} every simple $A$-module is equal to $L(\lambda)$ for some $\lambda \in \Irr \mathcal{C}(T)$. It just remains to show that $L(\lambda)$ is not isomorphic to $L(\mu)$ whenever $\lambda$ and $\mu$ are not isomorphic. So, assume that $L(\lambda) \simeq L(\mu)$. We know that $\Hd(\Res_{B^-}^A \ol{\Delta}(\lambda)) \simeq \Inf_T^{B^-} \lambda$. It follows that the head of $\Res_{B^-}^A \Hd(\ol{\Delta}(\lambda))$ is also isomorphic to $\Inf_T^{B^-} \lambda$. Similarly, the head of $\Res_{B^-}^A \Hd(\ol{\Delta}(\mu))$ is isomorphic to $\Inf_T^{B^-} \mu$. Since $L(\lambda) \simeq L(\mu)$, we have  $\Res_{B^-}^A \Hd(\ol{\Delta}(\lambda)) \simeq \Res_{B^-}^A \Hd(\ol{\Delta}(\mu))$, hence these modules have isomorphic heads which means that $\Inf_T^{B^-} \lambda \simeq \Inf_T^{B^-} \mu$ and therefore $\lambda \simeq \mu$.

The indecomposability of $\ol{\Delta}(\lambda)$ and $\ol{\nabla}(\lambda)$ is obvious since they have simple head, respectively socle.
\end{proof}

By (\ref{standard_compat_with_shift}), we have $\ol{\Delta}(\lambda \lbrack d \rbrack) = \ol{\Delta}(\lambda) \lbrack d \rbrack$ for any $d \in \bbZ$, hence 
\begin{equation}
L(\lambda \lbrack d \rbrack) \simeq L(\lambda) \lbrack d \rbrack \;.
\end{equation}
Setting 
\begin{equation}
\Irr_d \mathcal{G}(A) \dopgleich \lbrace L(\lambda) \mid \lambda \in \Irr \mathcal{G}_d(T) \rbrace \;,
\end{equation}
where $\mathcal{G}_d(T)$ is as introduced in Section \ref{section_t_modules}, it follows that 
\begin{equation}
\Irr_d \mathcal{G}(A) = \lbrace L(\lambda)[d] \mid \lambda \in \Irr \mathcal{G}_0(T) \rbrace \simeq \Irr \mathcal{M}(T) \simeq \Irr \mc{M}(A)
\end{equation}
and that
\begin{equation}
\Irr \mathcal{G}(A) \leftrightarrow \Irr \mathcal{G}_0(A) \times \bbZ \leftrightarrow \Irr \mathcal{M}(T) \times \bbZ \;.
\end{equation}
We can now extend the degree function defined in (\ref{degree_function}) to a function  
\begin{equation} \label{degree_function_on_A}
\deg \from \Irr \mathcal{G}(A) \to \bbZ
\end{equation}
via $\deg L(\lambda) \dopgleich \deg \lambda$. We note that $L(\lambda)$ is not necessarily concentrated in a single degree. 

\begin{lem} \label{L_dual}
For any $\lambda \in \Irr \mathcal{C}(T)$ we have $L(\lambda^*) \simeq L(\lambda)^*$.
\end{lem}

\begin{proof}
This follows from  $L(\lambda)^* \simeq (\Hd\ol{\Delta}(\lambda))^* \simeq \Soc \ol{\Delta}(\lambda)^* \simeq \Soc \ol{\nabla}(\lambda^*) \simeq L(\lambda^*)$.
\end{proof}





\subsection{Splitting and semisimplicity}

The following proposition was shown by Bonnafé and Rouquier \cite[Proposition 9.2.5]{BonnafeRouquier} for restricted rational Cherednik algebras. The argument also works, word for word, in our general setting. We repeat the proof here for convenience.

\begin{prop}[Bonnafé--Rouquier] \label{splitting}
If $\lambda \in \Irr \mathcal{C}(T)$, then $\End_{\mathcal{C}(A)}(L(\lambda)) = K$. In particular, each $L(\lambda)$ is absolutely simple and $A$ is a split $K$-algebra.
\end{prop}

\begin{proof}
Since $\End_{\mathcal{G}(A)}(L(\lambda)) \subs \End_{\mathcal{M}(A)}(L(\lambda))$, it is enough to consider the non-graded case. Recall that $\ol{\Delta}:{\mathcal{M}(T)} \rarr {\mathcal{M}(A)}$ is a functor and therefore it induces a $K$-algebra morphism $\End_T(\lambda) \rarr \End_A(\ol{\Delta}(\lambda))$. An endomorphism of $\ol{\Delta}(\lambda)$ maps the radical to the radical (see \cite[Proposition 5.1]{CurtisReiner}), thus induces an endomorphism of $L(\lambda)$. Hence, we get a $K$-algebra morphism $m_\lambda:\End_T(\lambda) \rarr \End_A(L(\lambda))$. By assumption, $T$ splits and therefore we have $\End_T(\lambda) = K$. Hence, if we can show that $m_\lambda$ is surjective, the claim follows. Recall from Lemma \ref{degree_zero_component} that $\ol{\Delta}(\lambda)_0 \simeq L(\lambda)_0 \simeq \lambda$ in $\mathcal{M}(T)$, where we identify $\lambda \in \mathcal{G}_0(T)$. In particular, $\lambda$ is a direct summand of $L(\lambda)$ in $\mathcal{M}(T)$. Let $\pi_\lambda:L(\lambda) \twoheadrightarrow \lambda$ and $\iota_\lambda:\lambda \hookrightarrow L(\lambda)$ be the projection and inclusion of $L(\lambda)_d \simeq \lambda$, respectively. These are morphisms in $\mathcal{M}(T)$ and we get a map $\Psi_\lambda: \End_A(L(\lambda)) \rarr \End_T(\lambda)$ mapping $\varphi$ to $\pi_\lambda \circ \varphi \circ \iota_\lambda$. Note that $\Psi_\lambda$ is just the restriction of a morphism onto the degree zero component. Now, let $\varphi \in \End_A(L(\lambda))$ and set $\tilde{\varphi} \dopgleich \varphi - m_\lambda(\varphi)$. It is easy to see that $\Psi_\lambda(m_\lambda(\varphi)) = \varphi$, so $\Psi_\lambda(\tilde{\varphi}) = 0$. Since $\iota_\lambda$ is injective, this implies that $\Ker(\tilde{\varphi}) \neq 0$. As $\End_A(L(\lambda))$ is a division ring by Schur's lemma, we must have $\tilde{\varphi} = 0$, so $\varphi = m_\lambda(\tilde{\varphi})$. Hence, $m_\lambda$ is surjective.
\end{proof}

\begin{cor} \label{semisimplicity}
The following are equivalent:
\begin{enum_thm}
\item $A$ is semisimple.
\item \label{semisimplicity:criterion} $T$ is semisimple and both $\ol{\Delta}(\lambda)$ and $\ol{\nabla}(\lambda)$ are simple for all $\lambda \in \Irr \mathcal{M}(T)$.
\end{enum_thm}
\end{cor}

\begin{proof}
Since $T$ splits by assumption, we have
\begin{equation}
    \dim_K T = \dim_K \Rad(T) + \sum_{\lambda \in \Irr \mathcal{M}(T)} (\dim_K \lambda)^2 \;.
\end{equation}
Moreover, we know that $A$ splits by Proposition \ref{splitting} and therefore
\begin{equation}
    \dim_KA = \dim_K \Rad(A) + \sum_{\lambda \in \Irr \mathcal{M}(T)} (\dim_K L(\lambda))^2 \;.
\end{equation}
The claim now follows using $\dim_K \ol{\Delta}(\lambda) = \dim_K \lambda \cdot \dim_K A^-$ and $\dim_K \ol{\nabla}(\lambda) = \dim_K \lambda \cdot \dim_K A^+$ for $\lambda \in \Irr \mathcal{C}(T)$, and $\dim_K A = \dim_K A^- \cdot \dim_K T \cdot \dim_K A^+$.

\end{proof}

\subsection{The top component} We can give a definite result about the structure of the \word{top component} of a simple $A$-module, i.e., for the homogeneous component of maximal degree. This will be used frequently in the paper.

\begin{lem} \label{degree_zero_component}
Let $\lambda \in \Irr \mathcal{G}_d(T)$. Then $\ol{\Delta}(\lambda)$ is generated in $\mathcal{G}(A)$ by any non-zero element of degree $d$ and $\lambda \simeq L(\lambda)_d \simeq \ol{\Delta}(\lambda)_d$ in $\mathcal{G}(T)$. 
\end{lem}

\begin{proof}
In Lemma \ref{basic_delta_lemma_2} we have already seen that $\ol{\Delta}(\lambda)_d \simeq \lambda$ as $T$-modules. From Lemma \ref{basic_delta_lemma} we know that $\ol{\Delta}(\lambda)$ is isomorphic to $B^- \otimes_T \lambda$ as a left $T$-module.  Let $0 \neq v \in \lambda$. Since $\lambda$ is simple, we have $Tv = \lambda$. If now $b \otimes w \in \ol{\Delta}(\lambda)$ with $b \in B^-$ and $w \in \lambda$, then there is $t \in T$ with $tv = w$ and therefore $bt(1 \otimes v) = bt \otimes v = b \otimes tv = b \otimes w$. Hence, $\ol{\Delta}(\lambda)$ is generated by $1 \otimes v$, which is an element of degree $d$. This implies that $\ol{\Delta}(\lambda)_d \cap \Rad \ol{\Delta}(\lambda) = 0$. Therefore the quotient morphism $\ol{\Delta}(\lambda) \twoheadrightarrow L(\lambda)$ induces an isomorphism $\ol{\Delta}(\lambda)_d \simeq L(\lambda)_d$ in $\mathcal{G}(T)$.
\end{proof}

Together with Lemma \ref{basic_delta_lemma_2}\ref{basic_delta_lemma_2:supp_contained} it thus follows that 
\begin{equation}
\udim L(\lambda) = (\dim \lambda) t^d  + \tn{ terms of lower degree}
\end{equation}
for $\lambda \in \Irr \mc{G}_d(T)$. In particular, $\deg L(\lambda)$ is the largest integer in $\Supp L(\lambda)$. 

\begin{cor} \label{nabla_d_component}
Let $\lambda \in \Irr \mathcal{G}_d(T)$. Then every non-zero proper submodule of $\ol{\nabla}(\lambda)$ contains $\ol{\nabla}(\lambda)_d$. In particular, $\Soc \ol{\nabla}(\lambda) \simeq L(\lambda)$ is generated by $\ol{\nabla}(\lambda)_d$.
\end{cor}

\begin{proof}
Let $0 \neq V < \ol{\nabla}(\lambda)$ and let $W \dopgleich \ol{\nabla}(\lambda)/V$. The quotient morphism $\ol{\nabla}(\lambda) \twoheadrightarrow W$ is a surjective but not injective morphism and so its dual $W^* \hookrightarrow \ol{\nabla}(\lambda)^*$ is an injective but not surjective morphism. The image $W'$ of this morphism is thus a proper submodule of $\ol{\nabla}(\lambda)^*$ and therefore contained in $\Rad \ol{\nabla}(\lambda)^*$. Note that
 \begin{equation}
 \ol{\nabla}(\lambda)^* = \left(\ol{\Delta}(\lambda^*)^*\right)^* \simeq \ol{\Delta}(\lambda^*) \;.
 \end{equation}
 Lemma \ref{basic_delta_lemma_2} applied to $\mathcal{T}^\op$ shows that $\Supp \ol{\Delta}(\lambda^*) = d + (\Supp A^+)^\op = d - \Supp A^+ \subset d - \bbN$. We know from Lemma \ref{degree_zero_component} applied to $\mathcal{T}^\op$ that $\Rad \ol{\Delta}(\lambda^*) \cap \ol{\Delta}(\lambda^*)_{d} = 0$ and so we see that $W'$ must be contained in $\ol{\Delta}(\lambda^*)_{d-\bbN_{>0}}$. Upon dualizing, this shows that the quotient morphism $\ol{\nabla}(\lambda) \twoheadrightarrow W$ is still surjective on $\ol{\nabla}(\lambda)_{d-\bbN_{>0}}$. But this means that
 \[
 \left( \ol{\nabla}(\lambda)_{d-\bbN_{>0}} + V \right)/V = \ol{\nabla}(\lambda)/V
 \]
 and therefore $\ol{\nabla}(\lambda)_{d-\bbN_{>0}} + V = \ol{\nabla}(\lambda)$. Since $V$ is graded and $d \in \Supp \ol{\nabla}(\lambda)$, this is only possible if $\ol{\nabla}(\lambda)_d \subs V$. This proves the first claim. Now, the statement applies in particular to $\Soc \ol{\nabla}(\lambda)$ so that $\ol{\nabla}(\lambda)_d \subs \Soc\ol{\nabla}(\lambda)$. As $\Soc \ol{\nabla}(\lambda)$ is the unique minimal submodule of $\ol{\nabla}(\lambda)$, we conclude that $\ol{\nabla}(\lambda)_d$ generates $\Soc \ol{\nabla}(\lambda)$.
\end{proof}

%

%
%
%
%
%

\subsection{Grothendieck groups} \label{sec_dec_matrices}
Recall that the shift operation endows the graded Gro\-then\-dieck groups $\K_0(\mc{G}(T))$ and $\K_0(\mc{G}(A))$ with a $\bbZ \lbrack t,t^{-1} \rbrack$-module structure such that $t$ acts by the shift $\lbrack 1 \rbrack$. We can view any graded $A$-module as a graded $T$-module, and thus obtain a $\bbZ \lbrack t,t^{-1}\rbrack$-module morphism 
\begin{equation}
\chi \from \K_0(\mathcal{G}(A)) \to \K_0(\mathcal{G}(T)) \;.
\end{equation}
The Grothendieck groups $\K_0(\mc{G}(T))$ and $\K_0(\mc{G}(A))$ are free $\bbZ \lbrack t,t^{-1} \rbrack$-modules with basis $\Irr \mc{G}_0(T)$ and $\Irr_0 \mc{G}(A)$, respectively. Let $\mathbf{C}_L$ be the matrix of the morphism $\chi$ in these bases, i.e., 
\begin{equation}
\mathbf{C}_L \dopgleich  ( \lbrack L(\lambda) : \mu \rbrack^{\mrm{gr}} )_{\lambda,\mu \in \Irr \mc{G}_0(T)} \in \Mat_{\Irr \mc{G}_0(T)}(\bbZ\lbrack t,t^{-1} \rbrack) \;.
\end{equation}
and let $\mathbf{D}_{\ol{\Delta}}$ be the decomposition matrix of the proper standard modules as graded $A$-modules, i.e.,
\begin{equation}
\mathbf{D}_{\ol{\Delta}} \dopgleich ( \lbrack \ol{\Delta}(\lambda) : L(\mu) \rbrack^{\mrm{gr}} )_{\lambda,\mu \in \Irr \mc{G}_0(T)} \in \Mat_{\Irr \mc{G}_0(T)}(\bbZ\lbrack t,t^{-1} \rbrack) \;.
\end{equation}
Evaluating at $t=1$ yields the ungraded decomposition matrices which we denote by $F\mathbf{C}_L$ and $F\mathbf{D}_{\ol{\Delta}}$, respectively. We will show that determining $\mathbf{C}_L$ is essentially equivalent to determining $\mathbf{D}_{\ol{\Delta}}$. Let
\begin{equation} \label{matrix_A}
\mathbf{C}_{\ol{\Delta}} \dopgleich (\lbrack \ol{\Delta}(\lambda):\mu \rbrack^{\mrm{gr}})_{\lambda,\mu \in \Irr \mc{G}_0(T)} \in \Mat_{\Irr \mc{G}_0(T)}(\bbZ\lbrack t,t^{-1} \rbrack)
\end{equation}
be the graded decomposition matrix of the proper standard modules as $T$-modules. By definition, it is clear that
\begin{equation}
\mathbf{C}_{\ol{\Delta}} = \mathbf{D}_{\ol{\Delta}}  \mathbf{C}_L  \;.
\end{equation}
Evaluating at $t=1$ yields the analogous relation 
\begin{equation}
F\mathbf{C}_{\ol{\Delta}} = (F\mathbf{D}_{\ol{\Delta}}) (F \mathbf{C}_L)
\end{equation}
for the ungraded decomposition matrices.

\begin{prop}\label{prop:injectiveK}
The morphism $\chi \from \K_0(\mc{G}(A)) \to \K_0(\mc{G}(T))$ is injective.
\end{prop}

\begin{proof}
Let $u \dopgleich t^{-1} $. Let $U$ be the $\bbZ \lbrack u \rbrack$-submodule of $\K_0(\mathcal{G}(A))$ generated by $\Irr_0 \mathcal{G}(A)$ and let $V$ be the $\bbZ \lbrack u \rbrack$-submodule of $\K_0(\mathcal{G}(T))$ generated by $\Irr_0 \mathcal{G}(T)$. If we can show that the restriction of $\chi$ to $U$ is injective, then, as $\chi$ is the localization of $\chi|_U$ in the multiplicative set $\lbrace u^n \mid n \in \bbN \rbrace$, $\chi$ is also injective by exactness of localization. An arbitrary non-zero element $x$ of $U$ is of the form $\sum_{\lambda \in \Lambda} f_\lambda \lbrack L(\lambda) \rbrack$ for some subset $\Lambda \subseteq \Irr_0 \mathcal{G}(T)$ and some non-zero $f_\lambda \in \bbZ \lbrack u \rbrack$. We then have $\chi(x) = \sum_{\lambda \in \Lambda} f_\lambda \chi([L(\lambda)])$. Now, recall from Lemma \ref{basic_delta_lemma_2}\ref{basic_delta_lemma_2:supp_contained} and Lemma \ref{degree_zero_component} that $\chi([L(\lambda)]) = \lbrack \lambda \rbrack + x_\lambda$, where $x_\lambda \in u V$. Hence, if $\chi(x) = 0$, then
\begin{equation} \label{graded_scalar_restriction_injective_equ}
\sum_{\lambda \in \Lambda} f_\lambda \lbrack \lambda \rbrack + f_\lambda x_\lambda = 0 \;.
\end{equation}
Let $b_\lambda$ be the trailing degree of $f_\lambda$, i.e., the minimum of the exponents of the indeterminate $u$ among the non-zero monomials in $f_\lambda$. Since $0 \neq f_\lambda \in \bbZ \lbrack u \rbrack$ by assumption, we have $b_\lambda \in \bbN$. Let $b$ be the minimum of all the $b_\lambda$. Then $f_\lambda x_\lambda \in u^{b+1} V$. Note that $V$ is a free $\bbZ \lbrack u \rbrack$-module with basis $\Irr_0 \mathcal{G}(T)$ and so the quotient $V/u^{b+1} V$ is a free $\bbZ \lbrack u \rbrack/(u^{b+1})$-module with basis indexed by the same set $\Irr_0 \mathcal{G}(T)$. Let $\ol{f}_\lambda$ be the image of $f_\lambda$ in $\bbZ \lbrack u \rbrack/(u^{b+1})$ and let $\ol{\lbrack \lambda \rbrack}$ be the image of $\lbrack \lambda \rbrack \in V$ in $V/u^{b+1} V$. Then equation (\ref{graded_scalar_restriction_injective_equ}) implies that $\sum_{\lambda \in \Lambda} \ol{f}_\lambda \ol{\lbrack \lambda \rbrack} = 0$ and hence $\ol{f}_\lambda = 0$ for all $\lambda$. But there is some $\mu \in \Lambda$ with $b_\mu = b$, implying that $\ol{f}_\mu \neq 0$; this is a contradiction. Hence, $\chi$ is injective on $U$.
\end{proof}

\begin{cor} \label{dec_matrix_relation}
The matrix $\mathbf{C}_L$ is invertible over $\bbQ (t)$, so 
\begin{equation} \label{dec_matrix_relation:equ}
\mathbf{D}_{\ol{\Delta}} =  \mathbf{C}_{\ol{\Delta}} \mathbf{C}_L^{-1} \;.
\end{equation}
\end{cor}

\begin{proof}
The matrix $\mathbf{C}_L$ is the matrix of the $\bbZ \lbrack t,t^{-1} \rbrack$-module morphism $\chi$ between two free modules of the same rank. By Proposition \ref{prop:injectiveK} it is injective, hence, after extending to $\bbQ(t)$, it is an isomorphism, so $\mathbf{C}_L$ is invertible. 
\end{proof}

\begin{ex}

We note that the proper standard modules $\ol{\Delta}(\lambda)$, $\lambda \in \Irr \mathcal{C}(T)$, do not necessarily form a basis of $\mrm{K}_0(\mathcal{C}(A))$. As an example consider the $K$-algebra $A = K\lbrack x,y \rbrack/(x^2,y^2)$ with triangular decomposition $K\lbrack x \rbrack/(x^2) \otimes_K K \otimes_K K \lbrack y \rbrack/(y^2)$ and $\mrm{deg}(x) =-1$, $\mrm{deg}(y)=1$. There is only one simple $K$-module, namely the trivial one, which we denote by $1$ and which we consider as graded in degree zero. Let $\ol{\Delta} \dopgleich \ol{\Delta}(1)$ and let $L=L(1)$. It is not hard to see that $\lbrack \ol{\Delta} \rbrack = \lbrack L \rbrack + t \lbrack L \rbrack$ in $\mrm{K}_0(\mathcal{G}(A))$. It is thus clear that $\lbrack \ol{\Delta} \rbrack$ cannot be a basis of $\mrm{K}_0(\mathcal{C}(A))$. 
\end{ex}

To determine $\mathbf{C}_{\ol{\Delta}}$, recall from Lemma \ref{basic_delta_lemma} that $\ol{\Delta}(\lambda) = (J^- \otimes_T \lambda) \oplus \lambda$ in $\mc{C}(T)$, so this boils down to determining the graded $T$-module structure of $J^-$ and understanding the decomposition of tensor products of simple $T$-modules. 

%



\subsection{Rigid modules}

We want to mention a special case where we have, for specific $\lambda$, a complete understanding of $L(\lambda)$. The \word{rigid quotient} of $A$ is the quotient algebra $\check{A} \dopgleich A / I$, where $I$ is the two-sided ideal of $A$ generated by $A^-_{<0}$ and $A^+_{>0}$. Since $A$ splits by Proposition \ref{splitting}, $\check{A}$ is also split. Note that by the triangular decomposition of $A$, we have a surjection $T \twoheadrightarrow \check{A}$ and $\check{A} = T/(T \cap I)$. The simple $\check{A}$-modules are precisely the simple $A$-modules $L(\lambda)$ with trivial action of $A^-_{<0}$ and $A^+_{>0}$. In this case, we say that $\lambda \in \Irr \mc{M}(T)$ is \word{rigid}.

\begin{ex}
In our standard example $A = K\lbrack x,y\rbrack/(x^2,y^2)$ it is clear that $I$ is generated by $x$ and $y$, so $\check{A} = K$, and therefore the unique simple $T$-module is rigid.
\end{ex}

\begin{lem}
$\lambda \in \Irr \mc{C}(T)$ is rigid if and only if $L(\lambda) \simeq \lambda$ in $\mc{C}(T)$.
\end{lem}

\begin{proof}
Without loss of generality we can work in the graded setting. If $L(\lambda) \simeq \lambda$, then $L(\lambda)$ is concentrated in a single degree, so $A^-_{<0}$ and $A^+_{>0}$ have to act trivially, i.e., $\lambda$ is rigid. Conversely, assume that $\lambda \in \Irr \mc{G}_d(T)$ is rigid, so $A^-_{<0}$ and $A^+_{>0}$ act trivially on $L(\lambda)$. From Lemma \ref{degree_zero_component} we know that $L(\lambda)_d \simeq \lambda$ as $T$-modules. Since $A^-_{<0}$ and $A^+_{>0}$ act trivially on $L(\lambda)$ and $T$ is concentrated in degree zero, every homogeneous component $L(\lambda)_i$ is an $A$-submodule of $L(\lambda)$. Hence, $L(\lambda)$ can have only one non-zero component, so $L(\lambda) = L(\lambda)_d \simeq \lambda$. 
\end{proof}

Rigid modules for restricted rational Cherednik algebras played an important role in an earlier paper \cite{BT-Cuspidal} by the authors. In this paper we classified the rigid modules for restricted rational Cherednik algebras for all but a few exceptional Coxeter groups. It is an open problem to determine the rigid modules for the other examples mentioned in the introduction.

\section{Highest weight theory} \label{sec:hwc}

In this section we show, without imposing any further assumptions on $A$, that the graded module category $\mc{G}(A)$ is a \word{standardly stratified category} in the sense of Losev–Webster \cite[\S2]{UniqueLosevWebester}. The layers of the stratification are the categories $\mathcal{G}_d(T) \simeq \mathcal{M}(T)$ defined in \S\ref{section_t_modules}. This implies that the graded representation theory of $A$ has a rich combinatorial structure. In the case where $T$ is semisimple, which is true in all the examples mentioned in the introduction, this standardly stratified structure is a highest weight structure, in the sense of Cline–Parshall–Scott \cite{CPS}. 

\subsection{Standardly stratified categories} \label{sec_standardly_strat}
Let us first recall the notion of a standardly stratified category. Our definition is based on the one given in \cite[\S2]{UniqueLosevWebester}, but we weaken some of the assumptions. The usual results one derives from a standardly stratified category still hold with these weaker assumptions. 




\begin{defn}[Losev–Webster]
Let $K$ be a field and let $\mathcal{C}$ be a $K$-linear finite length abelian category, with enough projectives, such that each simple object is absolutely simple. Let $\Lambda$ be a set indexing the isomorphism classes of simple objects in $\mathcal{C}$, with $L(\lambda)$ being the simple object corresponding to $\lambda \in \Lambda$. The projective cover of $L(\lambda)$ is denoted $P(\lambda)$. Let $\Xi$ be an interval finite  poset and let $d:\Lambda \rarr \Xi$ be a map with finite fibers. Then $\Lambda$ is equipped with a partial order $\leq$, defined by 
\begin{equation} \label{ss_partial_order}
\lambda < \mu \tn{ if and only if } d(\lambda) < d(\mu) \;.
\end{equation}
For $\xi \in \Xi$, let $\mathcal{C}_{\leq \xi}$, resp. $\mathcal{C}_{<\xi}$, be the Serre subcategory spanned by the $L(\lambda)$ with $d(\lambda) \leq \xi$, resp. $d(\lambda) < \xi$. Let $\mathcal{C}_\xi \dopgleich \mathcal{C}_{\leq \xi}/\mathcal{C}_{<\xi}$ be the quotient category, called a \word{layer} of $\mathcal{C}$, and let $\pi_\xi:\mathcal{C}_{\leq \xi} \rarr \mathcal{C}_\xi$ be the quotient functor. For $\lambda \in d^{-1}(\xi)$, let $L_\xi(\lambda)$ be the simple object of $\mathcal{C}_\xi$ corresponding to $\lambda$ and let $P_\xi(\lambda)$ be the projective cover of $L_\xi(\lambda)$ in $\mathcal{C}_\xi$. Suppose now that, for each $\xi \in \Xi$, the quotient functor $\pi_\xi$ admits an exact left adjoint functor $\Delta_\xi$, called the \word{standardization functor}. Then, for each $\lambda \in \Lambda$, we set  
\begin{equation}
\Delta(\lambda) \dopgleich \Delta_{d(\lambda)}(P_{d(\lambda)}(\lambda))
\end{equation}
and 
\begin{equation}
\ol{\Delta}(\lambda) \dopgleich \Delta_{d(\lambda)}(L_{d(\lambda)}(\lambda)) \;.
\end{equation}
These objects are called the \word{standard}, resp. \word{proper standard}, objects in $\mathcal{C}$.  The category $\mathcal{C}$, together with the additional data described above, is said to be \word{standardly stratified} if for each $\lambda \in \Lambda$ there is an epimorphism $P(\lambda) \twoheadrightarrow \Delta(\lambda)$ whose kernel admits a filtration by standard objects $\Delta(\mu)$ with $\mu > \lambda$.
\end{defn}

\begin{rem}
Assume that $\Delta(\lambda) = \ol{\Delta}(\lambda)$ for all $\lambda \in \Lambda$. In this case, each $L_\xi(\lambda)$ is projective and hence all layers are semisimple. Then the standardly stratified structure, as defined above, is actually a \word{highest weight structure}, as defined by Cline–Parshall–Scott \cite{CPS}. The simple objects are labeled by the poset $(\Lambda, \leq)$ as in (\ref{ss_partial_order}), and standard objects are $\Delta(\lambda)$.
\end{rem}



\subsection{Standard and costandard objects} \label{standard_and_costandard}

We now describe the standard and costandard objects in the category $\mc{G}(A)$. For $\lambda \in \Irr \mathcal{C}(T)$ we denote by $P(\lambda) \dopgleich P(L(\lambda))$ and $I(\lambda) \dopgleich I(L(\lambda))$ the projective cover, resp. the injective hull, of $L(\lambda)$ in $\mathcal{C}(A)$. Recall from \S\ref{notations} that $F(P(\lambda)) \simeq P(F(\lambda))$ and $F(I(\lambda)) \simeq I(F(\lambda))$ for all $\lambda \in \Irr \mathcal{G}(T)$, where $F$ denotes the respective functor forgetting the grading. Also note that $P(\lambda[n]) \simeq P(\lambda)[n]$ and that $I(\lambda[n]) \simeq I(\lambda)[n]$. Furthermore, note that if $\lambda \in \Irr \mathcal{C}(T)$, then since $P(\lambda^*)$ is the projective cover of $L(\lambda^*)$ in $\mathcal{C}(A^{\op})$, the dual $P(\lambda^*)^*$ is the injective hull of $L(\lambda^*)^* \simeq L(\lambda)$ in $\mathcal{C}(A)$, i.e., 
\begin{equation} \label{proj_cov_inj_hull_dual}
P(\lambda^*)^* = I(\lambda) \;.
\end{equation}
For the projective cover, resp. the injective hull, of $\lambda$ in $\mathcal{C}(T)$ we specifically write $P_T(\lambda)$, resp. $I_T(\lambda)$. They behave under dualizing just as in equation (\ref{proj_cov_inj_hull_dual}). For $\lambda \in \Irr \mathcal{C}(T)$ we define the associated \word{standard object} and \word{costandard object} as
\begin{equation}
\Delta(\lambda) \dopgleich \ol{\Delta}(P_T(\lambda)) \quad \tn{and} \quad \nabla(\lambda) \dopgleich \ol{\nabla}(I_T(\lambda)) \;,
\end{equation}
respectively. By definition we have
\begin{equation} \label{standard_dual}
\Delta(\lambda^*)^* = \ol{\Delta}(P_T(\lambda^*))^* \simeq \ol{\Delta}(I_T(\lambda)^*)^* = \ol{\nabla}(I_T(\lambda)) = \nabla(\lambda) \;.
\end{equation}
Since $\ol{\Delta}$ is exact, the epimorphism $P_T(\lambda) \twoheadrightarrow \lambda$ induces an epimorphism $\Delta(\lambda) \twoheadrightarrow \ol{\Delta}(\lambda)$. Dualizing shows that we have an embedding $\ol{\nabla}(\lambda) \hookrightarrow \nabla(\lambda)$. Clearly, if $T$ is semisimple, then $\Delta(\lambda) = \ol{\Delta}(\lambda)$ and $\nabla(\lambda) = \ol{\nabla}(\lambda)$. We will restrict to this situation soon, but first we study the general setting.\\

 
We say that $M \in \mathcal{C}(A)$ is \word{standardly filtered} if there is a filtration 
\begin{equation} \label{standardly_filtered_equation}
M = M^0 \supset M^1 \supset \ldots \supset M^{s-1} \supset M^s = 0
\end{equation}
in $\mathcal{C}(A)$ such that for all $0 \leq i < s$ we have $M^i/M^{i+1} \simeq \Delta(\lambda_i)$ for some $\lambda_i \in \Irr \mathcal{C}(T)$. We write $\mathcal{C}^\Delta(A)$ for the full subcategory of $\mathcal{C}(A)$ consisting of standardly filtered objects. The following lemma summarizes several facts proven in \cite[\S4]{HN}. It is important to note that these statements hold because $A \in \mathcal{C}(B^-)$ is free by Lemma \ref{borel_tensor_dec}.

\begin{lem}[Holmes--Nakano] \label{holmes-nakano-facts}
 The following holds:
\begin{enum_thm}
\item Let $\lambda,\mu \in \Irr {\mathcal{C}(T)}$. Then $\Res_{B^-}^A \Delta(\lambda)$ is the projective cover of $\Inf_{T}^{B^-} \lambda$ in $\mathcal{C}(B^-)$.
\item \label{holmes-nakano-facts:ext} For any $\lambda,\mu \in \Irr {\mathcal{C}(T)}$ we have
\[
\Ext^n_{\mathcal{C}(A)}(\Delta(\lambda),\ol{\nabla}(\mu)) = \left\lbrace \begin{array}{ll} K & \tn{if } \lambda = \mu \tn{ and } n=0, \\ 0 & \tn{else.} \end{array} \right.
\]
\item \label{holmes-nakano-facts:standardly_filtered_res_proj} If $M \in {\mathcal{C}^{\Delta}}(A)$, then $\Res_{B^-}^A M$ is projective in $\mathcal{C}(B^-)$.
\item \label{holmes-nakano-facts:proj_standardly_filtered} If $M \in {\mathcal{G}}(A)$ such that $\Res_{B^-}^A M$ is projective in $\mathcal{G}(B^-)$, then $M \in {\mathcal{G}^{\Delta}}(A)$. 
\end{enum_thm}
\end{lem}

\begin{cor} \label{projectives_standard_filt}
All projective objects in $\mathcal{C}(A)$ admit a standard filtration.
\end{cor}

\begin{proof}
In the graded case, the first assertion follows directly from Lemma \ref{holmes-nakano-facts}\ref{holmes-nakano-facts:proj_standardly_filtered}. In the ungraded case we can use  \cite[Corollary 3.4]{GG-Graded-Artin} which shows that every projective object in $\mathcal{M}(A)$ is gradable, i.e., for any projective $P \in \mathcal{M}(A)$ there is a projective  object $\tilde{P} \in \mathcal{G}(A)$ with $F(\tilde{P}) = P$, where $F$ is the  functor forgetting the grading. Since $\tilde{P}$ is standardly filtered, so too is $P$. 
\end{proof}

From the Ext-vanishing property in Lemma \ref{holmes-nakano-facts}\ref{holmes-nakano-facts:ext} one deduces easily by induction that  
\begin{equation}
\lbrack M:\Delta(\lambda) \rbrack \dopgleich \# \lbrace i \mid M_i/M_{i-1} \simeq \Delta(\lambda) \rbrace = \dim_K \Hom_{\mathcal{C}(A)} (M, \ol{\nabla}(\lambda))
\end{equation}
for $\lambda \in \Irr \mathcal{C}(T)$. Hence, this number is independent of the chosen filtration. In \cite[Theorem 4.5]{HN} it is proven that \word{Brauer reciprocity} holds in $\mathcal{C}(A)$, i.e.:

\begin{prop}[Holmes–Nakano] \label{brauer_reciprocity}
The relation
\begin{equation}\label{eq:BGG}
\left\lbrack P(\lambda) : \Delta(\mu) \right\rbrack = \left\lbrack \ol{\nabla}(\mu) : L(\lambda) \right\rbrack
\end{equation}
holds for any $\lambda \in \Irr \mathcal{C}(T)$. 
 \end{prop} 
 
In a similar fashion we say that $M \in \mathcal{C}(A)$ is \word{costandardly filtered} if there is a filtration 
\begin{equation}
0 = M_0 \subset M_1 \subset \ldots \subset M_{s-1} \subset M_s = M
\end{equation}
in $\mathcal{C}(A)$ such that for all $0 < i \leq s$ we have $M_i/M_{i-1} \simeq \nabla(\lambda_{i-1})$ for some $\lambda_i \in \Irr \mathcal{C}(T)$. We write $\mathcal{C}^\nabla(A)$ for the full subcategory of $\mathcal{C}(A)$ consisting of costandardly filtered objects. Applying Lemma \ref{holmes-nakano-facts} to $\mathcal{T}^\op$ and dualizing shows that $\mathcal{C}^\nabla(A)$ contains all injective objects of $\mathcal{C}(A)$. For the multiplicity of $\nabla(\lambda)$ in a filtration of $M \in \mathcal{C}^\nabla(A)$ we obtain 
\begin{equation}
\lbrack M:\nabla(\lambda) \rbrack = \dim_K \Hom_{\mathcal{C}(A)}(\ol{\Delta}(\lambda),M)
\end{equation}
and we have the dual Brauer reciprocity
 \begin{equation} \label{eq:dual_BGG}
\lbrack I(\lambda) : \nabla(\mu) \rbrack = \lbrack \ol{\Delta}(\mu):L(\lambda) \rbrack \;.
\end{equation}

%

\subsection{Standardly stratified structure}\label{sec:SSStructure}

We define a partial order $\leq$ on $\Irr \mathcal{G}(T)$ by
\begin{equation}
\lambda < \mu \Longleftrightarrow \deg \lambda < \deg \mu \;.
\end{equation}
This order is obviously interval-finite, but notice that there are non-comparable elements in general, namely those having the same degree. Note that duality $(-)^*$ yields an isomorphism of posets $\Irr \mathcal{G}(T) \simeq \Irr \mathcal{G}(T^\op)$. For $d \in \bbZ$ let $\mathcal{G}_{\leq d}(A)$ be the full subcategory of $\mathcal{G}(A)$ consisting of objects $M$ such that $\lbrack M:L(\lambda) \rbrack \neq 0$ implies $\deg \lambda \leq d$. The full subcategory $\mathcal{G}_{<d}(A)$ is defined similarly. From Lemma \ref{basic_delta_lemma_2} we see that
\begin{equation} \label{filtered_cat_support_condition}
M \in \mathcal{G}_{\leq d}(A) \Longleftrightarrow \Supp M \subset d - \bbN \;.
\end{equation}
Both $\mathcal{G}_{\leq d}(A)$ and $\mathcal{G}_{<d}(A)$ are Serre subcategories of $\mathcal{G}(A)$. We write 
\begin{equation}
\mathcal{G}_d(A) \dopgleich \mathcal{G}_{\leq d}(A)/\mathcal{G}_{<d}(A)
\end{equation}
for the quotient and 
\begin{equation}
\pi_d: \mathcal{G}_{\leq d}(A) \rarr \mathcal{G}_d(A)
\end{equation}
for the quotient functor. The category $\mathcal{G}_d(A)$ is abelian and $\pi_d$ is an exact and essentially surjective functor. 

\begin{defn} \label{highest_weight_def}
We say that $M \in \mathcal{G}(A)$ has \word{highest weight} $\lambda$ if $\lbrack M:L(\mu) \rbrack \neq 0$ implies $\mu \leq \lambda$ and $\lbrack M:L(\lambda) \rbrack = 1$. 
\end{defn}

One can similarly say that $M$ has \word{lowest weight} $\lambda$ if $\lbrack M:L(\mu) \rbrack \neq 0$ implies $\mu \geq \lambda$ and $\lbrack M:L(\lambda) \rbrack = 1$. However, we will not require this notion in this article. We note if $M$ has highest weight $\lambda$, then $L(\lambda)$ can occur anywhere in a composition series of $M$. But if $\Hd M \simeq L(\lambda)$ then this forces $L(\lambda)$ to occur at the \textit{top} of any composition series. 

\begin{lem} \label{delta_is_highest_weight}
Both $\ol{\Delta}(\lambda)$ and $\ol{\nabla}(\lambda)$ have highest weight $\lambda$.
\end{lem}

\begin{proof}
Let $\lambda \in \Irr \mathcal{G}(T)$. We first show that $\ol{\nabla}(\lambda)$ has highest weight $\lambda$. If $\lbrack \ol{\nabla}(\lambda):L(\mu) \rbrack \neq 0$, then $\Supp L(\mu) \subs \Supp \ol{\nabla}(\lambda)$. By Lemma \ref{basic_nabla_lemma_2} we know that $\Supp \ol{\nabla}(\lambda) \subset \deg \lambda - \bbN$ and from Lemma \ref{degree_zero_component} we know that $\deg \mu \in \Supp L(\mu)$. Hence, $\deg \mu \in \deg \lambda - \bbN$, so $\deg \mu \leq \deg \lambda$, implying that $\mu \leq \lambda$. By Theorem \ref{heads} we have $\Soc \ol{\nabla}(\lambda) \simeq L(\lambda)$ and by Lemma \ref{nabla_d_component} we have $\deg \lambda \notin \Supp \ol{\nabla}(\lambda)/\Soc \ol{\nabla}(\lambda)$, hence $\lbrack \ol{\nabla}(\lambda)/\Soc \ol{\nabla}(\lambda):L(\lambda) \rbrack = 0$, so $\lbrack \ol{\nabla}(\lambda):L(\lambda) \rbrack = 1$. This shows that $\ol{\nabla}(\lambda)$ has highest weight $\lambda$. Applying this to $\mathcal{T}^\op$ and dualizing shows that $\ol{\Delta}(\lambda)$ has highest weight $\lambda$.  
\end{proof}



Again note that if $M \in \mathcal{G}^\Delta(A)$ has highest or lowest weight $\lambda$, we cannot locate where $\Delta(\lambda)$ occurs in a standard filtration. But if $\Hd M \simeq L(\lambda)$, then $\Delta(\lambda)$ must occur at the \textit{top} of any standard filtration. Similarly if $M \in \mathcal{G}^\nabla(A)$ has highest or lowest weight $\lambda$ and $\Soc M \simeq L(\lambda)$, then $\nabla(\lambda)$ must occur at the \textit{bottom} of any costandard filtration.

\begin{cor} \label{proj_inj_lowest_weight}
$P(\lambda)$ admits a finite decreasing filtration
\begin{equation} \label{verma_coverma_properties:verma_filtration_equ}
P(\lambda) = F^{0} \supset F^1 \supset \cdots \supset F^{l} = 0
\end{equation}
with quotients $F^i/F^{i+1} \simeq \Delta(\lambda_i)$ such that $\lambda_0=\lambda$ and $\lambda_i > \lambda$ for all $i > 0$. Similarly, $I(\lambda)$ admits a finite increasing filtration 
\begin{equation}
0 = F_0 \subset F_1 \subset \cdots \subset F_m = I(\lambda)
\end{equation}
with quotients $F_i/F_{i-1} \simeq \nabla(\lambda_{i-1})$ such that $\lambda_0 =\lambda$ and $\lambda_i > \lambda$ for all $i > 0$.
\end{cor}

\begin{proof}
We know from Corollary \ref{projectives_standard_filt} that $P(\lambda)$ has a standard filtration. The claim about the filtration now follows directly from Lemma \ref{delta_is_highest_weight} using Brauer reciprocity, Proposition \ref{brauer_reciprocity}:
\[
\lbrack P( \lambda) : \Delta(\lambda) \rbrack = \lbrack \ol{\nabla}(\lambda) : L(\lambda) \rbrack = 1
\] 
and
\[
\lbrack P(\lambda) : \Delta(\mu) \rbrack = \lbrack \ol{\nabla}(\mu) : L(\lambda) \rbrack \Rightarrow \lambda \leq \mu \;.
\]
Applying this to $\mathcal{T}^\op$ and using duality yields the claim for the injective hull.
\end{proof}

Lemma \ref{delta_is_highest_weight} shows in particular that the functors $\ol{\Delta},\ol{\nabla}:\mathcal{G}(T) \rarr \mathcal{G}(A)$  restrict to functors
\begin{equation}
\ol{\Delta}_d,\ol{\nabla}_d: \mathcal{G}_d(T) \rarr \mathcal{G}_{\leq d}(A) \;,
\end{equation}
where $\mathcal{G}_d(T)$ is as defined in \S\ref{section_t_modules}. Below, we will show that $\mathcal{G}_d(A) \simeq \mathcal{G}_d(T)$ and that under this identification $\ol{\Delta}_d$ and $\ol{\nabla}_d$ are left, respectively right, adjoint to the quotient functor $\pi_d$. To this end, we will need the following general lemma which is dual to \cite[III.2, Proposition 5]{Gabriel:1962un}.

\begin{lem}\label{lem:stupidadjunctions}
Let $\mathcal{A},\mathcal{B}$ be abelian categories and $F : \mathcal{A} \rightarrow \mathcal{B}$ an exact functor admitting a right adjoint $G : \mathcal{B} \rightarrow \mathcal{A}$ such that the unit $1_{\mathcal{B}} \rightarrow F \circ G$ of the adjunction is an isomorphism. Then, the functor $F':\mathcal{A}/\Ker F \rarr \mathcal{B}$ induced by $F$ is an equivalence with quasi-inverse $G' \dopgleich \pi \circ G$, where $\pi : \mathcal{A} \rightarrow \mathcal{A} / \Ker F$ is the quotient functor. 
\end{lem}

\begin{proof}
First we note that the fact that $F$ is exact implies that $\Ker F$ is a Serre subcategory of $\mathcal{A}$ and hence the quotient $\mathcal{A} / \Ker F$ is well-defined. Since the adjunction $1_{\mathcal{B}} \rightarrow F \circ G$ is an isomorphism, the functor $F'$ is essentially surjective. We will show that $G'$ is right adjoint to $F'$. Let $U \in \mathcal{A} / \Ker F$ and $M \in \mathcal{B}$. Choosing $V \in \mathcal{A}$ such that $U = \pi(V)$, we have a map
\begin{align*}
\Hom_{\mathcal{B}}(F'(U),M) & = \Hom_{\mathcal{B}}(F(V),M)  \\
 & = \Hom_{\mathcal{A}}(V,G(M)) \stackrel{\pi_{V,G(M)}}{\longrightarrow} \Hom_{\mathcal{A} / \Ker F}(\pi(V),\pi \circ G(M)) \\
& = \Hom_{\mathcal{A} / \Ker F}(U,G'(M)) \;. 
\end{align*}
Thus, we need to show that $\pi_{V,G(M)}$ is an isomorphism. We begin by noting that the adjunction $\Hom_{\mathcal{B}}(F(N),M) = \Hom_{\mathcal{A}}(N,G(M))$ implies that if $N \subset G(M)$ is a subobject such that $F(N) = 0$, then $N = 0$. This implies that 
$$
\Hom_{\mathcal{A} / \Ker F}(\pi(V),\pi \circ G(M)) = \underset{V'}{\mrm{colim}} \  \Hom_{\mathcal{A}}(V',G(M))
$$
where the colimit is over all $V' \subset V$ such that $V/ V' \in\Ker F$. Let $\phi$ be an element of $\Hom_{\mathcal{A}}(V,G(M))$ such that $\pi_{V,G(M)}(\phi) = 0$. Explicitly, $\pi_{V,G(M)}(\phi)  = \underset{V'}{\mrm{colim}} \ \phi |_{V'}$. Thus, if $\pi_{V,G(M)}(\phi)  =0$ then there exists $V'$ such that $\phi |_{V'} = 0$. This means that $\phi$ factors through a map $V/ V' \rightarrow G(M)$. But $V / V' \in \Ker F$ implies that $\Im \phi \subset G(M)$ also belongs to $\Ker F$. Thus, $\phi = 0$ and $\pi_{V,G(M)}$ is surjective. Similarly, if $\psi \in \Hom_{\mathcal{A} / \Ker F}(\pi(V),\pi \circ G(M))$, then by definition this is a collection of morphisms $\psi' : V' \rightarrow G(M)$ such that $\psi'' = \psi' |_{V''}$ if $V'' \subset V'$. In particular, there exists $\psi_0 : V \rightarrow G(M)$ such that $\psi' = \psi_0 |_{V'}$, i.e., $\pi_{V,G(M)}(\psi_0) = \psi$. Thus, $\pi_{V,G(M)}$ is an isomorphism and $G'$ is right adjoint to $F'$. 

Notice that the unit $1_{\mathcal{B}} \rightarrow F' \circ G' = F' \circ \pi \circ G = F \circ G$ is an isomorphism by assumption. Therefore we just need to check that the counit $\eps:G' \circ F' \rightarrow 1_{\mathcal{A}/ \Ker F}$ is an isomorphism. For $U \in \mathcal{A} / \Ker F$ consider the exact sequence
$$
0 \rightarrow \Ker(\eps_U) \rightarrow G' \circ F' (U) \overset{\eps_U}{\rightarrow} U \rightarrow \Coker(\eps_U) \rightarrow 0
$$
in $\mathcal{A}/\Ker F$. Applying the exact functor $F'$ shows that $F'(\Ker(\eps_U)) = 0$ and $F'(\Coker(\eps_U)) = 0$. But $F'$ is conservative by construction. Thus, $\Ker(\eps_U)=0$ and $\Coker(\eps_U) = 0$, implying that $\eps_U$ is an isomorphism. 
\end{proof}

\begin{lem} \label{pi_delta_adjoint}
The category $\mathcal{G}_d(A)$ is canonically equivalent to $\mathcal{G}_d(T)$. Under this identification, the functor $\ol{\Delta}_d$ is an exact left adjoint to $\pi_d$ and the functor $\ol{\nabla}_d$ is an exact right adjoint to $\pi_d$. 
\end{lem}

\begin{proof}
Let $\omega_d:\mathcal{G}_{\leq d}(A) \rarr \mathcal{G}_{d}(T)$ be the projection functor assigning to $M \in \mathcal{G}_{\leq d}(A)$ the homogeneous component $M_{d}$ considered as a $T$-module and to a morphism the restriction onto this component. This is an exact functor. By Lemma \ref{basic_delta_lemma_2} we have a natural isomorphism $1_{\mathcal{G}_{d}(T)} \rightarrow \omega_d \circ \ol{\Delta}_d$. Moreover, for $M \in \mathcal{G}_{\leq d}(A)$ we have a natural morphism $\ol{\Delta}_d \circ \omega_d(M) \rarr M$ by multiplication. This yields an adjunction with $\omega_d$ right adjoint to $\ol{\Delta}_d$. The unit of this adjunction is an isomorphism. We claim that $\Ker \omega_d = \mathcal{G}_{< d}(A)$. If $L(\mu) \in \mathcal{G}_{<d}(A)$, then $d \notin \Supp L(\mu) \subset \deg \mu - \bbN$, which implies that  $\omega_d(L(\mu)) = 0$. On the other hand, if $\omega_d(M) = 0$ then by definition $d \notin \Supp M$ and hence $M \in \mathcal{G}_{<d}(A)$. Lemma \ref{lem:stupidadjunctions} now shows that $\omega_d = \varpi_d \circ \pi_d$, with $\varpi_d : \mathcal{G}_d(A) \rightarrow \mathcal{G}_d(T)$ an equivalence. The claim for $\ol{\nabla}_d$ follows as usual by dualizing this result for $\mathcal{T}^\op$.
\end{proof}

Combining Corollary \ref{proj_inj_lowest_weight} and Lemma \ref{pi_delta_adjoint} we obtain our first main theorem.

\begin{thm} \label{HWC}
The category $\mathcal{G}(A)$ is a standardly stratified category with respect to the degree function $\deg : \Irr \mathcal{G}(T) \rarr \bbZ$, with standard objects $\Delta(\lambda)$, and costandard objects $\nabla(\lambda)$. 
\end{thm}

We deduce:

\begin{cor}
Duality $(-)^*:{\mathcal{G}}(A) \rarr \mathcal{G}(A^{\op})^\circ$, with induced map on posets $(-)^*:\Irr {\mathcal{G}}(T) \rarr \Irr {\mathcal{G}}(T^{\op})$, defines an equivalence of standardly stratified categories.
\end{cor}


\begin{cor} \label{semisimple_hwc}
If $T$ is semisimple, then $\mathcal{G}(A)$ is a highest weight category.
\end{cor}

\subsection{Multi-gradings}\label{sec:multigrading}

We shortly want to address a generalization from $\mathbb{Z}$-gradings to multi-gradings. Fix a $K$-split torus $\mathbb{T} \simeq (K^{\times})^n \subset \mathrm{Aut}(A)$. Let $X = \Hom(\mathbb{T},K^{\times})$. We fix a subset $I \subset X$ such that  $I$ is a basis of $X$. In this setting, we say that $\mc{T} = (A^{-},T,A^+)$ is a \textit{triangular decomposition} if $A^-,T$ and $A^+$ are again graded subalgebras satisfying (a), (c), (d) and (e) of Definition \ref{defn:triangle}, together with 
\begin{center}
(b') $\Supp A^+ \subset \N I$, $\Supp A^- \subset -\N I$, and $T$ is concentrated in degree zero. 
\end{center}
We consider the category $\mathcal{G}_X(A)$ of $X$-graded left $A$-modules. The simple modules in this category are labeled by $\Lambda := \Irr T \times X$. Define a partial ordering on $\Lambda$ by $(\mu,v) \prec (\lambda,w)$ if and only if $v - w \in \N I$.

\begin{thm}
The pair $(\mathcal{G}_X(A), \preceq)$ is a standardly stratified category. 
\end{thm}

\begin{proof}
One can repeat the proof of Theorem \ref{HWC} in this more general setting. Alternatively, one can deduce the theorem directly from Theorem \ref{HWC}. Choose $\rho^{\vee} \in \Hom(K^{\times}, \mathbb{T})$ such that $\langle \alpha , \rho^{\vee} \rangle = 1$ for all $\alpha \in I$. This defines a ``partial forgetful'' functor $F_X$ from $\mathcal{G}_X(A)$ to the category of $\Z$-graded $A$-modules. Since  $(\mu,v) \prec (\lambda,w)$ implies that $(\mu,\langle v,\rho^{\vee} \rangle) < (\lambda,\langle w,\rho^{\vee} \rangle)$, the theorem follows. 
\end{proof}

\begin{cor}\label{cor:multigrad}
If $T$ is semisimple then $(\mathcal{G}_X(A), \preceq)$ is a highest weight category 
\end{cor}

\subsection{Implications for Ext-groups} We go back to the $\mathbb{Z}$-graded setting.

\begin{tcolorbox}
For the remainder of the article, unless explicitly stated otherwise, we assume that $T$ is \textit{semisimple}. Recall that in this case we have $\Delta(\lambda) = \ol{\Delta}(\lambda)$ and $\nabla(\lambda) = \ol{\nabla}(\lambda)$ for all $\lambda$. In particular, $\mathcal{G}(A)$ is a highest weight category.
\end{tcolorbox}

The highest weight structure on $\mc{G}(A)$ immediately implies several results about the Ext-groups in this category. The following properties are proven in \cite{CPS} for an arbitrary highest weight category.

\begin{cor} \label{hwc_ext_properties}
Let $\lambda,\mu \in \Irr \mathcal{G}(T)$. The following holds:
\begin{enum_thm}
\item \label{hwc_ext_properties:delta} If $\Ext_{\mathcal{G}(A)}^n(\Delta(\lambda),L(\mu)) \neq 0$ or $\Ext_{\mathcal{G}(A)}^n(\Delta(\lambda),\Delta(\mu)) \neq 0$, then $\mu \geq \lambda$ and $n$ is at most the maximal length of a chain between $\lambda$ and $\mu$. Moreover, if $n>0$, then $\mu > \lambda$.
\item \label{hwc_ext_properties:nabla} If $\Ext_{\mathcal{G}(A)}^n(L(\mu),\nabla(\lambda)) \neq 0$ or $\Ext_{\mathcal{G}(A)}^n(\nabla(\mu),\nabla(\lambda)) \neq 0$, then $\mu \geq \lambda$ and $n$ is at most the maximal length of a chain between $\lambda$ and $\mu$. Moreover, if $n>0$, then $\mu > \lambda$.
\item \label{hwc_ext_properties:zero_for_large} Let $M \in \mathcal{G}(A)$. Then $\Ext_{\mathcal{G}(A)}^n(M,\nabla(\lambda)) = 0$ and $\Ext_{\mathcal{G}(A)}^n(\Delta(\lambda),M) = 0$ for $n$ sufficiently large.
\item \label{hwc_ext_properties:simples} If $\Ext_{\mathcal{G}(A)}^1(L(\lambda),L(\mu)) \neq 0$, then $\lambda > \mu$ or $\lambda < \mu$.
\item \label{hwc_ext_properties:end_delta} $\End_{\mathcal{G}(A)}(\Delta(\lambda)) \simeq \End_{\mathcal{G}(A)}(L(\lambda))$ and $\End_{\mathcal{G}(A)}(\nabla(\lambda)) \simeq \End_{\mathcal{G}(A)}(L(\lambda))$.
\end{enum_thm}
\end{cor}

\begin{proof}
The statement for $\Ext_{\mathcal{G}(A)}^n(L(\mu),\nabla(\lambda)) \neq 0$ in part \ref{hwc_ext_properties:nabla} is \cite[Lemma 3.8(b)]{CPS}. The statement for $\Ext_{\mathcal{G}(A)}^n(\nabla(\mu),\nabla(\lambda)) \neq 0$ is proven for $n=1$ in \cite[Lemma 3.2(b)]{CPS} but the argument works for general $n$ using the first part for general $n$: if $\Ext_{\mathcal{G}(A)}^n(\nabla(\mu),\nabla(\lambda)) \neq 0$, there is some composition factor $L(\tau)$ of $\nabla(\mu)$ with $\Ext_{\mathcal{G}(A)}^n(L(\tau),\nabla(\lambda)) \neq 0$. Since $\tau \leq \mu$, the statement follows. Part \ref{hwc_ext_properties:delta} is dual to \ref{hwc_ext_properties:nabla}. Parts \ref{hwc_ext_properties:zero_for_large} and \ref{hwc_ext_properties:simples} are \cite[Lemma 3.2(b)]{CPS} and \cite[Lemma 3.8(c)]{CPS}, respectively. Part \ref{hwc_ext_properties:end_delta} follows from $\Ext_{\mathcal{G}(A)}^1(\Delta(\lambda),\Delta(\lambda)) = 0$.
\end{proof}

Since all simple objects of $\mathcal{G}(A)$ are absolutely simple by Proposition \ref{splitting}, the very last statement of Corollary \ref{hwc_ext_properties} simplifies to:

\begin{cor} \label{hom_delta_delta_cor}
$\End_{\mathcal{G}(A)}(\Delta(\lambda)) \simeq K$ and $\End_{\mathcal{G}(A)}(\nabla(\lambda)) \simeq K$.
\end{cor}

\begin{lem}\label{claim:filtr}
If $M \in \mc{G}^{\Delta}(A)$, then $M$ has a filtration $M = M_0 \supset M_1 \supset \cdots \supset M_s = \{ 0 \}$ such that $M_i / M_{i+1} \simeq \Delta(\lambda_i)$ and $\deg \lambda_i \le \deg \lambda_{i+1}$ for all $i = 0, \ds, s-1$.
\end{lem}

\begin{proof}
The proof is by induction on the length of a (any) standard filtration of $M$. Choose $M' \subset M$ such that $M' \in \mc{G}^{\Delta}(A)$ and $M / M' \simeq \Delta(\lambda)$, for some $\lambda$. Then we can choose a filtration $M' = M_1' \supset M_2' \supset \cdots \supset M_s' = \{ 0 \}$ such that $M_i' / M_{i+1}' \simeq \Delta(\lambda_i)$ and $\deg \lambda_i \le \deg \lambda_{i+1}$. There exists some $j$ such that $\deg \lambda_{j-1} < \deg \lambda \le \deg \lambda_{j}$. We may then, without loss of generality, quotient $M$ by $M_j'$ and assume that $j = s$ i.e. $\deg \lambda_i < \deg \lambda$ for all $i < s$. Then we claim that the short exact sequence 
$$
0 \rightarrow M' \rightarrow M \rightarrow \Delta(\lambda) \rightarrow 0
$$
splits. This follows by induction on $s$, using the fact that $\Ext^1_{\mc{G}(A)}(\Delta(\lambda),\Delta(\lambda_i)) = 0$ for all $i < s$; see Corollary \ref{hwc_ext_properties} (a). Thus, $M \simeq M' \oplus \Delta(\lambda)$, and it is clear that we can cook up a filtration on $M$ with the desired properties. 
\end{proof}

\begin{lem} \label{P_filtration_ordered_by_deg}
For any $\lambda \in \Irr \mathcal{G}(T)$ there is a standard filtration of $P(\lambda)$ as in (\ref{verma_coverma_properties:verma_filtration_equ}) with the additional property that $\deg \lambda_i \leq \deg \lambda_{i+1}$
for all $0 \leq i<l$, where $\lambda_0 = \lambda$. The analogous statement for $I(\lambda)$ also holds.
\end{lem}

\begin{proof}
By Corollary \ref{proj_inj_lowest_weight}, $P(\lambda)$ has a filtration $P(\lambda) = F_0 \supset F_1 \supset \cdots $ such that $F_0 / F_1 \simeq \Delta(\lambda)$ and $F_i / F_{i+1} \simeq \Delta(\lambda_i)$, with $\deg \lambda_i > \deg \lambda$, for $i > 0$. Now the lemma follows by applying Lemma \ref{claim:filtr} to $F_1 \in \mc{G}^{\Delta}(A)$. 
\end{proof}


\subsection{BGG property} \label{BGG_section}

By analogy with category $\mc{O}$ for a semisimple complex Lie algebra, we introduce the following terminology:

\begin{defn}
$A$ is \word{BGG} if $\lbrack \ol{\Delta}(\lambda) \rbrack = \lbrack \ol{\nabla}(\lambda) \rbrack$ in $\K_0(\mathcal{G}(A))$ for all $\lambda \in \Irr \mathcal{G}(T)$. 
\end{defn}

Because of the compatibility of standard and costandard modules with respect to shifts, see (\ref{standard_compat_with_shift}), it is sufficient to check the equality $\lbrack \ol{\Delta}(\lambda) \rbrack = \lbrack \ol{\nabla}(\lambda) \rbrack$ only for $\lambda \in \Irr \mathcal{G}_0(T)$. Below we show that the BGG property is equivalent to a symmetry between the left and right Borel subalgebras, and this property can easily be verified in examples, in particular for the VIP examples. From Proposition \ref{prop:injectiveK} we immediately obtain:

\begin{cor}\label{cor:chi_BGG}
If $\chi( \lbrack \ol{\Delta}(\lambda) \rbrack) = \chi(\lbrack \ol{\nabla}(\lambda) \rbrack)$ for all $\lambda \in \Irr \mathcal{G}_0(T)$, then $A$ is BGG.\qed
\end{cor}




\begin{prop}\label{prop:BGGbimodule}
Suppose that $T$ is semisimple. Then $A$ is BGG if and only if $B^- \simeq (B^+)^\circledast$ as graded $T$-bimodules (i.e., $B_i^- \simeq (B_{-i}^+)^\circledast$ as $T$-bimodules for all $i \in \bbZ$).
\end{prop}

\begin{proof}
By Lemma \ref{basic_delta_lemma_2} and Lemma \ref{basic_nabla_lemma_2} we have $\ol{\Delta}(\lambda)_i \simeq B_i^- \otimes_T \lambda$ and $\ol{\nabla}(\lambda)_i \simeq (\lambda^\circledast \otimes_T B_{-i}^+)^\circledast$ as  $T$-modules, for all $\lambda \in \Irr \mathcal{G}_0(T) \simeq \Irr \mathcal{M}(T)$ and all $i \in \bbZ$. Since $T$ is split semisimple by assumption, it is separable and it follows from \cite[\S7, Ex. 20]{Bourb-Alg-8} that $T \otimes_K T^{\mrm{op}}$ is semisimple. Hence, the category of (graded) $T$-bimodules is semisimple. Furthermore, it follows from \cite[\S7, No. 4, Théorème 2]{Bourb-Alg-8} and \cite[\S7, No. 7, Proposition 8]{Bourb-Alg-8} that the  simple $T$-bimodules are precisely the modules $\mu \otimes_K \lambda^\circledast$ with $\lambda,\mu \in \Irr \mathcal{M}(T)$. Hence, in the category of $T$-bimodules we can write 
\[
B^-_i = \bigoplus_{\lambda,\mu \in \Irr {\mathcal{M}(T)}} (\mu \otimes_K \lambda^\circledast)^{\oplus n_{ \mu \lambda}^{(i)}}
\]
for each $i \in \bbZ$ and some $n_{\mu \lambda}^{(i)} \in \bbN$. Since $T$ is split semisimple, we have
\[
(\mu \otimes_K \lambda^\circledast) \otimes_T \nu = \mu \otimes_K (\lambda^\circledast \otimes_T \nu) = \mu \otimes_K \Hom_T(\lambda,\nu) = \left\lbrace \begin{array}{ll} \mu & \tn{if } \lambda = \nu \\ 0 & \tn{else} \end{array} \right.
\]
in $\mathcal{M}(T)$ for every $\nu \in \Irr {\mathcal{M}(T)}$. Hence,
\[
\ol{\Delta}(\lambda)_i \simeq B_i^- \otimes_T \lambda = \bigoplus_{\mu \in \Irr {\mathcal{M}(T)} }\mu^{\oplus n_{\mu \lambda}^{(i)}} 
\]
in $\mathcal{M}(T)$ and therefore
\begin{equation} \label{bgg_property_1}
n_{\mu \lambda}^{(i)} = \lbrack \ol{\Delta}(\lambda)_i : \mu \rbrack \;.
\end{equation}
In a similar fashion, we can write 
\[
(B_{-i}^+)^\circledast = \bigoplus_{\lambda,\mu \in \Irr {\mathcal{M}(T)}} (\mu \otimes_K \lambda^\circledast)^{\oplus m_{\mu \lambda}^{(i)}}
\]
for some $m_{\mu \lambda}^{(i)} \in \bbN$. From this we get
\[
B_{-i}^- = \bigoplus_{\lambda,\mu \in \Irr {\mathcal{M}(T)}} (\lambda \otimes_K \mu^\circledast)^{\oplus m_{\mu \lambda}^{(i)}} \;,
\]
hence
\[
\lambda^\circledast \otimes_T B_{-i}^+ = \bigoplus_{\mu \in \Irr {\mathcal{M}(T)}} (\mu^\circledast)^{\oplus m_{\mu \lambda}^{(i)}} \;,
\]
in $\mathcal{M}(T^{\op})$, and therefore
\[
\ol{\nabla}(\lambda)_i \simeq (\lambda^\circledast \otimes_T B_{-i}^+)^\circledast = \bigoplus_{\mu \in \Irr {\mathcal{M}(T)}} \mu^{\oplus m_{\mu \lambda}^{(i)}} 
\]
in $\mathcal{M}(T)$. Consequently,
\begin{equation}\label{bgg_property_2}
m_{\mu \lambda}^{(i)} = \lbrack \ol{\nabla}(\lambda)_i: \mu \rbrack \;.
\end{equation}
The claim now follows at once from equations (\ref{bgg_property_1}) and (\ref{bgg_property_2}) together with Corollary \ref{cor:chi_BGG}.
\end{proof}




\subsection{Families and standard families}  \label{sec_standard_families}

Since $A$ is a finite-dimensional algebra, it has a decomposition $A = \bigoplus_{i=1}^n B_i$ into indecomposable subrings $B_i$, given as $B_i = Ac_i$ for the primitive idempotents $c_i$ of the center of $A$. This induces a decomposition of the module category $\mc{M}(A) = \bigoplus_{i=1}^n \mc{M}(B_i)$. In particular, every simple $A$-module belongs to a unique block $B_i$. This induces a partition of $\Irr \mc{M}(A)$. We can pull this back to a partition of $\Irr \mc{M}(T)$ using the bijection $L \from \Irr \mc{M}(T) \to \Irr \mc{M}(A)$. The parts of this partition are called the \word{families} of $A$, and the partition is denoted $\mrm{Fam}(A)$. Note that, even though not encoded in the notation, the families depend on the choice of a triangular decomposition. 

Consider the graph with vertices $\Irr \mc{M}(T)$ and an edge between $\lambda$ and $\mu$ if they occur in the same proper standard module $\ol{\Delta}(\eta)$ for some $\eta$. We call the connected components of this graph the \word{standard families} of $A$, denoted $\mrm{Std}(A)$. Since a proper standard module is indecomposable, all its constituents lie in the same block. Hence if $L(\lambda)$ and $L(\mu)$ lie in the same standard family, they also lie in the same family.  

\begin{lem}
If $A$ is BGG, the standard families are equal to the families.
\end{lem}

\begin{proof}
For $\lambda, \mu \in \Irr \mc{M}(T)$ let us write $\lambda \sim \mu$ if  $L(\lambda)$ and $L(\mu)$ are constituents of $P(\eta)$ for some $\eta$. Assume that this is the case. Since $P(\eta)$ has a standard filtration, there is $\alpha$ with $\lbrack P(\eta):\Delta(\alpha) \rbrack \neq 0$ and $\lbrack \Delta(\alpha) : L(\lambda) \rbrack \neq 0$. Similarly, there is $\beta$ such that  $\lbrack P(\eta):\Delta(\beta) \rbrack \neq 0$ and $\lbrack \Delta(\beta) : L(\mu) \rbrack \neq 0$. Now, by Brauer reciprocity (\ref{eq:BGG}) and the BGG property we obtain $\lbrack \Delta(\alpha):L(\eta) \rbrack \neq 0$ and $\lbrack \Delta(\beta) : L(\eta) \rbrack \neq 0$. We thus see that $\lambda$ and $\eta$ lie in the same standard family, and also $\mu$ and $\eta$ lie in the same standard family. Hence, $\lambda$ and $\mu$ lie in the same standard family. Since the block relation is generated by $\sim$, this proves the claim.
\end{proof}

\section{Tilting theory}\label{more_on_tilting}

An object $M \in \mathcal{G}(A)$ is said to be \word{tilting} if $M \in \mc{G}^t(A) \dopgleich \mc{G}^{\Delta}(A) \cap \mc{G}^{\nabla}(A)$, i.e., $M$ has both a standard and a costandard filtration. It follows directly from Corollary \ref{projectives_standard_filt} and its dual version that $\mc{G}^t(A)$ is closed under direct sums and under direct summands in $\mathcal{G}(A)$, so it is a Krull–Schmidt category. By Corollary \ref{projectives_standard_filt} the projective-injective objects are tilting. We show that the converse holds.

\begin{thm} \label{tilings_projinj}
The tilting objects in $\mathcal{G}(A)$ are precisely the projective-injective objects.
\end{thm}

\begin{proof}
Suppose that $M \in {\mathcal{G}}(A)$ is projective-injective. Since $M$ is projective, clearly $\Res_{B^-}^A M$ is projective, so $M \in \mathcal{G}^{\Delta}(A)$ by Lemma \ref{holmes-nakano-facts}. Moreover, since $M$ is injective, its dual $M^*$ is projective, so $\Res_{B^+}^A M^*$ is projective, hence its dual $\Res_{B^+}^A M$ is injective. The dual version of Corollary \ref{projectives_standard_filt} thus shows that $M \in \mathcal{G}^{\nabla}(A)$. Consequently, $M \in \mathcal{G}^{t}(A)$.  Conversely, suppose that $M \in \mathcal{G}^{t}(A)$. Let $0 = M_0 \subset M_1 \subset \cdots \subset M_s = M$ be a standard filtration with $M_j/M_{j-1} \simeq \Delta(\lambda_j)$. Let $q:M \twoheadrightarrow \Delta(\lambda_s)$ be the quotient morphism. We know from Proposition \ref{proj_inj_lowest_weight} that $\Delta(\lambda_s)$ is at the top of a standard filtration of $P(\lambda_s)$, so we have a quotient morphism $\pi:P(\lambda_s) \twoheadrightarrow \Delta(\lambda_s)$. Due to the projectivity of $P(\lambda_s)$ there is a morphism $\eta:P(\lambda_s) \rarr M$ making the diagram
\begin{equation} \label{tilting_proj_inj_diag1}
\begin{tikzcd}
& M \arrow[twoheadrightarrow]{d}{q} \\
P(\lambda_s) \arrow[twoheadrightarrow]{r}[swap]{\pi} \arrow[dashed]{ru}{\eta} & \Delta(\lambda_s)
\end{tikzcd}
\end{equation}
commutative. By Proposition \ref{proj_inj_lowest_weight}, $\Ker \pi$ also has a standard filtration. Since $M$ is tilting, it also has a costandard filtration. An inductive application of the Ext-vanishing statement in Lemma \ref{holmes-nakano-facts}\ref{holmes-nakano-facts:ext} thus shows that $\Ext^1_{\mathcal{G}(A)}(M,\Ker \pi) = 0$. Hence, applying $\Hom_{\mathcal{G}(A)}(M,-)$ to the exact sequence
\[
\begin{tikzcd}[column sep=small]
0 \arrow{r} & \Ker \pi \arrow{r} & P(\lambda_s) \arrow{r}{\pi} & \Delta(\lambda_s) \arrow{r} & 0 
\end{tikzcd}
\]
yields an exact sequence
\[
\begin{tikzcd}[column sep=small]
0 \arrow{r} & \Hom_{\mathcal{G}(A)}(M,\Ker \pi) \arrow{r} & \Hom_{\mathcal{G}}(M,P(\lambda_s)) \arrow{r} & \Hom_{\mathcal{G}(A)}(M,\Delta(\lambda_s)) \arrow{r} & 0 \;.
\end{tikzcd}
\]
In particular, there is $\nu \in \Hom_{\mathcal{G}(A)}(M,P(\lambda_s))$ making the diagram
\begin{equation} \label{tilting_proj_inj_diag2}
\begin{tikzcd}
M \arrow{r}{\nu} \arrow[twoheadrightarrow]{d}[swap]{q} & P(\lambda_s)  \arrow[twoheadrightarrow]{dl}{\pi}  \\
\Delta(\lambda_s)
\end{tikzcd}
\end{equation}
commutative. From Diagrams (\ref{tilting_proj_inj_diag1}) and (\ref{tilting_proj_inj_diag2}) we obtain a commutative diagram
\[
\begin{tikzcd}
P(\lambda_s) \arrow{rr}{\nu \circ \eta} \arrow[twoheadrightarrow]{dr}[swap]{\pi} && P(\lambda_s) \arrow[twoheadrightarrow]{dl}{\pi} \\
& \Delta(\lambda_s) \arrow[twoheadrightarrow]{d} \\
& L(\lambda_s)
\end{tikzcd}
\]  
The uniqueness of projective covers of $L(\lambda_s)$ now shows that $\nu \circ \eta$ is an isomorphism. In particular, $\nu$ is surjective and therefore $P(\lambda_s)$ is a direct summand of $M$. We have thus shown that the projective cover corresponding to the top part in a standard filtration of a tilting object is a direct summand. By induction on the length of the standard filtration we obtain that $M$ is in fact projective. Now, the same result applies to $\mathcal{T}^\op$, showing that the tilting object $M^* \in \mathcal{G}^{t}(A^\op)$ is projective. Hence $M$ is injective.
\end{proof}

\subsection{Self-injectivity} \label{self_injectivity}

For the moment, $A$ can be an arbitrary finite-dimensional graded $K$-algebra. Recall that $A$ is said to be \word{self-injective} if the left $A$-module $A$ is injective. This is equivalent to the class of projective objects of $\mathcal{M}(A)$ being equal to the class of injective objects of $\mathcal{M}(A)$. Using Lemma \ref{forget_functor} we see:

\begin{cor} \label{self_injective_proj_inj}
The algebra $A$ is self-injective if and only if the class of projective objects in $\mathcal{G}(A)$ is equal to the class of injective objects in $\mathcal{G}(A)$. \qed
\end{cor}

Suppose that $A$ is self-injective. Then the projective cover $P(S)$ for $S \in \Irr \mathcal{C}(A)$ is also an indecomposable injective object and so it has a simple socle, say $\nu_\mathcal{C}(S)$. In other words, $\Soc P(S) \simeq \Hd P(\nu_\mathcal{C}(S))$. We get a map 
\begin{equation}
\nu_\mathcal{C}: \Irr \mathcal{C}(A) \rarr \Irr \mathcal{C}(A) \;,
\end{equation}
called the (graded) \word{Nakayama permutation} of $A$. In the ungraded case it is well-known that this is indeed a permutation. In the graded setting, note that $\Soc P(S \lbrack n \rbrack) \simeq \Soc P(S) \lbrack n \rbrack$. Hence, $\nu_\mathcal{G}(S \lbrack n \rbrack) = \nu_\mathcal{G}(S) \lbrack n \rbrack$ for any $S \in \Irr \mathcal{G}(A)$, so once we know $\nu_\mathcal{G}(S)$ for all $S \in \Irr_0 \mathcal{G}(A)$, we know $\nu_\mathcal{G}(S)$ for all $S \in \Irr \mathcal{G}(A)$, and this shows that $\nu_\mathcal{G}$ is also a permutation.

\begin{defn} \label{graded_frob_defn}
The algebra $A$ is said to be \word{$d$-Frobenius} if there is a linear map $\Phi : A \rightarrow K$ such that $\Phi(A_i) = 0$ for all $i \neq d$ and $\Ker \Phi$ does not contain any non-trivial left (equivalently, right) ideal of $A$. If, in addition, $\Phi(ab) = \Phi(ba)$ for all $a,b \in A$, then we say that $A$ is \word{$d$-symmetric}. If $d = 0$, then we say that $A$ is \word{graded Frobenius}, resp. \word{graded symmetric}. 
\end{defn}

\begin{lem}\label{lem:gradedhomcheck}
Let $\Phi: A \rightarrow K$ be a linear map such that $\Phi(A_i) = 0$ for all $i \neq d$. Then $A$ is $d$-Frobenius if and only if $\Phi$ does not contain any non-trivial \textnormal{graded} left (equivalently, right) ideal of $A$.
\end{lem}

\begin{proof}
If $A$ is $d$-Frobenius, the condition clearly holds. Conversely, let $I$ be a non-zero left ideal. We need to show that $\Phi(I) \neq 0$. Let $a \in I$ be non-zero. We can write $a = \sum_{ j \in \Z} a_j$, with $0 \neq a_i \in A_i$. Since $Aa_i$ is a graded left ideal, we have $\Phi(A a_i) \neq 0$. So, there exists a homogeneous element $b \in A$, $\deg b = d - i$, such that $\Phi(b a_i) \neq 0$. But then $\Phi(b a) = \sum_j \Phi(b a_j) = \Phi(b a_i) \neq 0$
since $\Phi(b a_j) = 0$ for all $j \neq i$. 
\end{proof}

Clearly, if $A$ is $d$-Frobenius, it is Frobenius in the usual sense, thus self-injective. The proof of the following is easily adapted from the non-graded setting, c.f. \cite[\S 1.6]{Benson}. 

\begin{lem} \label{frobenius_dual_iso}
The algebra $A$ is $d$-Frobenius if and only if 
\begin{equation}
({}_{A^{\op}} A^{\op})^* \simeq A[-d]
\end{equation}
as graded $A$-modules.
\end{lem}

\begin{lem} \label{graded_symmetric_trivial_nakayama}
If $A$ is graded symmetric then the graded Nakayama permutation is trivial. Hence $\Soc P \simeq \Hd P$ in $\mc{G}(A)$ for any projective indecomposable object $P \in \mc{G}(A)$.
\end{lem}

\begin{proof}
The statement for the ungraded Nakayama permutation is well-known, see \cite[Theorem 1.6.3]{Bensonfd}. In the graded case, an indecomposable projective objects of $\mc{G}(A)$ is of the form $A\lbrack d\rbrack e$ for a primitive idempotent of $A$ of degree zero and some integer $d \in \bbZ$, see \cite[Proposition 5.8]{GG-Graded-Artin}. One can now use the same proof as in \textit{loc. cit.} to prove the statement in the graded setting.
\end{proof}

\subsection{Ringel's tilting objects}

We return to assuming that $A$ has a triangular decomposition. Using Corollary \ref{self_injective_proj_inj} we obtain from Theorem \ref{tilings_projinj}:

\begin{cor} \label{self_inj_tiltings_projective}
If $A$ is self-injective, then the tilting objects in $\mathcal{G}(A)$ are precisely the projective objects. In particular, the indecomposable tilting objects are precisely the $P(\lambda)$ for $\lambda \in \Irr \mathcal{G}(T)$.
\end{cor}

In this section, we show that $\mc{G}(A)$ admits an abstract tilting theory, in the sense of \cite[Appendix A]{hwtpaper2}. Recall from \S\ref{self_injectivity} that if $A$ is self-injective, then we have a graded Nakayama permutation $\nu:\Irr \mc{G}(A) \rarr \Irr \mc{G}(A)$ defined by
\begin{equation} \label{nakayama_perm}
\Soc P(\lambda) \simeq L(\nu(\lambda)) \;,
\end{equation}
Since $P(\lambda)$ is indecomposable and injective, it is the injective hull of its socle, so 
\begin{equation} \label{p_inj_nu}
P(\lambda) = I(\nu(\lambda)) 
\end{equation}
and therefore
\begin{equation}
\Hd I(\lambda) = L(\nu^{-1}(\lambda)) \;.
\end{equation}
The algebra $A^\op$ is self-injective too, and using duality we obtain
\begin{equation} \label{nakayama_dual}
\nu(\lambda^*) = \nu^{-1}(\lambda)^* \;,
\end{equation}
where we have denoted the Nakayama permutation of $A^\op$ again by $\nu$.

\begin{thm} \label{tiltings_main_thm}
Suppose that $A$ is self-injective. Then $P(\lambda)$ is a \textnormal{highest} weight object for any $\lambda \in \Irr \mathcal{G}(T)$. Moreover:
\begin{enum_thm}
\item If $\lambda^h$ denotes the highest weight of $P(\lambda)$, then the map $\lambda \mapsto \lambda^h$ is a permutation on $\Irr \mathcal{G}(T)$. 

\item $\Soc \Delta(\lambda) \simeq L(\lambda^\dagger)$, where $(-)^\dagger \dopgleich \nu \circ (-)^{h^{-1}}$ is a permutation on $\Irr \mathcal{G}(T)$. 

\item \label{tiltings_main_thm:Delta_bottom} $\Delta(\lambda^h)$ is at the bottom of any standard filtration of $P(\lambda)$.

\item \label{tiltings_main_thm:h_dual} We have $(\lambda^h)^* = (\nu(\lambda)^*)^h$.

\item $\nabla(\lambda^h)$ is at the top of any costandard filtration of $P(\lambda)$.

\item $\Hd \nabla(\lambda) = L(\nu^{-1}(\lambda^\dagger)) = L(\lambda^{h^{-1}})$.

\end{enum_thm}
\end{thm}

\begin{proof}
Let $\lambda \in \Irr \mathcal{G}(T)$. Let us first show that $P(\lambda)$ is a highest weight object. We know from Lemma \ref{P_filtration_ordered_by_deg} that there is a standard filtration $P(\lambda) = F^0 \supset F^1 \supset \ldots \supset F^{l-1} \supset F^l = 0$ with quotients $F^i/F^{i+1} \simeq \Delta(\lambda_i)$ such that additionally $\deg \lambda_{i} \leq \deg \lambda_{i+1}$ for all $0 \leq i < l$. We claim that $\deg \lambda_{l-2} < \deg \lambda_{l-1}$. Suppose that $\deg \lambda_{l-2} = \deg \lambda_{l-1}$. Then $\lambda_{l-2}$ and $\lambda_{l-1}$ are not comparable, so certainly $\lambda_{l-2} \not< \lambda_{l-1}$, hence 
\[
\Ext_{\mathcal{G}(A)}^1(\Delta(\lambda_{l-2}),\Delta(\lambda_{l-1})) = 0
\]
by Corollary \ref{hwc_ext_properties}. Consequently, the exact sequence
\[
\begin{tikzcd}[column sep=small]
0 \arrow{r} & F^{l-1} = \Delta(\lambda_{l-1}) \arrow{r} & F^{l-2} \arrow{r} & F^{l-2}/F^{l-1} = \Delta(\lambda_{l-2}) \arrow{r} & 0 
\end{tikzcd}
\]
splits so that $F^{l-2} \simeq \Delta(\lambda_{l-1}) \oplus \Delta(\lambda_{l-2})$. But this clearly contradicts the simplicity of the socle of $P(\lambda)$. We thus must have $\deg \lambda_{l-2} <  \deg \lambda_{l-1}$. This then implies that in fact $\deg \lambda_i < \deg \lambda_{l-1}$ for all $0 \leq i<l-1$, so $\lambda_{i} < \lambda_{l-1}$ for all $0 \leq i < l-1$. Hence, $P(\lambda)$ has highest weight $\lambda_{l-1} \gleichdop \lambda^h$. Since $P(\lambda)$ has simple socle $L(\nu(\lambda))$, we must have
\begin{equation}
L( (\lambda^h)^\dagger) \simeq \Soc \Delta(\lambda^h) \simeq L(\nu(\lambda)) \;.
\end{equation}
Hence, all standard objects $\Delta(\lambda)$ have pairwise non-isomorphic simple socle, and this forces $\Delta(\lambda^h)$ to appear at the bottom of any standard filtration of $P(\lambda)$. Since $P(\lambda)$ has highest weight $\lambda^h$, it follows that $P(\lambda)^*$ has highest weight $(\lambda^h)^*$. Now, since $A$ is self-injective, $P(\lambda)^*$ is projective with head $\Soc P(\lambda)^* = L(\nu(\lambda))^* = L(\nu(\lambda)^*)$, so $P(\lambda^*) = P(\nu(\lambda)^*)$. It follows that $(\lambda^h)^* = (\nu(\lambda)^*)^h$. We thus know from part \ref{tiltings_main_thm:Delta_bottom} that $\Delta( (\lambda^h)^* )$ is at the bottom of any standard filtration of $P(\lambda)^*$. Dualizing shows that $\nabla(\lambda^h)$ is at the top of any costandard filtration of $P(\lambda)$. Moreover,
\[
\Soc \Delta((\lambda^h)^*) = \Soc \Delta( (\nu(\lambda)^*)^h) = L(\nu (\nu(\lambda^*))) = L(\lambda^*) \;,
\]
using (\ref{nakayama_dual}), so $\Hd \nabla(\lambda^h) \simeq L(\lambda)$. 
\end{proof}

From Theorem \ref{tiltings_main_thm} we immediately obtain:

\begin{cor} \label{good_tilting_theory}
Suppose that $A$ is self-injective. For $\lambda \in \Irr \mathcal{G}(T)$ define
\begin{equation} \label{tilting_object_defn}
T(\lambda) \dopgleich P(\nu^{-1}(\lambda^\dagger)) = P(\lambda^{h^{-1}}) \;.
\end{equation}
This is an indecomposable tilting object in $\mathcal{G}(A)$. It has highest weight $\lambda$, an injection $\Delta(\lambda) \hookrightarrow T(\lambda)$, and a projection $T(\lambda) \twoheadrightarrow \nabla(\lambda)$. Moreover, the map $\lambda \mapsto T(\lambda)$ is a bijection between $\Irr \mathcal{G}(T)$ and the isomorphism classes of indecomposable tilting objects in $\mathcal{G}(A)$. \qed
\end{cor}

We record the following consequence of Corollary \ref{self_inj_tiltings_projective} and Brauer reciprocity:

\begin{cor}
Suppose that $A$ is self-injective and BGG. Then 
\begin{equation}
\lbrack T:\Delta(\lambda) \rbrack = \lbrack T:\nabla(\lambda) \rbrack
\end{equation}
for any tilting object $T \in \mathcal{G}(A)$. \qed
\end{cor}

Note that if $A$ is self-injective, so is $A^\op$ and therefore we have an analogous tilting theory in $\mathcal{G}(A^\op)$. This is linked to the one in $\mathcal{G}(A)$ by duality:

\begin{lem}
Suppose that $A$ is self-injective. Then $T(\lambda^*) = T(\lambda)^*$ in $\mathcal{G}(A^\op)$. \qed
\end{lem}


\begin{ex} \label{example_no_tilting}
Here is an example where the category $\mathcal{G}(A)$ does not contain any tilting objects. Recall from Example \ref{first_examples_of_triang_dec} the triangular decomposition $K\lbrack x \rbrack/(x^2) \otimes_K K \otimes_K K \lbrack y \rbrack/(y^2)$ of $A = K \langle x,y \rangle/ \langle x^2,yx,y^2 \rangle$ with $\deg(x)= -1$ and $\deg(y)= 1$. It follows from Theorem \ref{heads} that $A$ has only one simple module $L$ up to isomorphism, so it only has one indecomposable projective module $P$ and only one indecomposable injective module $I$ (again up to isomorphism). Suppose that we can show that $A$ is not self-injective. Then there exists a finitely generated projective $A$-module which is not injective. But this must imply that $I$ is not isomorphic to $P$ as otherwise all projective modules would be injective. This in turn implies that there is no projective-injective $A$-module, so there is no tilting object in $\mathcal{G}(A)$ by Theorem \ref{tilings_projinj}. 

We now  argue that $A$ is not self-injective. The subspace $\langle x,y,xy \rangle_K$ of $A$ is clearly a nilpotent ideal. For dimension reasons the Jacobson radical is then already equal to this ideal. From this one easily obtains that
\[
\Soc(_AA) = \lbrace a \in A \mid \Rad(A)a = 0 \rbrace = \langle x,xy \rangle_K
\]
and
\[
\Soc(A_A) = \lbrace a \in A \mid a \Rad(A) = 0 \rbrace = \langle y,xy \rangle_K \;,
\]
where we use the elementary description of the left (right) socle as the left (right) annihilator of the radical, see \cite[Lemma 58.3]{CR}. Hence, $\Soc(_A A ) \neq \Soc(A_A)$, so $A$ is not self-injective by \cite[Theorem 58.12]{CR}. 
\end{ex}

\section{Triangular dualities} \label{duality_from_anti}

If $A$ is self-injective, then $A$ itself is a tilting object in $\mathcal{G}(A)$. Our aim is to show that if $A$ is graded Frobenius, then this tilting object is fixed by certain dualities on $\mathcal{G}(A)$. This property will play a key role in \cite{hwtpaper2}.

Note that a (graded) automorphism $\tau$ of $A$ induces an equivalence $\,^\tau(-): \mathcal{C}(A) \rarr \mathcal{C}(A)$, called the \word{twist} by $\tau$. For $M \in {\mathcal{C}}(A)$ the action of $A$ on the twisted module $\,^\tau M \in {\mathcal{C}}(A)$ is given by $a \star m := \tau(a) m$ for $a \in A$ and $m \in M$. Hence, if $\tau$ is an (anti-graded) anti-automorphism of $A$, it is a graded isomorphism $A \rarr A^\op$ and so we get an equivalence 
\begin{equation}
\,^\tau (-): \mathcal{C}(A) \rarr \mathcal{C}(A^\op) \;.
\end{equation} 
In order to make meaningful statements, the anti-involution is required to respect the triangular structure on $A$. More precisely:

\begin{defn}\label{defn:triantinv}
An anti-graded anti-automorphism $\tau$ of $A$ is said to be a \textit{triangular anti-involution} if it satisfies the following conditions:
\begin{enum_thm}
\item $\tau$ is of order $2$.
\item \label{duality_assumption_T_swap} $\tau(A^-) = A^+$.
\item \label{duality_assumption_T_simples} $\,^\tau \lambda \simeq \lambda^*$ as $T^\op$-modules for all $\lambda \in \Irr \mathcal{M}(T)$.
\end{enum_thm}
\end{defn}

We assume now that $\tau$ is triangular. Since $\tau$ is anti-graded, property \ref{duality_assumption_T_swap} implies that $\tau$ stabilizes $T$ and so it induces an anti-automorphism of $T$. The induced twist $\mathcal{C}(T) \rarr \mathcal{C}(T^\op)$ is of course just the restriction of the twist $\mathcal{C}(A) \rarr \mathcal{C}(A^\op)$. Property  \ref{duality_assumption_T_simples} concerns this restriction. Note that $\tau(B^-) = B^+$. A straightforward check shows that we have an equality of functors
 \begin{equation}
\,^\tau(-) \circ (-)^* = (-)^* \circ \,^\tau(-) : \mathcal{C}(A) \rarr \mathcal{C}(A)^\circ
\end{equation}
and we denote this functor by $\bbD$. The above equation shows directly that $\mathbb{D}^2 \simeq \id_{\mathcal{C}(A)}$, so $\bbD$ is a contravariant involution on $\mathcal{C}(A)$. 
%
 %

The following theorem is essentially due to Holmes and Nakano \cite[Theorem 5.1]{HN}. It shows that $\bbD$ is a duality on $\mc{G}(A)$ fixing the simple objects, so it is a \word{strong duality} in the sense of Cline–Parshall–Scott \cite[\S1.2]{StratifiedEnd}.

\begin{thm}[Holmes--Nakano] \label{triangular_duality}
Assume that $A$ is equipped with a triangular anti-involution~$\tau$. Then for any $\lambda \in \Irr \mathcal{C}(T)$ we have canonical isomorphisms
\[
\mathbb{D}(\Delta(\lambda)) \simeq \nabla(\lambda) \;, \quad \mathbb{D}(L(\lambda)) \simeq L(\lambda) \;, \quad \mathbb{D}(P(\lambda)) \simeq I(\lambda) \;, \quad \mathbb{D}(I(\lambda)) \simeq P(\lambda) 
\]
in $\mathcal{C}(A)$. Moreover, $A$ is BGG.
\end{thm}

\begin{proof}
We assume that $\lambda \in \Irr \mathcal{G}_0(T)$, the general result follows by degree shifting. By definition, we have $\nabla(\lambda) = (A^\op \otimes_{(B^-)^\op} \lambda^*)^* = (\lambda^\circledast \otimes_{B^-} A)^\circledast$. Since $\,^\tau \lambda \simeq \lambda^*$ by assumption, we have $\nabla(\lambda) \simeq ( \,^\tau\lambda \otimes_{B^-} A)^\circledast$. We claim that there is an $A$-module isomorphism
\[
\phi:  ( \,^\tau \lambda \otimes_{B^-} A)^\circledast \longrightarrow \,^\tau \left((A \otimes_{B^{+}} \lambda)^* \right) = \mathbb{D}(\Delta(\lambda)) \;.
\] 
This proves that $\mathbb{D}(\Delta(\lambda)) \simeq \nabla(\lambda)$. Note that the vector space structure is not affected by twisting. Thus, for $f \in (\,^\tau \lambda \otimes_{B^-} A)^\circledast$, we define a $K$-linear function $\phi(f):A \otimes_{B^{+}} \lambda \rarr K$, i.e. an element $\phi(f) \in \,^\tau \left((A \otimes_{B^{+}} \lambda)^* \right)$, by
\[
\phi(f)(\underbrace{a \otimes v}_{\in A \otimes_{B^+}\lambda}) \dopgleich f(\underbrace{v \otimes \tau(a)}_{\in \,^\tau \lambda \otimes_{B^-} A}) 
\]
for $a \in A$ and $v \in \lambda$. We first need to check that $F \dopgleich \phi(f)$ is indeed a $K$-linear map $A \otimes_{B^+} \lambda \rarr K$. This amounts to showing that 
\[
F( a b^+ \otimes v) = F(a \otimes b^+ v)
\]
for all $b^+ \in B^+$. In fact, by definition we have
\[
F( a b^+ \otimes v) = f(v \otimes \tau(ab^+)) = f(v \otimes \tau(b^+) \tau(a)) = f( b^+ v  \otimes \tau(a)) = F(a \otimes b^+ v)\;,
\]
using the fact that $\tau(b^+) \in B^-$ and that $v \otimes \tau(a) \in \,^\tau \lambda \otimes_{B^-} A$. Hence, $\phi$ is well-defined. It is clear that $\phi$ is a $K$-vector space morphism, and if $\phi(f) = 0$, then clearly $f=0$. Thus, $\phi$ is injective. Since the $K$-vector space dimensions of the domain and codomain of $\phi$ are equal, it follows immediately that $\phi$ is a $K$-vector space isomorphism. All that remains to show is that $\phi$ is an $A$-module morphism. Therefore, let $a' \in A$. We need to show that
\begin{equation} \label{phi_a_morphism}
\phi(a'f) = a'\phi(f) \;.
\end{equation}
Recall the definition of the $A$-action on a dual module given in \S\ref{notations}. For the left hand side of (\ref{phi_a_morphism}) we have
\[
\phi(a'f)(a \otimes v) = (a'f)(v \otimes \tau(a)) = f(v \otimes \tau(a)a') \;.
\]
Noting that the codomain of $\phi$ is a $\tau$-twisted module we get for the right hand side
\begin{align*}
(a'\phi(f))(a \otimes v) &= (\phi(f) \tau(a'))(a \otimes v) = \phi(f)(\tau(a') a \otimes v) = f(v \otimes \tau(\tau(a')a)) \\ &= f(v \otimes \tau(a) a') \;.
\end{align*}
Hence, we indeed have equality in (\ref{phi_a_morphism}). This shows that $\mathbb{D}(\Delta(\lambda)) \simeq \nabla(\lambda)$. This implies $\,^\tau \nabla(\lambda) \simeq \Delta(\lambda)^*$. Since $\Soc (\nabla(\lambda)) \simeq L(\lambda)$, we get
\[
\,^\tau L(\lambda) \simeq \Soc (\,^\tau \nabla(\lambda)) \simeq \Soc (\Delta(\lambda)^*) \simeq (\Hd(\Delta(\lambda)))^* \simeq L(\lambda)^* \;,
\]
so $\mathbb{D}(L(\lambda)) \simeq L(\lambda)$.

Since $\mathbb{D}$ is a duality, it maps projective covers to injective hulls, so $\mathbb{D}(P(\lambda))$ is an injective hull of a simple module. Applying $\mathbb{D}$ to the epimorphism $P(\lambda) \twoheadrightarrow L(\lambda)$ yields a monomorphism $L(\lambda) \simeq \mathbb{D}(L(\lambda)) \hookrightarrow \mathbb{D}(P(\lambda))$, showing that $\mathbb{D}(P(\lambda))$ is the injective hull of $L(\lambda)$, so $\mathbb{D}(P(\lambda)) \simeq I(\lambda)$ by uniqueness of the injective hull. From this we immediately obtain $\mathbb{D}(I(\lambda)) \simeq P(\lambda)$. For all $\lambda,\mu \in \Irr \mathcal{G}(T)$ we now have 
\[
\lbrack \Delta(\lambda):L(\mu) \rbrack = \lbrack \bbD( \Delta(\lambda)):\bbD (L(\mu)) \rbrack = \lbrack \nabla(\lambda):L(\mu) \rbrack,
\]
and this shows that $A$ is BGG.
\end{proof}

\begin{lem} \label{duality_functor_dual}
We have $\bbD(_AA) \simeq \left(_{A^\op}A^\op\right)^*$ in $\mathcal{G}(A)$.
\end{lem}

\begin{proof}
Since $_AA \in \mathcal{G}(A)$ is projective and since $\bbD$ is a contravariant equivalence, the object $\bbD(_AA) \in \mathcal{G}(A)$ is injective. If we decompose ${}_A A = \bigoplus_{\lambda \in \Irr \mc{G}(T)} P(\lambda)^{\oplus n_{\lambda}}$, then 
$$
n_{\lambda} = \dim \Hom_{\mc{G}(A)}(A,L(\lambda)) = \dim L(\lambda)_0,
$$
and Theorem \ref{triangular_duality} implies that $\bbD(A) = \bigoplus_{\lambda \in \Irr \mc{G}(T)} I(\lambda)^{\oplus n_{\lambda}}$. Analogously, we have $\left(_{A^\op}A^\op\right)^* = \bigoplus_{ \lambda \in \Irr \mathcal{G}(T)} I(\lambda)^{\oplus m_\lambda}$ with 
\begin{align*}
m_{\lambda} & = \dim \Hom_{\mc{G}(A)}\left( L(\lambda),\left(_{A^\op}A^\op\right)^* \right) \\
 & = \dim \Hom_{\mc{G}(A^{\op})}( {}_{A^\op}A^\op, L(\lambda)^*) = \dim L(\lambda)^*_0. 
\end{align*}
Since standard duality preserves the grading, $\dim L(\lambda)^*_ 0 = \dim L(\lambda)_0$ and hence $m_{\lambda} = n_{\lambda}$ for all $\lambda \in \Irr \mc{G}(T)$. 
\end{proof}

\begin{cor}\label{cor:gradedFrobeniusduality}
If $A$ is graded Frobenius, then $\bbD(A) \simeq A$ in $\mathcal{G}(A)$.
\end{cor}

\begin{proof}
This follows from Lemma \ref{duality_functor_dual} and Lemma \ref{frobenius_dual_iso}.
\end{proof}

\section{Abstract Kazhdan–Lusztig theory}  \label{sec:kl_theory}

The degree function $\deg : \Irr \mc{G}(T) \rightarrow \Z$ can be thought of as a length function on $\mc{G}(A)$, in the sense of \cite{CPSKL}. Then it is natural to ask when $\mc{G}(A)$ admits an abstract Kazhdan-Lusztig theory, as in \emph{loc. cit.} In our setting, this means that 
\begin{equation}\label{eq:KL1}
\Ext_{\mc{G}(A)}^m(\Delta(\lambda),L(\mu)) \neq 0 \quad \Rightarrow \quad m \equiv \deg \lambda - \deg \mu \ \mod \ 2, 
\end{equation}
and 
\begin{equation}\label{eq:KL2}
\Ext_{\mc{G}(A)}^m(L(\mu),\nabla(\lambda)) \neq 0 \quad \Rightarrow \quad m \equiv \deg \lambda - \deg \mu \ \mod  \ 2,  
\end{equation}
for all $\lambda,\mu \in \Irr \mc{G}(T)$. 

In all that follows, an abstract Kazhdan-Lusztig theory will always be in relation to the function $\deg$. We say that $A$ satisfies the KL-property if both (\ref{eq:KL1}) and (\ref{eq:KL2}) hold. 
Recall that we have:

\begin{lem}
There are canonical isomorphisms
\begin{equation}\label{eq:KL1a}
\Ext_{\mc{G}(A)}^m(\Delta(\lambda),L(\mu)) \simeq \Ext_{\mc{G}(B^{+})}^m(\lambda,L(\mu)),
\end{equation}
and 
\begin{equation}\label{eq:KL2a}
\Ext^m_{\mc{G}(A)}(L(\mu),\nabla(\lambda)) \simeq \Ext^m_{\mc{G}(B^{+})}(L(\mu), \lambda).
\end{equation}
\end{lem}

The above adjunctions make it clear that the KL-property is really about the structure of $L(\mu)$ as an $A^{-}$-module and as an $A^{+}$-module. Let $K$ denote the trivial $A^{-}$-module, resp. the trivial $A^{+}$-module, concentrated in degree zero.  

\begin{prop}
The algebra $A$ has the KL-property if and only if 
\begin{equation}\label{eq:KL3}
\Ext_{\mc{G}(A^+)}^m(K,L(\mu)) \neq 0 \quad \Rightarrow \quad m \equiv \deg \mu \ \mod \ 2, 
\end{equation}
and 
\begin{equation}\label{eq:KL4}
\Ext_{\mc{G}(A^-)}^m(L(\mu),K) \neq 0 \quad \Rightarrow \quad m \equiv \deg \mu \ \mod  \ 2,
\end{equation}
for all $\mu \in \Lambda$. 
\end{prop}

\begin{proof}
Firstly, it is clear that one can just take $\deg \lambda = 0$ in (\ref{eq:KL1a}) and (\ref{eq:KL2a}), provided $\mu$ ranges over the whole of $\Irr \mc{G}(T)$. Moreover, if we think of $T$ as being the regular representation, concentrated in degree zero then (\ref{eq:KL1a}) and (\ref{eq:KL2a}) hold if and only if
$$
\Ext_{\mc{G}(A)}^m(\Delta(T),L(\mu)) \neq 0 \quad \Rightarrow \quad m \equiv \deg \mu \ \mod \ 2, 
$$
and 
$$
\Ext_{\mc{G}(A)}^m(L(\mu),\nabla(T)) \neq 0 \quad \Rightarrow \quad m \equiv \deg \mu \ \mod  \ 2.  
$$
Then (\ref{eq:KL3}) follows from (\ref{eq:KL1a}) because $T = B^+ \otimes_{A^-} K$ as graded left $B^+$-modules. Similarly, (\ref{eq:KL4}) follows from (\ref{eq:KL2a}) because
$$
T \simeq T^* \simeq (K \otimes_{A^-} B^-)^*
$$
as graded left $B^-$-modules. 
\end{proof} 

In order to have concrete examples of algebras satisfying the KL-property, we consider the case where both $A^-$ and $A^+$ are commutative local complete intersections. That is, we assume that there exists a positively graded vector space $U$ and a homogeneous subspace $V \subset K[U]$, with $V \cap U^* = \{ 0 \}$ such that $A^- = K[U] / \langle V \rangle$ is a complete intersection. In particular, $\dim V = \dim U$. Dually, $A^+ = K[U^*] / \langle V^* \rangle$. Let $x_1, \ds, x_n$ be a homogeneous basis of $U^* \subset K[U]$ and $f_1, \ds, f_n$ a homogeneous  basis of $V$. 

\begin{prop}\label{prop:degreesKL}
Let $A^-$ and $A^+$ be as above and assume that every $L(\lambda) \simeq \lambda$ is irreducible as a $T$-module, i.e., all simple $A$-modules are rigid. Then $A$ has the KL-property if and only if every $\deg x_i$ is odd and every $\deg f_j$ is even. 
\end{prop}

\begin{proof}
We begin by noting that our assumption on $L(\lambda)$ implies that $L(\lambda)$ restricts to $\dim \lambda$ copies of the trivial representation (suitably shifted) for $A^-$ and for $A^+$. Therefore conditions (\ref{eq:KL3}) and (\ref{eq:KL4}) reduce to 
$$
\Ext_{\mc{G}(A^+)}^n(K,K[i]) \neq 0 \ \textrm{or} \ \Ext_{\mc{G}(A^-)}^n(K,K[i]) \neq 0 \quad \Rightarrow \quad n \equiv i \ \mod \ 2. 
$$
We consider only $A^-$ since the situation for $A^+$ is identical. Tate \cite{Tate} gives an explicit graded free resolution of the trivial $A^-$-module $K$ in the case of complete intersections. His construction implies that 
$$
\Ext^m_{\mc{G}(A^-)}(K,K[i])^* \simeq \left( \bigoplus_{\alpha + 2 \beta = m} \wedge^{\alpha} U^* \otimes \mathrm{Sym}^{\beta} V \right)_i. 
$$
The claim of the proposition follows.   
\end{proof}

We note that the space $\Ext^{\idot}_{\mc{M}(A^+)}(K,L(\mu))$ is a bigraded left $\Ext^{\idot}_{\mc{M}(A^+)}(K,K)$-module. The degree zero subspace (with respect to the internal grading) of this module is $\Ext_{\mc{G}(A^{+})}^{\idot}(K,L(\mu))$, which is a graded left $\Ext_{\mc{G}(A^{+})}^{\idot}(K,K)$-module. If we are still in the situation where $A^{+} = K[U^*] / \langle V^* \rangle$ is a complete intersection, then $\Ext^{\idot}_{\mc{M}(A^{+})}(K,L(\mu))$ is, in particular, a bigraded module over 
$$
\mathrm{Sym}^{\beta} V \subset \Ext^{\idot}_{\mc{M}(A^{+})}(K,K), 
$$
see \cite[Theorem 5]{Sjodin}. Thus, the support of $\Ext^{\idot}_{\mc{M}(A^{+})}(K,L(\mu))$ is a closed subvariety $V(\lambda)$ of $V^*$. It would be interesting to study these closed subvarieties for restricted rational Cherednik algebras, in the way that support theory is used in the study of restricted enveloping algebras. 

\begin{ex}
If $A = K[x,y] / (x^n,y^n)$, then Proposition \ref{prop:degreesKL} implies that $A$ has the KL-property if and only if $n$ is even. 
\end{ex}

There are several other special situations where one can check the KL-property. For instance, recall that a positively graded, connected algebra $R$ (not necessarily commutative) is \emph{Koszul} if the trivial module admits a graded free resolution $P^{\idot}$ with $P^i = R^{m_i} [-i]$. Then it is immediate that:

\begin{lem}
If $T = K$ and both $A^-$ and $A^+$ are Koszul, then $A$ satisfies the KL-property.  
\end{lem}

\section{Examples} \label{example_section} 

In this final section, we explore the implications of our results for various examples. We first address some ``toy'' examples to illustrate the various pathologies that can occur within the general framework. Then we consider the more substantial examples mentioned in the introduction:
\begin{enumerate}
\item Restricted enveloping algebras $\overline{U}(\mf{g}_K)$;
\item Lusztig's small quantum groups $\mathbf{u}_{\zeta}(\mf{g})$, at a root of unity $\zeta$; \item Hyperalgebras $\hypalg_r(\mf{g}) \dopgleich \mathrm{Dist}(G_r)$ on the Frobenius kernel $G_r$; 
\item Finite quantum groups $\mc{D}$ associated to a finite group $G$;
\item Restricted rational Cherednik algebras $\overline{\H}_{\bc}(W)$ at $t = 0$;
\item The center of smooth blocks of RRCAs at $t=0$;
\item RRCAs $\overline{\H}_{1,\bc}(W)$ at $t = 1$ in positive characteristic.
\end{enumerate}

\subsection{Toy examples} \label{sec:toy}

\begin{ex}
If $T$ is any split $K$-algebra, considered as a graded algebra concentrated in degree zero, then $A = T$ admits a triangular decomposition with $A^{-} = A^{+} = K$. 
\end{ex}

\begin{ex}\label{ex:commutativetridecomp}
Let $A^{+}$ be any $\bbN$-graded, connected commutative finite dimensional algebra and $A^{-}$ the same ring but with opposite grading. Then $A = A^{-} \otimes_K A^{+}$ is $\Z$-graded with triangular decomposition. Notice that if $A^{+}$ is chosen to be Gorenstein, then so too is $A$.  
\end{ex}

\begin{ex}
Let $V$ be a $K$-vector space and $W \subset \mathrm{GL}(V)$ a finite group. Let $K[V]_W$ denote the coinvariant algebra, which is $\bbN$-graded connected with $V^*$ in degree one. Then $A = K[V]_W \rtimes W$ admits a triangular decomposition with $A^{-} = K$, $T = K W$ and $A^{+} = K[V]_W$. 
\end{ex}

\subsection{Hyperalgebras}\label{sec:hyperalg}

Let $G$ be a connected, finite dimensional semisimple algebraic group over $\C$ and $G_{\Z}$ the corresponding split Chevalley $\Z$-group, with split maximal torus $T_{\Z}$ as defined in Section II 1.1 of \cite{Jantzen}. We fix a field $K$ of characteristic $p > 0$. Set $G_K = G_{\Z} \otimes_{\Z} K$ and $T_K = T_{\Z} \otimes_{\Z} K$. We follow the conventions of \textit{loc. cit.} throughout this section. Let $\mf{g}_{\Z} = \mathrm{Lie}\ G_{\Z}$ and $\mf{g}_K = \mf{g}_{\Z} \otimes_{\Z} K$. We assume that
\begin{enumerate}
\item $p$ is odd and a good prime for $G_K$. 
\item $\mf{g}_K$ has a non-degenerate $G_K$-invariant bilinear form. 
\item $K$ contains the algebraic closure of $\mathbb{F}_p$. 
\end{enumerate}

For each $r \ge 1$, let $G_r$ denote the $r$-th Frobenius kernel of $G_K$. Then $K[G_r]$ is a finite dimensional Hopf algebra and its dual $\hypalg_r(\mf{g}) := K[G_r]^*$ is the \word{$r$-th hyperalgebra} of $G_K$. In particular, when $r = 1$, $\hypalg_1(\mf{g}) = \overline{U}(\mf{g}_K)$ is the restricted enveloping algebra of $\mf{g}_K$. 

Let $X = \Hom(T_K,K^{\times})$ denote the weight lattice and $R \subset X$ the set of roots of $\mf{g}_K$ with respect to $T_K$. Notice that $X$ is independent of the choice of $K$ since $T_K$ is split. Let $\Delta = \{ \alpha_1, \ds, \alpha_s \}$ denote the set of simple roots in $R$ with respect to some polarization $R^+ \subset R$. Set $s := |\Delta|$ to be the rank of $\mf{g}$. If $\langle - , - \rangle$ is the pairing between $X$ and $Y = X^*$, then let $\varpi_i \in Y$ be the fundamental coweights, with $\langle \alpha_i, \varpi_j \rangle = \delta_{i,j}$. Let $\rho$ be the half-sum of positive roots and $\rho^{\vee} = \sum_{i = 1}^s \varpi_i$. The group $G_K$ acts on $\hypalg_r(\mf{g})$ by conjugation. By restriction, so too does $T_K$. This makes $\hypalg_r(\mf{g})$ into an $X$-graded algebra. Define a $\Z$-grading on $\hypalg_r(\mf{g})$ by 
\begin{equation}\label{eq:Zgradinghyper}
\hypalg_r(\mf{g})_i = \bigoplus_{\langle \lambda , \rho^{\vee}\rangle = i} \hypalg_r(\mf{g})_{\lambda}.
\end{equation}
As defined in Section II 3.1 of \cite{Jantzen}, the algebra $\hypalg_r(\mf{g})$ is generated by 
$$
\{ X_{\alpha,n(\alpha)}, H_{i,m(i)}, U_{\alpha,n(\alpha)} \ | \ \alpha \in R^+, i = 1, \ds, s, \ 0 \le n(\alpha),m(i) < p^r \}.
$$
Let $\hypalg_r^+(\mf{g})$ be the subalgebra generated by all $\{ X_{\alpha,n(\alpha)} \ | \ \alpha \in R^+, \ 0 \le n(\alpha),m(i) < p^r \}$, $\hypalg_r^0(\mf{g})$ the subalgebra generated by $\{ H_{i,m(i)} \ | \ i = 1, \ds, s, \ 0 \le m(i) < p^r \}$ and $\hypalg_r^-(\mf{g})$ the subalgebra generated by $\{ U_{\alpha,n(\alpha)} \ | \ \alpha \in R^+, \ 0 \le n(\alpha),m(i) < p^r \}$. Then \cite[II, Lemma 3.3]{Jantzen} implies that $\hypalg_r(\mf{g})$ admits an ambidextrous triangular decomposition
\begin{equation}\label{eq:hyperamb}
\hypalg_r^-(\mf{g}) \otimes \hypalg_r^0(\mf{g}) \otimes \hypalg_r^+(\mf{g}) \stackrel{\sim}{\longrightarrow}  \hypalg_r(\mf{g}) \stackrel{\sim}{\longleftarrow}  \hypalg_r^+(\mf{g}) \otimes \hypalg_r^0(\mf{g}) \otimes \hypalg_r^-(\mf{g}),
\end{equation}
as $\Z$-graded algebras. The algebra $\hypalg_r^+(\mf{g})$ has $K$-basis $\{ \prod_{\alpha \in R^+} X_{\alpha,n(\alpha)} \ | \ 0 \le n(\alpha) < p^r \}$, $\hypalg_r^0(\mf{g})$ has $K$-basis $\{ \prod_{i = 1}^s H_{i,m(i)} \ | \ 0 \le m(i) < p^r \}$ and $\hypalg_r^-(\mf{g})$ has $K$-basis $\{ \prod_{\alpha \in R^+} U_{-\alpha,n'(\alpha)} \ | \ 0 \le n'(\alpha) < p^r\}$. Using (\ref{eq:hyperamb}), these basis give a $K$-basis of $ \hypalg_r(\mf{g})$. 

Let $\mc{G}_X(\hypalg_r(\mf{g}))$ denote the category of $X$-graded $\hypalg_r(\mf{g})$-modules. The commutative algebra $\hypalg_r^0(\mf{g})$ is split semi-simple by assumption (3) above. Let 
\begin{equation}\label{eq:Xi}
X_m := \{ \lambda \in X \ | \ 0 \le \langle \lambda,\varpi_i \rangle < m, \ \forall \ i = 1, \ds, s \}.
\end{equation}
As explained in \cite[\S II 3.7]{Jantzen}, $\Spec \hypalg_r^0(\mf{g})$ equals $X / p^r X$. The set $X_{p^r}$ is a natural section of the quotient map $X \rightarrow X / p^r X$. This defines canonical bijections
$$
\Irr \mc{M}(\hypalg_r(\mf{g})) \ \stackrel{1:1}{\longleftrightarrow} \ X/ p^r X \ \stackrel{1:1}{\longleftrightarrow}\  X_{p^r}.
$$
For $\lambda \in X$, write $\overline{\lambda}$ for its image in $X / p^r X$. 

\begin{prop}\label{prop:hyperproperties} \hfill

\begin{enum_thm}
\item The hyperalgebra $\hypalg_r(\mf{g})$ is equipped with a triangular anti-involution. 
\item The hyperalgebra $\hypalg_r(\mf{g})$ is BGG. 
\item The hyperalgebra $\hypalg_r(\mf{g})$ is graded symmetric and $\hypalg_r^-(\mf{g})$, $\hypalg_r^+(\mf{g})$ are Frobenius. 
\item If $N$ is the top non-zero degree of $\hypalg_r^+(\mf{g})$, then $N = 2 \langle \rho, \rho^{\vee} \rangle (p^r - 1)$.
\end{enum_thm}
\end{prop}

\begin{proof}
As explained in \cite[II, 1.16 and 9.4]{Jantzen}, there is an anti-graded anti-involution $\tau : \hypalg_r(\mf{g}) \rightarrow \hypalg_r(\mf{g})$ such that $\tau(\hypalg_r^+(\mf{g})) = \hypalg_r^-(\mf{g})$ and $\tau$ is the identity on the canonical generators of $\hypalg_r^0(\mf{g})$. This implies that $\tau$ is a triangular anti-involution. By Theorem \ref{triangular_duality}, this implies that $\hypalg_r(\mf{g})$ is also BGG. 

The fact that the hyperalgebra $\hypalg_r(\mf{g})$ is a graded symmetric algebra was shown by Humphreys \cite{HumpSymmetric}. The fact that $\hypalg_r^-(\mf{g})$ and $\hypalg_r^+(\mf{g})$ are Frobenius follows from \cite[Lemma 3.1]{Drupieski2}.

Recall the basis of $\hypalg_r^+(\mf{g})$ described above. The element $\prod_{\alpha \in R^+} X_{\alpha, n(\alpha)}$ has degree $\sum_{\alpha \in R^+} \mathrm{ht}(\alpha) n(\alpha)$. Therefore, the element of highest degree is  $\prod_{\alpha \in R^+} X_{\alpha, p^r - 1}$, which has degree
$$
(p^r - 1) \sum_{\alpha \in R^+} \mathrm{ht}(\alpha) = 2 \langle \rho, \rho^{\vee} \rangle (p^r - 1) \;.
$$
\end{proof}

Corollary \ref{cor:multigrad} implies that $\mc{G}_X(\hypalg_r(\mf{g}))$ is a highest weight category. This category was considered in \cite{AndersenKaneda}, though not from the point of view of highest weight categories.

We note that it is not true, except when $r = 1$, that the subalgebra $\hypalg_r^+(\mf{g})$ is generated by $\hypalg_r^+(\mf{g})_1$, and similarly for $\hypalg_r^-(\mf{g})$.

As explained in \cite[\S 2.1]{JantzenBlocks} (see also \cite[Chapter II.9]{Jantzen}), the category $\mc{G}_X(\hypalg_r(\mf{g}))$ is very closely related to the category of $G_r T$-modules; the latter is the full subcategory of the former defined by conditions (1) and (2) of \cite[Definition 2.1]{JantzenBlocks}.  Applying Corollary \ref{cor:multigrad} to the category $\mc{G}_X(\hypalg_r(\mf{g}))$, we recover the well-known result \cite[Example 6.4]{CPSOttowa} that the category of $G_r T$-modules, with the dominance ordering, is a highest weight category. Since $\mathrm{Irr} \ \hypalg_r(\mf{g})_0$ is in bijection with $X / p^r X$, the set $\Irr \mc{G}(\hypalg_r(\mf{g}))$ is in bijection with $X / p^r X \times \Z$. In this case, one can use results from the literature to compute the permutation $h$ on $\Irr \mc{G}(\hypalg_r(\mf{g}))$. Let $w_0 \in W$ be the longest element. 

\begin{lem}\label{lem:permuhyper}
	If $\lambda = \left(\overline{\lambda}, i\right) \in X / p^r X \times \Z$, then 
	$$
	\lambda^h = \left(\overline{w_0 \lambda - 2 \rho}, i + 2(p^r - 1) \langle \rho, \rho^{\vee} \rangle- 2 \langle \lambda_0, \rho^{\vee} \rangle \right),  
	$$
	where $\lambda_0$ is the unique lift of $\overline{\lambda}$ in $X_{p^r}$. 
\end{lem}

\begin{proof}
	It follows from the definition of $h$ that $(\lambda[j])^h = (\lambda^h)[j]$. Therefore it suffices to compute $\lambda^h$ for some choice of $i$. Also, by definition $\lambda^h$ is the highest weight of the projective module $P(\lambda)$. If $F_X : \mc{G}_X(\hypalg_r(\mf{g})) \rightarrow \mc{G}(\hypalg_r(\mf{g}))$ is the forgetful functor, then we wish to lift $P(\lambda)$ to an object in the full subcategory of $\mc{G}_X(\hypalg_r(\mf{g}))$ consisting of $G_r T$-modules. This will allow us to apply results about projective $G_r T$-modules. Using the notation of \cite{Jantzen}, for $\lambda \in X$, one has $F_X(\widehat{Q}_r(\lambda)) \simeq P\left(\overline{\lambda},\langle \lambda, \rho^{\vee} \rangle \right)$. By Lemma II 11.6 of \textit{loc. cit.}, the highest weight of $\widehat{Q}_r(\lambda)$ is $w_0 \lambda_0 + 2 (p^r - 1) \rho + p^r \lambda_1$, where $\lambda = \lambda_0 + p^r \lambda_1$ with $\lambda_0 \in X_{p^r}$ and $\lambda_1 \in X$. This implies that $P\left(\overline{\lambda},\langle \lambda, \rho^{\vee} \rangle \right)$ has highest weight 
	$$
	\left(\overline{w_0 \lambda_0 + 2 (p^r - 1) \rho + p^r \lambda_1}, \langle w_0 \lambda_0 + 2 (p^r - 1) \rho + p^r \lambda_1, \rho^{\vee} \rangle \right). 
	$$
	The result follows. 
\end{proof}

When considering rational $G$-modules, there is also a natural definition of tilting modules:  those with both a ``good'' filtration and a Weyl filtration. The indecomposable tilting modules $T_G(\mu)$ for $G$ are naturally labelled by the dominant weights $\mu \in X^+$. Restricting to $G_r T$ and applying $F_X$, we get modules $F_X(T_G(\mu) |_{G_r T})$. In general it is hard to describe these modules. In particular, they are \textit{not} tilting modules in our sense. However, using results in the literature, one can show that every tilting module in $\mc{G}(A)$ admits (up to a shift in grading) a lift to a tilting module for $G$. More precisely, it is a consequence of \cite[II E.9 (1)]{Jantzen} that for $p \ge 2 \mathbf{h} - 2$ (where $\mathbf{h}$ is the Coxeter number):

\begin{prop}\label{prop:tiltingrestilting}
	For each $\mu \in (p^r -1)\rho + X_{p^r}$, $F_X(T_G(\mu) |_{G_r T}) \simeq T(\overline{\mu},\langle \mu, \rho^{\vee} \rangle)$.
\end{prop}

\subsection{Restricted enveloping algebras}

We assume that conditions (1)-(3) from Section \ref{sec:hyperalg} continue to hold. As noted above, the first hyperalgebra $\hypalg_1$ is just the restricted enveloping algebra $\overline{U}(\mf{g}_K)$ of $\mf{g}_K$. Set $\overline{U}(\mf{n}^-) := \hypalg_1^-(\mf{g})$, $\overline{U}(\mf{n}^+) := \hypalg_1^+(\mf{g})$ and $\overline{U}(\mf{h}_K) :=\hypalg_1^0(\mf{g})$. As a special case of Proposition \ref{prop:hyperproperties} above, we note that:

\begin{cor}\hfill
\begin{enum_thm}
\item The restricted enveloping algebra $\overline{U}(\mf{g}_K)$ is equipped with a triangular anti-involution. 
\item $\overline{U}(\mf{g}_K)$ is BGG. 
\item $\overline{U}(\mf{g}_K)$ is graded symmetric and $\overline{U}(\mf{n}^-)$ and $\overline{U}(\mf{n}^+)$ are Frobenius. 
\item If $N$ is the top non-zero degree of $\overline{U}(\mf{n}^+)$, then $N = 2 \langle \rho, \rho^{\vee} \rangle (p - 1)$.
\item The subalgebra $\overline{U}(\mf{n}^+)$ is generated by $\overline{U}(\mf{n}^+)_1$, and $\overline{U}(\mf{n}^-)$ is generated by $\overline{U}(\mf{n}^+)_{-1}$.
\end{enum_thm}
\end{cor}

Not only does one get highest weight categories by considering the category $\mc{G}_X(\overline{U}(\mf{g}_K))$ of $X$-graded $\overline{U}(\mf{g}_K)$-modules, or the corresponding category $\mc{G}(\overline{U}(\mf{g}_K))$ of $\Z$-graded modules, but one can also change the grading. These standardly stratified categories play an important role in \cite{JantzenModular}.

\subsection{Lusztig's small quantum groups}

Let $G$, $G_{\Z}$ etc. be as in Section \ref{sec:hyperalg}, but take now $K = \C$. Let $\ell > 1$ be an odd number, coprime to $3$ if $G$ is of type $\mathsf{G}_2$, and let $\zeta$ be a primitive $\ell$th root in $\C$ (our assumptions ensure that the $l_i$ in \cite[\S 8.1]{QuantumGroupsat1} equal $\ell$ for all $i$). If $q$ is an indeterminate, then we denote by $U_q(\mf{g})$ the Drinfeld-Jimbo quantum group, over $\Q(q)$, associated to the simple Lie algebra $\mf{g}_{\C}$. The algebra $U_q(\mf{g})$ is generated by $\{ E_i, F_i, K_i^{\pm 1} \}_{i = 1}^r$, satisfying the relations \cite[(a1)-(a5)]{QuantumGroupsat1}. Let $\mc{A} = \Z[q,q^{-1}]$. Then $U_q(\mf{g})_{\Z}$ is the $\mc{A}$-subalgebra of $U_q(\mf{g})$ generated by all divided powers 
$$
E_i^{(k)} := \frac{E_i^k}{[k]!}, \quad F_i^{(k)} := \frac{F_i^k}{[k]!}, \quad K_i^{\pm 1},
$$
where $i = 1, \ds, r$ and $k \in \N$. Here $[k]! = \prod_{j = 1}^k \frac{q^{j} - q^{-j}}{q - q^{-1}}$ is the quantum factorial. 

The restricted quantum group $U_{\zeta}(\mf{g})$ is defined to be the algebra $U_q(\mf{g})_{\Z} \otimes_{\mc{A}} \C$, where $\mc{A} \rightarrow \C$ sends $q$ to $\zeta$. Then $E_i^{\ell} = F_i^{\ell} = 0$, $K_i^{2 \ell} = 1$ and $K_i^{\ell}$ is central in $U_{\zeta}(\mf{g})$. Finally, Lusztig's small quantum group $\mathbf{u}_{\zeta}$ is the subalgebra of $U_{\zeta}(\mf{g})$ generated by all $E_i, F_i$ and $K_i$, where $i = 1, \ds r$. It is a Hopf algebra of dimension $2^r \ell^{\dim \mf{g}}$. Again, both $U_{\zeta}(\mf{g})$ and $\mathbf{u}_{\zeta}$ are $X$-graded with 
\begin{equation}\label{eq:gradequantum}
\deg E_i^{(k)} = k \alpha_i, \quad \deg F_i^{(k)} = -k \alpha_i, \quad \deg K_i^{\pm 1} = 0.
\end{equation}
As in (\ref{eq:Zgradinghyper}), this makes $U_{\zeta}(\mf{g})$ and $\mathbf{u}_{\zeta}$ into $\Z$-graded algebras by pairing weights with $\rho$. Let $\mathbf{u}_{\zeta}^-$, be the subalgebra generated by $\{ F_i \}_{i = 1}^r$,  $\mathbf{u}_{\zeta}^+$ the subalgebra generated by $\{ E_i \}_{i = 1}^r$, and $\mathbf{u}_{\zeta}^0$ be the subalgebra generated by $\{ K_i^{\pm 1} \}_{i = 1}^r$. By \cite[Theorem 8.3]{QuantumGroupsat1}, we have:

\begin{lem}
Multiplication defines a triangular decomposition 
$$
\mathbf{u}_{\zeta}^- \otimes_{\C} \mathbf{u}_{\zeta}^0 \otimes_{\C} \mathbf{u}_{\zeta}^+ \stackrel{\sim}{\longrightarrow}  \mathbf{u}_{\zeta},
$$
of $\mathbf{u}_{\zeta}$. The algebra is ambidextrous. 
\end{lem}

The algebra $\mathbf{u}_{\zeta}^0$ is the quotient of $\C[K_1^{\pm 1}, \ds, K_r^{\pm 1}]$ by the ideal generated by all $K_i^{2 \ell}  - 1$. Thus, we can identify $\Spec \mathbf{u}_{\zeta}^0$ with 
$$
X_{2 \ell} := \{ \lambda \in X \ | \ 0 \le \langle \lambda, \varpi_i \rangle < 2 \ell, \ i = 1, \ds, s \},
$$
as in (\ref{eq:Xi}). Hence, we have a natural identification $\Irr \mc{M}(\mathbf{u}_{\zeta}) = X_{2 \ell}$. 

\begin{prop}\label{prop:smallquant}
Let $\mathbf{u}_{\zeta}$ denote Lusztig's  small quantum group.
\begin{enum_thm}
\item $\mathbf{u}_{\zeta}$ is equipped with a triangular anti-involution. 
\item $\mathbf{u}_{\zeta}$ is BGG. 
\item $\mathbf{u}_{\zeta}$ is graded symmetric and $\mathbf{u}_{\zeta}^-$, $\mathbf{u}_{\zeta}^+$ are Frobenius. 
\item The subalgebra $\mathbf{u}_{\zeta}^+$ is generated by $(\mathbf{u}_{\zeta}^+)_1$, and $\mathbf{u}_{\zeta}^-$ is generated by $(\mathbf{u}_{\zeta}^-)_{-1}$.
\item If $N$ is the top non-zero degree of $\mathbf{u}_{\zeta}^+$, then $N = 2 \langle \rho, \rho^{\vee} \rangle (\ell - 1)$.
\end{enum_thm}
\end{prop}

\begin{proof}
The anti-involution $\tau$ on $U_q(\mf{g})_{\Z}$ swapping $E_i^{(n)}$ and $F^{(n)}_i$, and fixing each $K_i$, descends to an anti-graded anti-involution $\tau : \mathbf{u}_{\zeta} \rightarrow \mathbf{u}_{\zeta}$ such that $\tau(\mathbf{u}_{\zeta}^+) = \mathbf{u}_{\zeta}^-$, and $\tau$ is the identity on the canonical generators of $\mathbf{u}_{\zeta}^0$. This implies that $\tau$ is a triangular anti-involution. By Theorem \ref{triangular_duality}, this implies that $\mathbf{u}_{\zeta}$ is also BGG. 

The fact that the small quantum group $\mathbf{u}_{\zeta}$ is a graded symmetric algebra is noted in \cite[Proposition 3.1]{KulshammerRep}. The fact that $\mathbf{u}_{\zeta}^-$ and $\mathbf{u}_{\zeta}^+$ are Frobenius follows from \cite[Lemma 3.1]{Drupieski2}.

The algebra $\mathbf{u}_{\zeta}^+$ has a basis 
$$
\left\{ \prod_{\alpha \in R^+} E_{\alpha}^{(n(\alpha))} \  \bigg| \ 0 \le n(\alpha) < \ell \right\}.
$$
Here each $E_{\alpha}^{(m)}$ is a certain product of Lusztig's braid automorphisms applied to a generator $E_i$. The element $\prod_{\alpha \in R^+} E_{\alpha}^{(n(\alpha))}$ has degree $\sum_{\alpha \in R^+} \mathrm{ht}(\alpha) n(\alpha)$. Therefore, the element of highest degree is  $\prod_{\alpha \in R^+} E_{\alpha}^{(\ell - 1)}$, which has degree
$$
(\ell- 1) \sum_{\alpha \in R^+} \mathrm{ht}(\alpha) = 2 \langle \rho, \rho^{\vee} \rangle (\ell - 1) \;.
$$
\end{proof}

Analogues of Lemma \ref{lem:permuhyper} and Proposition \ref{lem:permuhyper} hold for Lusztig's small quantum group too, see \cite[Section 5]{AndersenTilting} and \cite[Proposition 3.1]{AndersenRigid}.

\subsection{Finite quantum groups}\label{sec:finiteqauntumgroups}

While this paper was in preparation, the preprint \cite{Vay} appeared. In that preprint the author studies finite quantum groups $\mc{D}$ associated to a finite group $G$. First we recall briefly the definition of a finite quantum group, as defined in \textit{loc. cit.}. Let $G$ be a finite group, $K$ an algebraically closed field of characteristic zero, and $V$ be a Yetter-Drinfeld module for $K G$, such that the associated Nichols algebra $\mf{B}(V)$ is finite dimensional. Let $\mf{B}(\overline{V})$ be the Nichols algebra of the Yetter-Drinfeld module $\overline{V}$ determined by the isomorphism $\mf{B}(\overline{V}) \rtimes K G \simeq [ (\mf{B}(V) \rtimes K G)^*]^{\op}$. Then $\mc{D}$ is defined to be the Drinfeld double of the bosonization $\mf{B}(V) \rtimes K G$. Let $\mc{D}(G)$ denote the Drinfeld double of $G$. Then $\mc{D}$ admits an ambidextrous triangular decomposition 
$$
\mf{B}(V) \otimes_K \mc{D}(G) \otimes_K  \mf{B}(\overline{V}) \stackrel{\sim}{\longrightarrow} \mc{D} \stackrel{\sim}{\longleftarrow} \mf{B}(\overline{V}) \otimes_K \mc{D}(G) \otimes_K  \mf{B}(V).
$$
Here $\mc{D}$ is $\Z$-graded by putting $\mc{D}(G)$ in degree zero, the degree of $\mf{B}^i(V)$ is $-i$ and the degree of $\mf{B}^j(\overline{V})$ is $j$. The algebras $\mf{B}(V) \rtimes K G$ and $\mf{B}(\overline{V}) \rtimes K G$ are graded subalgebras of $\mc{D}$.  

It follows \cite[Equation (1)]{Vay} that $\mc{D}$ is BGG. It is also noted in \textit{loc. cit.} that the algebras $\mf{B}(\overline{V}) $ and $\mf{B}(V)$ are Frobenius and $\mc{D}$ is graded symmetric. Therefore the results of this article are applicable to $\mc{D}$. In this way one can recover some of the results of \textit{loc. cit.} from our general framework.

\subsection{Restricted rational Cherednik algebras}\label{sec:RRCA}

Let $(\h,W)$ be a finite complex reflection group. That is, $W$ is a non-trivial finite subgroup of $\GL(\h)$, for some finite-dimensional complex vector space $\h$, such that $W$ is generated by its set $\mathrm{Ref}(W)$ of reflections, i.e., by those elements $s \in W$ such that $\Ker(\id_\fh - s)$ is of codimension one in $\fh$. Let $( \cdot, \cdot ) : \mathfrak{h} \times \mathfrak{h}^* \rightarrow \C$ be the natural pairing defined by $(y,x) = x(y)$. For $s \in \mathrm{Ref}(W)$ we fix $\alpha_s \in \mathfrak{h}^*$ to be a basis of the one-dimensional space $\Im (s - 1)|_{\mathfrak{h}^*}$ and $\alpha_s^{\vee} \in \mathfrak{h}$ to be a basis of the one-dimensional space $\Im (s - 1)|_{\mathfrak{h}}$, normalized so that $\alpha_s(\alpha_s^\vee) = 2$. The group $W$ acts on $\mathrm{Ref}(W)$ by conjugation. Choose a function $\mathbf{c} : \mathrm{Ref}(W) \rightarrow \C$ which is invariant under $W$-conjugation and choose a complex number $t \in \bbC$. 

The rational Cherednik algebra $\H_{t,\mathbf{c}}(W)$, as introduced by Etingof and Ginzburg \cite{EG}, is the quotient of the skew group algebra of the tensor algebra, $T(\mf{h} \oplus \mf{h}^*) \rtimes W$, by the ideal generated by the relations $[x,x'] = [y,y'] = 0$ for all $x,x' \in \mathfrak{h}^*$ and $y,y' \in \mathfrak{h}$, and 
\begin{equation}\label{eq:rel}
[y,x] = t(y,x)  - \!\! \sum_{s \in \mathrm{Ref}(W)} \bc(s) (y,\alpha_s)(\alpha_s^\vee,x) s \;, \quad \forall \ y \in \h, \ x \in \h^* \;. 
\end{equation}
We are mainly concerned with the case $t=0$, and set $\H_\bc(W) \dopgleich \H_{0,\bc}(W)$. 

A fundamental result for rational Cherednik algebras, proved by Etingof and Ginzburg \cite[Theorem 1.3]{EG}, is that the \word{Poincar\'e-Birkhoff-Witt (PBW) property} holds for all $\mbf{c}$. This says that multiplication  
\begin{equation}\label{eq:PBW}
 \C [\h] \otimes_\C \C W \otimes_\C \C [\h^*] \rarr \H_\bc(W) 
\end{equation}
is an isomorphism of $\bbC$-vector spaces. The rational Cherednik algebra is naturally $\Z$-graded by $\deg x = -1$ for $x \in \h^*$, $\deg y = +1$ for $y \in \h$, and $\deg w = 0$ for $w \in W$. If $R^{-} := \C[\h]^W$, $R^{+} := \C[\h^*]^W$ and $R = R^{-} \otimes_{\C} R^{+}$, then by \cite[Proposition 3.6]{Baby}, the ring  $R$ is a central subalgebra of $\H_{\mbf{c}}(W)$. 

The \textit{restricted rational Cherednik algebra} $\overline{\H}_{\mbf{c}}(W)$ is defined to be the quotient
\begin{displaymath}
\overline{\H}_{\mbf{c}}(W) = \frac{\H_{\mathbf{c}}(W)}{\left\langle R_+ \right\rangle} \;,
\end{displaymath}
where $R_+$ denotes the augmentation ideal of elements with zero constant term. This algebra was originally introduced, and extensively studied, by Gordon \cite{Baby}. The coinvariant algebras 
$$
\C[\h]^{\mrm{co} W} := \frac{\C[\h]}{\langle R^{-}_+ \rangle} \quad \textrm{and} \quad \C[\h^*]^{\mrm{co} W} := \frac{\C[\h^*]}{\langle R^{+}_+ \rangle}
$$
are graded subalgebras of $\overline{\H}_{\mbf{c}}(W)$. The PBW property implies that the algebra $\overline{\H}_{\mbf{c}}(W)$ admits an ambidextrous triangular decomposition  
\begin{equation} \label{rrca_triangular}
\C [\h]^{\mrm{co} W} \otimes_\C \C W \otimes_\C \C [\h^*]^{\mrm{co}W} \stackrel{\sim}{\longrightarrow}  \overline{\H}_{\mbf{c}}(W) \stackrel{\sim}{\longleftarrow} \C [\h^*]^{\mrm{co} W} \otimes_\C \C W \otimes_\C \C [\h]^{\mrm{co}W} \;.
\end{equation}


We record the relevant properties of the algebra here. As noted in the proof, they are all consequences of known results in the literature. 

\begin{prop}\label{prop:propertiesRRCA} 
	Let $\overline{\H}_{\mbf{c}}(W)$ be a restricted rational Cherednik algebra. 
	\begin{enum_thm} 
		\item \label{prop:propertiesRRCA:symm} $\overline{\H}_{\mbf{c}}(W)$ is graded symmetric and $\C [\h]^{\mrm{co} W}$ and $\C [\h^*]^{\mrm{co}W}$ are Frobenius. 
		\item \label{prop:propertiesRRCA:top} If $N$ is the top non-zero degree of $\C [\h]^{\mrm{co} W}$, then $N = | \mathrm{Ref} (W)|$. 
		\item \label{prop:propertiesRRCA:wellgen} The algebra $\overline{\H}_{\mbf{c}}(W)$ is well-generated i.e. the subalgebra $\C [\h]^{\mrm{co} W}$ is generated by $\C [\h]^{\mrm{co} W}_1$, and $\C [\h^*]^{\mrm{co} W}$ is generated by $\C [\h^*]^{\mrm{co} W}_{-1}$. 
		\item \label{prop:propertiesRRCA:inv} If $W$ is a Coxeter group, then $\overline{\H}_{\mbf{c}}(W)$ admits a triangular anti-involution. 
	\end{enum_thm} 
\end{prop} 

\begin{proof} 
	Part \ref{prop:propertiesRRCA:symm}: It was shown in \cite{BGS} that $\overline{\H}_{\mbf{c}}(W)$ is graded symmetric, and it is well-known that the coinvariant algebras are Frobenius. Part \ref{prop:propertiesRRCA:top} follows from \cite[17-4 Theorem A, 23-1 Theorem A]{Kane}. Part \ref{prop:propertiesRRCA:wellgen} is clear, and Part \ref{prop:propertiesRRCA:inv} is explained in \cite[\S 4.7]{Baby}. 
\end{proof} 

Crucially, results of Section \ref{BGG_section} allow us to show that restricted rational Cherednik algebra is BGG; this was not known previously to hold in complete generality. 

\begin{lem} \label{lem:propertiesRRCABGG} 
	The restricted rational Cherednik algebra $\overline{\H}_{\mbf{c}}(W)$ is BGG. 
\end{lem} 

\begin{proof} 
	The result is a consequence of Proposition \ref{prop:BGGbimodule} if we can show that there is an isomorphism of graded $W$-bimodules 
	$$ 
	\left( \C[\h]^{\mrm{co} W} \rtimes W \right)^* \simeq \C[\h^*]^{\mrm{co} W} \rtimes W. 
	$$ 
	It suffices to show that there is an isomorphism of left $W$-module $\left( \C[\h]^{\mrm{co} W} \right)^*_i \simeq \C[\h^*]^{\mrm{co} W}_i$ for all $i \in \Z$. If we identify $\C[\h^*] = \mathrm{Sym} \ \h$ with constant coefficient differential operators on $\h$, then this is a consequence of the fact that we have a graded $W$-equivariant perfect pairing 
	$$ 
	( - , - ) : \C[\h] \times \mathrm{Sym} \ \h \rightarrow \C, \quad (f,p) := p(f) (0). 
	$$ 
\end{proof} 

\subsection{Triangular decomposition of the centre of a smooth block} \label{center_smooth_block}

In Example \ref{ex:commutativetridecomp}, we described how one can create families of examples of commutative algebras with triangular decomposition, by starting with a positively graded commutative algebra. In this section, we show that such examples arise ``in nature''; they correspond (up to Morita equivalence) to blocks of the restricted rational Cherednik algebra with only one simple module. See the main result, Theorem \ref{thm:RRCAblock}, for the precise statement. Another reason for focusing on blocks with only one simple module is that when the parameter $\bc$ is generic, a ``typical'' block (i.e. most blocks of $\overline{\H}_{\bc}(W)$) has only one simple module. 

In order to better understand the graded structure of those blocks of $\overline{\H}_{\bc}(W)$ that contain only one simple object, we need to use the fact that $\overline{\H}_{\bc}(W)$ is a quotient of $\H_{\bc}(W)$. Let $H := \H_{\bc}(W)$ and let $Z$ denote the centre of $H$. The inclusion of $R$ into $Z$ defines a finite surjective morphism $\Upsilon : \Spec Z \longrightarrow \Spec R$. Recall that $R^{-}$ and $R^{+}$ are graded subalgebras, isomorphic to polynomial rings, with $\Supp R^{-} \subset -\N$ and $\Supp R^{+} \subset \N$ and  
$$
\overline{H} := \overline{\H}_{\bc}(W) = \frac{H}{\langle R_+ \rangle}.
$$
We choose a point $a \in \Upsilon^{-1}(0)$ that is contained in the smooth locus of $\Spec Z$. We summarize standard facts about this situation, relating the representation theory of $\overline{H}$ to the geometry of $\Spec Z$; see \cite{Baby} for details. Firstly, the blocks of $\overline{H}$ are in bijection with the closed points in $\Upsilon^{-1}(0)$. Secondly, since $a \in \Upsilon^{-1}(0)$ lies in the smooth locus of $\Spec Z$, it is shown in \cite[\S 5.3]{Baby} that the corresponding block $\mc{B}_{a}$ of $\overline{H}$ only has one simple module, $L(\lambda)$ say. We may assume that $\lambda$ is concentrated in degree zero. Let $\mf{m} \subset Z$ be the maximal ideal corresponding to the closed point $a$. The cotangent space $V := \mf{m} / \mf{m}^2$ is a graded vector space with a homogeneous symplectic form of degree zero; see \cite{BellEnd} for details. This implies that it decomposes into a direct sum $V = V^+ \oplus V^-$ of strictly positive weights $V^+$ and strictly negative weights $V^-$. Choose homogeneous  elements $x_1, \ds, x_n, y_1, \ds, y_n$ of $Z$ such that the image of the $x_i$ in $V^+$, resp. of the $y_j$ in $V^-$, are a basis. Then Nakyama's Lemma implies that $\{ x_1, \ds, x_n, y_1, \ds, y_n \}$ generate $\mf{m}$ in the local ring $Z_{(a)}$. We let $D = K[x_1, \ds, x_n, y_1, \ds, y_n]$ denote the subring of $Z$ generated by the $x_i$ and $y_j$. Let $Z_a$ denote the image of $Z$ in the block $\mc{B}_a$.

\begin{lem}\label{lem:lemtrcentre}\hfill

\begin{enum_thm}
\item \label{lem:lemtrcentre:pol} The ring $D$ is a polynomial ring.
\item \label{lem:lemtrcentre:equ}The centre of $\mc{B}_a$ equals $Z_a$. 
\item \label{lem:lemtrcentre:surj} The canonical map $D \rightarrow Z_a$ is surjective. 
\end{enum_thm}
\end{lem}

\begin{proof}
For part \ref{lem:lemtrcentre:pol} it suffices to show that the image of $D$ inside the localization $Z_{(a)}$ is a polynomial ring. This follows from the fact that $Z_{(a)}$ is a regular local ring and $\{ x_1, \ds, x_n, y_1, \ds, y_n \}$ form a basis of $\mf{m}  / \mf{m}^2$. Part \ref{lem:lemtrcentre:equ} is shown in \cite{Baby}, where it is noted that the block $\mc{B}_a$, having only one simple module, is isomorphic to matrices of size $|W|$ over $Z_a$. This means that there exists an idempotent $b \in \mc{B}_a$ such that $b \mc{B}_a b \simeq Z_a$ and $\mc{B}_a b$ is the projective cover of $L(\lambda)$. In this instance, one can take $b = e$, the trivial idempotent in $\C W$. Notice that $b$ is homogeneous of degree zero. Also, the Morita equivalence implies that $\dim_K b L(\lambda) = 1$. Part \ref{lem:lemtrcentre:surj} is a consequence of the fact that $Z_a$ is a quotient of $Z_{(a)}$ so it is generated by the image of $\mf{m}$.  
\end{proof}

Next, if we think of $\Delta(\lambda)$ as a ``baby Verma module'' for $H$, then we also have the usual Verma module, and its opposite:
$$
\boldsymbol{\Delta}(\lambda) = H \otimes_{\C[\h^*] \rtimes W} \lambda, \quad \boldsymbol{\Delta}^{\dagger}(\lambda) = H \otimes_{\C[\h] \rtimes W} \lambda.
$$
The action of the centre $Z$ on these modules defines morphisms $D \rightarrow \End_{H}\left(\boldsymbol{\Delta}(\lambda)\right)$ and $D \rightarrow \End_{H}\left(\boldsymbol{\Delta}^{\dagger}(\lambda)\right)$. By \cite[Corollary 4.4]{BellEnd}, both morphisms are surjective, and their images can be identified with $\C[x_1, \ds, x_n]$ and $\C[y_1, \ds, y_n]$ respectively. Let $Z_a^{-}$, resp.  $Z_a^{+}$, be the subalgebra of $Z_a$ generated by $x_1, \ds, x_n$, resp. by $y_1, \ds, y_n$. 

\begin{lem}\label{lem:EndGorenstein}
If we set $E^{-} := \End_{\overline{H}}(\Delta(\lambda))$ and $E^{+} := \End_{\overline{H}}(\Delta^{\dagger}(\lambda))$, then the action of $Z$ on $\Delta(\lambda)$ and $\Delta^{\dagger}(\lambda)$ induce isomorphisms
$$
Z_a^{-} \stackrel{\sim}{\longrightarrow} E^{-}, \quad Z_a^{+} \stackrel{\sim}{\longrightarrow} E^{+}.
$$
Moreover, the algebras $Z_a^{-}$ and $Z_a^{+}$ are Frobenius.  
\end{lem}

\begin{proof}
We consider only $E^{-}$, since the argument for $E^{+}$ is identical. We have an exact sequence $P(\lambda) \rightarrow \Delta(\lambda) \rightarrow 0$. Applying $\Hom_{\overline{H}}(P(\lambda), - )$ and using the identification $P(\lambda) = \mc{B}_a e$, we get an exact sequence
$$
e \mc{B}_a e \rightarrow e \Delta(\lambda) \rightarrow 0. 
$$
This implies that $e \Delta(\lambda)$ is a cyclic $Z_a$-module. Hence, $E^{-}$ is a quotient of $Z_a$. Since $Z_a$ surjects onto $\End_{\overline{H}}(\Delta(\lambda)) \simeq \End_{Z_a}(b \Delta(\lambda))$, and $b \Delta(\lambda)$ is a positively graded, cyclic $Z_a$-module, it follows that $Z_a^{-}$ surjects onto $\End_{Z_a}(b \Delta(\lambda))$. Therefore the difficulty is in showing that $b \Delta(\lambda)$ is a faithful $Z_a^{-}$-module. Assume that $f \in Z_a^{-}$ is a non-zero homogeneous element that acts as zero on $\Delta(\lambda)$. Then $\deg(f) > 0$ since $\deg(x_i) > 0$ for all $i$. We can choose a homogeneous lift $f' \in \C[x_1, \ds, x_n ]$ of $f$. Then $f'$ acts faithfully on $e \boldsymbol{\Delta}(\lambda)$, but by zero on 
$$
e \Delta(\lambda) = \frac{e \boldsymbol{\Delta}(\lambda)}{R^{-}_+ \cdot e \boldsymbol{\Delta}(\lambda)}.
$$
Since $e \boldsymbol{\Delta}(\lambda)$ is a free rank one module over $\C[x_1, \ds, x_n]$, this implies that $f' \in R^{-}_+$. Hence $f = 0$, contradicting our initial assumption.  

It also follows that $E^{-}$ equals  $\End_{H}(\boldsymbol{\Delta}(\lambda)) / \langle R^{-}_+ \rangle $. Since $E^{-}$ is finite dimensional and both of the algebras 
$$
\End_{H}(\boldsymbol{\Delta}(\lambda)) \simeq \C[x_1, \ds, x_n], \quad \textrm{and} \quad R^{-}
$$
are both polynomial rings of the same dimension, the algebra $E^{-}$ is a graded complete intersection. In particular, it is Frobenius. 
\end{proof} 

The goal of this section is to show that the block $\mc{B}_a$, up to graded Morita equivalence is described entirely in terms of $Z_a^{-}$ and $Z_a^{+}$.  Since the algebras $Z_a^{-}$ and $Z_a^{+}$ are Frobenius, the algebra $Z_a^{-} \otimes_{\C} Z_a^{+}$ is also naturally a Frobenius algebra.   

\begin{thm}\label{thm:RRCAblock}
Multiplication 
$$
\mu :  Z^{-}_a \otimes_{\C} Z^{+}_a \longrightarrow  Z_a
$$ 
is an isomorphism of Frobenius algebras. In particular,  $Z^{-}_a \otimes_{\C} Z^{+}_a$ is a graded symmetric algebra. 
\end{thm} 

\begin{proof}
Since $Z_a$ is generated by $x_1, \ds, x_n, y_1, \ds, y_n$, $\mu$ is surjective. On the other hand, it is well-known, e.g. \cite[Corollary 5.8]{Baby}, that 
$$
\dim Z_a = (\dim \lambda)^2 = \dim E^{-} \cdot \dim E^{+}.
$$
This, combined with the isomorphisms of Lemma \ref{lem:EndGorenstein}, implies that $\mu$ is an isomorphism. Since the symmetric structure on a commutative, graded local algebra is uniquely defined up to a scalar (the socle is one-dimensional), the isomorphism can be made to identify the symmetric form on the algebras. 
\end{proof}  




\subsection{The Kazhdan-Lusztig property}

In this section we consider the question when restricted rational Cherednik algebras satisfy the KL-property. Beginning with the case $\bc = 0$, Proposition \ref{prop:degreesKL} immediately implies:

\begin{prop}
	The restricted rational Cherednik algebra $\overline{\H}_{0}(W)$ at $\bc = 0$ satisfies the KL-property if and only if the degrees of a set of homogeneous algebraically independent generators of $\C[\h]^W$ are all even. 
\end{prop}

In particular, we see that $\overline{\H}_{0}(W)$ has the KL-property when $W$ is the wreath product $\Z_{2 \ell} \wr \s_n$, provided $\ell > 0$. For $W = \s_n$ with $n \ge 3$, $\overline{\H}_{0}(W)$ does not have the KL-property. 

At the other extreme, we consider the set-up of Section \ref{center_smooth_block}. Thus, $a \in \Upsilon^{-1}(0)$ is a closed point, contained in the smooth locus of $\Spec Z$, and $\mc{B}_a$ is the corresponding block of $\overline{\H}_{\bc}(W)$ (which has only one simple module). By Theorem \ref{thm:RRCAblock}, $\mc{B}_a$ is graded equivalent to a commutative ring $ Z^-_a \otimes_{\C} Z^{+}_a$. Here $Z^{-}_a$ is a local complete intersection ring. Moreover, one can (in most cases) explicitly compute the degrees of the generators $x_i$ and relations $f_j$, as in Proposition \ref{prop:degreesKL}. If $\lambda \in \Irr W$ labels the unique simple in the block $\mc{B}_a$, then let $f_{\lambda}(q)$ be the fake polynomial associated to $\lambda$ and $b_{\lambda}$ its trailing degree. Recall that this means that 
\begin{align*}
f_{\lambda}(q) & = \sum_{i \ge 0} q^i \dim \Hom_{W}(\lambda,\C[\mf{h}]^{\mathrm{co} W}_i) \\
 & = \alpha q^{b_{\lambda}} + \textrm{ higher order terms.}
\end{align*}
Here $q$ is just a formal variable, and $\alpha \in \Z$ is non-zero. If $d_1, \ds, d_n$ are the degrees of $(W,\mf{h})$ then \cite[Theorem 4.1]{BellEnd} implies that there exist positive integers $e_1, \ds, e_n$ such that 
$$
\prod_{i = 1}^n \frac{1 - q^{d_i}}{1 - q^{e_i}} = q^{-b_{\lambda}} f_{\lambda}(q). 
$$ 

\begin{prop}\label{prop:RRCAsingleblock}
	Assume that $\mc{B}_a$ is a block of $\overline{\H}_{c}(W)$ with unique simple $L(\lambda)$. 
	\begin{enum_thm}
		\item If $\{ d_1, \ds, d_n \} \subset 2 \Z$ and $\{ e_1, \ds, e_n \} \subset 1 + 2 \Z$ then $\mc{B}_a$ has the KL-property. 
		\item The block $\mc{B}_a$ does not satisfy the KL-property if
		$$
		\{ d_1, \ds, d_n \} \smallsetminus \{ e_1, \ds, e_n \} \not\subset 2 \Z \quad \textrm{or} \quad \{ e_1, \ds, e_n \} \smallsetminus \{ d_1, \ds, d_n \} \not\subset 1 + 2 \Z.
		$$
	\end{enum_thm}
\end{prop}

\begin{proof}
	By \cite[Theorem 4.1]{BellEnd} and Theorem \ref{thm:RRCAblock}, $A^{+} = \C[U^*] / \langle V^* \rangle$, where the character of $U \subset \C[U^*]$ is $\sum_{i = 1}^n q^{-e_i}$ and the character of $V^*$ is $\sum_{i = 1}^n q^{-d_i}$. If $\{ d_1, \ds, d_n \} \subset 2 \Z$ and $\{ e_1, \ds, e_n \} \subset 1 + 2 \Z$ then $U \cap V^* = \{ 0 \}$ and it follows from Proposition \ref{prop:degreesKL} that $\mathcal{B}_a$ satisfies the KL-property. 
	
	On the other hand, if either of the conditions in (2) hold, then there are homogeneous subspaces $U_1 \subset U$ and $V_1^* \subset V^*$ such that $U_1 \cap V_1^* = \{ 0 \}$, $A^{+} \simeq \C[U^*_1] / \langle V^*_1 \rangle$ and either the weights in $U_1$ are not all odd, or the weights of $V_1^*$ are not all even. Either way, it follows from Proposition \ref{prop:degreesKL} that $\mc{B}_a$ does not satisfy the KL-property. 
\end{proof}

\begin{ex}
	The two situations considered in Proposition \ref{prop:RRCAsingleblock} do not exhaust all possibilities, e.g. $\{ d_1, d_2 \} = \{ 1,2 \}$ and $\{ e_1, e_2 \} = \{ 1, 3 \}$. This is because, in situations such as this, one is not able to guarantee just from the combinatorics that the requirement $V \cap U^* = \{ 0 \}$ holds. This condition is essential in the proof of Proposition \ref{prop:degreesKL}. 
\end{ex}

\begin{ex}
	Let $W = \s_n$, the symmetric group, and take $\bc \neq 0$. In this case the variety $\Spec Z$ is smooth. Therefore, as explained in Section \ref{center_smooth_block}, every block of $\overline{\H}_{\bc}(\s_n)$ contains only one simple object. Thus, we may apply Proposition \ref{prop:RRCAsingleblock}. If $n \ge 4$ and we take $\lambda$ to be the two row partition $(n-2,2)$ of $n$, then in this case $\{ e_1, \ds, e_n \} \smallsetminus \{ d_1, \ds, d_n \} = \{ 1, 2 \}$. This implies that $\overline{\H}_{\bc}(\s_n)$ does not satisfy the KL-property when $\bc \neq 0$. On the other hand, $\overline{\H}_{\bc}(\s_n)$ does satisfy the KL-property when $n = 2,3$ and $\bc \neq 0$. More generally, if $W = \Z_{\ell} \wr \s_n$ and $\underline{\lambda} = (\lambda^{(1)}, \dots, \lambda^{(\ell)})$ is an $\ell$-multipartition of $n$ then $\{ d_1, \dots, d_n \} = \{ m,2m,\dots, nm \}$ and 
	$$
	\{ e_1, \dots, e_n\} = \left\{ m h_{\lambda^{(i)}}(x) \ | \ x \in \lambda^{(i)}, \ i = 1, \dots, n \right\}.
	$$ 
	From this one can deduce that $\overline{\H}_{\bc}(\Z_{\ell} \wr \s_n )$ does not satisfy the KL-property when $n \ge 4$ and $\bc$ is generic (again, the variety $\Spec Z$ is smooth in this case and the results of Section \ref{center_smooth_block} apply). 
\end{ex}

\subsection{Restricted rational Cherednik algebras in positive characteristic}

We note briefly that restricted rational Cherednik algebras at $t = 1$ in positive characteristic, as studied in \cite{MoPositiveChar}, are also examples of graded algebras admitting a triangular decomposition. Let $K$ be an algebraically closed field of characteristic $p > 0$ and $(\h,W)$ a pseudo-reflection group over $K$, as defined in \cite[\S 2.1]{MoPositiveChar}. As in \textit{loc. cit.} we assume that $p$ does not divide $|W|$. The associated restricted rational Cherednik algebra $\overline{\H}_{1,\bc}(W)$ is defined in Section 3.3 of \cite{MoPositiveChar}. As in \textit{loc. cit.}, 
$$
K[\h]^{p  \mrm{co} W} := K[\h] / \left\langle K\left[\h^{(1)}\right]_+^W \right\rangle
$$
denotes the $p$-coinvariant ring. Here $\h^{(1)}$ is the Frobenius twist of $\h$. 

\begin{lem}\label{lem:gradedWisop}
As graded left $W$-modules, 
$$
\left( K[\h]^{p \mrm{co} W} \right)^* \simeq K [\h^*]^{p \mrm{co} W}.
$$
\end{lem}

\begin{proof}
This is done by breaking the statement into two. By \cite[Lemma 2.4]{MoPositiveChar}, we have isomorphisms of graded $W$-modules 
$$
K[\h]^{p \mrm{co} W} \simeq \left( K[\h] / \left\langle K\left[\h^{(1)}\right]_+ \right\rangle \right) \otimes_K K\left[\h^{(1)}\right]^{\mrm{co} W}, 
$$
$$
K[\h^*]^{p \mrm{co} W} \simeq \left( K[\h^*] / \left\langle K\left[(\h^*)^{(1)}\right]_+ \right\rangle \right) \otimes_K K\left[(\h^*)^{(1)}\right]^{\mrm{co} W},
$$
and hence
$$
\left( K[\h]^{p \mrm{co} W} \right)^*\simeq \left( K[\h] / \left\langle K\left[\h^{(1)}\right]_+ \right\rangle \right)^* \otimes_K \left( K\left[\h^{(1)}\right]^{\mrm{co} W} \right)^*.
$$
Thus, we must establish isomorphisms of graded $W$-modules, 
\begin{equation}\label{eq:Wgradedtwist}
\left( K[\h] / \langle K[\h^{(1)}]_+ \rangle \right)^* \simeq K[\h^*] / \langle K[(\h^*)^{(1)}]_+ \rangle
\end{equation}
and
\begin{equation}\label{eq:Wgradedtwist2}
\left( K[\h^{(1)}]^{\mrm{co} W} \right)^* \simeq K[(\h^*)^{(1)}]^{\mrm{co} W}. 
\end{equation}
Regarding the latter, we first note that $\h^{(1)} \simeq \h$ as $W$-modules. Therefore, in order to show the isomorphism (\ref{eq:Wgradedtwist2}), it suffices to show that 
\begin{equation}\label{eq:Wgradedtwist3}
\left( K[\h]^{\mrm{co} W} \right)^* \simeq K[\h^*]^{\mrm{co} W} 
\end{equation} 
holds. This follows from the proof of Lemma \ref{lem:propertiesRRCABGG}, by a suitable base change.  

We establish the isomorphism (\ref{eq:Wgradedtwist}). If we identify $K[\h^*] = \mathrm{Sym} \ \h$ with constant coefficient crystalline differential operators on $\h$, then we have a graded $W$-equivariant pairing 
$$
( - , - ) : K [\h] \times \mathrm{Sym} \ \h \rightarrow K, \quad (f,p) := p(f) (0). 
$$
This descends to a perfect pairing 
$$
( - , - ) : \frac{K[\h]}{\langle K[\h^{(1)}]_+ \rangle} \times \frac{\mathrm{Sym} \ \h}{\langle (\mathrm{Sym} \ \h^{(1)} )_+ \rangle}  \rightarrow K,
$$
proving (\ref{eq:Wgradedtwist}). 
\end{proof}

The analogue of Proposition \ref{prop:propertiesRRCA}, and Lemma \ref{lem:propertiesRRCABGG}, in this setting read:

\begin{prop}\label{prop:propertiesRRCAp}
Let $\overline{\H}_{1,\mbf{c}}(W)$ be a restricted rational Cherednik algebra, defined over a field $K$ of characteristic $p >0$, where $p$ does not divide the order of $W$. 
\begin{enum_thm}
\item \label{prop:propertiesRRCAp:BGG} $\overline{\H}_{\mbf{c}}(W)$ is BGG. 
\item \label{prop:propertiesRRCAp:symm}$\overline{\H}_{1,\mbf{c}}(W)$ is graded symmetric and $K [\h]^{p \mrm{co} W}$ and  $K [\h^*]^{p \mrm{co}W}$ are Frobenius. 
\item \label{prop:propertiesRRCAp:top}If $N$ is the top non-zero degree of $K [\h]^{p \mrm{co} W}$, then $N = (p-1) \dim \h + p | \mathrm{Ref} (W) |$.
\item \label{prop:propertiesRRCAp:well}The algebra $\overline{\H}_{1,\mbf{c}}(W)$ is well-generated i.e. the subalgebra $K [\h]^{p \mrm{co} W}$ is generated by $K [\h]^{p \mrm{co} W}_1$, and $K [\h^*]^{p \mrm{co} W}$ is generated by $K [\h^*]^{p \mrm{co} W}_{-1}$.
\item \label{prop:propertiesRRCAp:inv}If $W$ is a Coxeter group, then $\overline{\H}_{1,\mbf{c}}(W)$ admits a triangular anti-involution. 
\end{enum_thm}
\end{prop}

\begin{proof}
The proof of \ref{prop:propertiesRRCAp:BGG} is similar to the proof of Lemma \ref{lem:propertiesRRCABGG}. It is a consequence of Proposition \ref{prop:BGGbimodule} if we can show that there is an isomorphism of graded $W$-bimodules 
$$
\left( K[\h]^{p \mrm{co} W} \rtimes W \right)^* \simeq K[\h^*]^{p \mrm{co} W} \rtimes W.
$$
Again, it suffices to show that there is an isomorphism of graded left $W$-module $\left( K[\h]^{p \mrm{co} W} \right)^* \simeq K [\h^*]^{p \mrm{co} W}$. This is Lemma \ref{lem:gradedWisop}. Part \ref{prop:propertiesRRCAp:symm} is explained in Section 3.3 of \cite{MoPositiveChar}. Part (3) follows from the isomorphism of \cite[Lemma 2.3(2)]{MoPositiveChar} and Proposition \ref{prop:propertiesRRCA} \ref{prop:propertiesRRCAp:top}. Part \ref{prop:propertiesRRCAp:well} is clear. Finally, the triangular anti-involution in part \ref{prop:propertiesRRCAp:inv} is defined in exactly the same way as in characteristic zero, the key fact being that there is a non-degenerate $W$-equivariant bilinear form $( - , - ) : \h \times \h \rightarrow K$ since $\h$ is the reduction of the reflection representation in characteristic zero. 
\end{proof}


%

{\small
\bibliography{biblo}{}
\bibliographystyle{abbrv}
}

\end{document}